\makeatletter

\documentclass[
    11pt,
    a4paper,
    oneside,
    openright,
    center,
    chapterbib,
    crosshair,
    reqno,
    headcount,
    headline,
    indent,
    indentfirst=false,
    portrait,
    phonetic,
    oldernstyle,
    onecolumn,
    sfbold,
    upper,
]{amsart}

\let\th@plain\relax

\PassOptionsToPackage{T2A}{fontenc} % T1,OT1,T2A,OT2
\PassOptionsToPackage{utf8}{inputenc} % utf8
\PassOptionsToPackage{cyrpart}{cyrillic}
\PassOptionsToPackage{british,ngerman,russian}{babel}
\PassOptionsToPackage{
    foot,
}{amsaddr}
\PassOptionsToPackage{
    style=numeric-comp,
    citestyle=authoryear,
}{biblatex}
\PassOptionsToPackage{
    english,
    ngerman,
    russian,
    capitalise,
}{cleveref}
\PassOptionsToPackage{
    margin=10pt,
    position=bottom,
    justification=centering,
    font=small,
    labelformat=simple,
    labelfont={sc},
    labelsep={period},
    textfont={it},
}{caption}
\PassOptionsToPackage{
    margin=10pt,
    position=bottom,
    justification=centering,
    font=small,
    labelformat=parens,
    labelfont={bf},
    labelsep={space},
    textfont={it},
}{subcaption}
\PassOptionsToPackage{framemethod=TikZ}{mdframed}
\PassOptionsToPackage{
    bookmarks=true,
    bookmarksopen=false,
    bookmarksopenlevel=0,
    bookmarkstype=toc,
    raiselinks=true,
    colorlinks=true,
    hyperfigures=true,
    anchorcolor=red,
    citecolor=green,
    linkcolor=blue,
    urlcolor=blue,
}{hyperref} % <- must be done using \hypersetup
\PassOptionsToPackage{normalem}{ulem}
\PassOptionsToPackage{
    amsmath,
    thmmarks,
}{ntheorem}
\PassOptionsToPackage{
}{thmtools}
\PassOptionsToPackage{
    table,
}{xcolor}
\PassOptionsToPackage{
    all,
    color,
    curve,
    frame,
    import,
    knot,
    line,
    movie,
    rotate,
    textures,
    tile,
    tips,
    web,
    xdvi,
}{xy}
\PassOptionsToPackage{
    reset,
    left=1in,
    right=1in,
    top=20mm,
    bottom=20mm,
    heightrounded,
}{geometry}
\PassOptionsToPackage{
    symbol*, % use symbols to number the footnotes.
    multiple, % comma-separates symbols on multiply footnoted words.
}{footmisc}
\PassOptionsToPackage{overload}{textcase}
\PassOptionsToPackage{indentafter}{titlesec}
\PassOptionsToPackage{
    shortcuts, % allows \=/ and \==/ for unbreakable hypens
}{extdash}

\usepackage{amsaddr}
\usepackage{amsfonts}
\usepackage{amsmath}
\usepackage{amssymb}
\usepackage{ntheorem} % must occur after ams*-packages
\usepackage{thmtools} % must occur after ntheorem; use to establish compatibility with cleveref
\usepackage{array}
\usepackage{babel}
\usepackage{braket}
\usepackage{bbding}
\usepackage{bm}
\usepackage{bbm}
\usepackage{bibentry}
\usepackage{booktabs}
\usepackage{bold-extra} % <- needed for bold+scshape
\usepackage{calc}
\usepackage{cancel}
\usepackage{caption} %% bevor [subcaption]
\usepackage{changepage}
\usepackage{cjhebrew}
\usepackage{cmlgc}
\usepackage{colonequals}
\usepackage{color}
\usepackage{comment}
\usepackage{datetime}
\usepackage{dsfont}
\usepackage{etex}
\usepackage{etoolbox}
\usepackage{eurosym}
\usepackage{extdash}
\usepackage{fancybox}
\usepackage{fancyhdr}
\usepackage{float}
\usepackage{fontenc}
\usepackage{footmisc}
\usepackage{fp}
\usepackage{geometry}
\usepackage{graphicx}
\usepackage{ifpdf}
\usepackage{ifthen}
\usepackage{ifoddpage}
\usepackage{ifnextok}  % funktioniert unter BEAMERMODUS nicht
\usepackage{inputenc}
\usepackage{latexsym}
\usepackage{lineno}
\usepackage{listings}
\usepackage{longtable}
\usepackage{lscape}
\usepackage{mathtools} % <- need this for commands like \DeclarePairedDelimiter{\norm}{\lVert}{\rVert}
\usepackage{mathrsfs}
\usepackage{multicol}
\usepackage{multirow}
\usepackage{nameref}
\usepackage{nowtoaux}
\usepackage{paralist}
\usepackage{enumerate} % <- after parlist
\usepackage{enumitem} % <- enables more robust [label=...] options for enmerate
\usepackage{pgf}
\usepackage{pgfplots}
\usepackage{phonetic}
\usepackage{proof}
\usepackage{qtree}
\usepackage{refcount}
\usepackage{savesym}
\usepackage{stmaryrd}
\usepackage{synttree}
\usepackage{subcaption}
\usepackage{suffix}
\usepackage{yfonts} % <— Altgotische Fonts
\usepackage{textcase} % <- causes issues with amsart
\usepackage{tikz}
\usepackage{xy}
\usepackage{ulem} % <– f\"ur besseren \underline-Befehl (\ul)
\usepackage{wrapfig}
\usepackage{xcolor}
\usepackage{xspace}
\usepackage{xstring}
\usepackage{arydshln}
\usepackage{hyperref}
\usepackage{cleveref} % must occur AFTER hyperref

\pgfplotsset{compat=newest}
\usetikzlibrary{math}

\usetikzlibrary{
    angles,
    arrows,
    automata,
    calc,
    decorations,
    decorations.pathmorphing,
    decorations.pathreplacing,
    positioning,
    patterns,
    patterns.meta,
    quotes,
}

\savesymbol{corresponds}
\savesymbol{Diamond}
\savesymbol{emptyset}
\savesymbol{ggg}
\savesymbol{int}
\savesymbol{lll}
\savesymbol{RectangleBold}
\savesymbol{langle}
\savesymbol{rangle}
\savesymbol{hookrightarrow}
\savesymbol{hookleftarrow}
\savesymbol{Asterisk}
\usepackage{mathabx}
\usepackage{wasysym}

\restoresymbol{x}{corresponds}
\restoresymbol{x}{Diamond}
\restoresymbol{x}{emptyset}
\restoresymbol{x}{ggg}
\restoresymbol{x}{int}
\restoresymbol{x}{lll}
\restoresymbol{x}{RectangleBold}
\restoresymbol{x}{langle}
\restoresymbol{x}{rangle}
\restoresymbol{x}{hookrightarrow}
\restoresymbol{x}{hookleftarrow}
\restoresymbol{x}{Asterisk}

\ifpdf
    \usepackage{pdfcolmk}
\fi

\usepackage{mdframed}

\DeclareFontFamily{U}{MnSymbolA}{}
\DeclareFontShape{U}{MnSymbolA}{m}{n}{
    <-6> MnSymbolA5
    <6-7> MnSymbolA6
    <7-8> MnSymbolA7
    <8-9> MnSymbolA8
    <9-10> MnSymbolA9
    <10-12> MnSymbolA10
    <12-> MnSymbolA12
}{}
\DeclareFontShape{U}{MnSymbolA}{b}{n}{
    <-6> MnSymbolA-Bold5
    <6-7> MnSymbolA-Bold6
    <7-8> MnSymbolA-Bold7
    <8-9> MnSymbolA-Bold8
    <9-10> MnSymbolA-Bold9
    <10-12> MnSymbolA-Bold10
    <12-> MnSymbolA-Bold12
}{}
\DeclareSymbolFont{MnSymA}{U}{MnSymbolA}{m}{n}
\DeclareMathSymbol{\lcirclearrowright}{\mathrel}{MnSymA}{252}
\DeclareMathSymbol{\lcirclearrowdown}{\mathrel}{MnSymA}{255}
\DeclareMathSymbol{\rcirclearrowleft}{\mathrel}{MnSymA}{250}
\DeclareMathSymbol{\rcirclearrowdown}{\mathrel}{MnSymA}{251}
\DeclareFontFamily{U}{MnSymbolC}{}
\DeclareSymbolFont{MnSyC}{U}{MnSymbolC}{m}{n}
\DeclareFontShape{U}{MnSymbolC}{m}{n}{
    <-6>  MnSymbolC5
    <6-7>  MnSymbolC6
    <7-8>  MnSymbolC7
    <8-9>  MnSymbolC8
    <9-10> MnSymbolC9
    <10-12> MnSymbolC10
    <12->   MnSymbolC12%
}{}
\DeclareMathSymbol{\powerset}{\mathord}{MnSyC}{180}
\DeclareMathSymbol{\righthalfcap}{\mathbin}{MnSyC}{186}

\DeclareMathAlphabet{\mathpzc}{OT1}{pzc}{m}{it}
\DeclareMathAlphabet{\blackboardfont}{U}{BOONDOX-ds}{m}{n}

\def\boolwahr{true}
\def\boolfalsch{false}
\def\boolleer{}

\let\boolinappendix\boolfalsch
\let\boolinmdframed\boolfalsch

\newcount\bufferctr
\newcount\bufferreplace

\newlength\rtab
\newlength\gesamtlinkerRand
\newlength\gesamtrechterRand
\newlength\ownspaceabovethm
\newlength\ownspacebelowthm
\newlength\aboveequation
\newlength\belowequation
\def\secnumberingpt{.}
\def\secnumberingseppt{.}
\def\subsecnumberingseppt{}
\def\thmnumberingpt{.}
\def\thmnumberingseppt{}
\def\thmForceSepPt{.}

\definecolor{leer}{gray}{1}
\definecolor{boxgrau}{gray}{0.85}
\definecolor{dunkelgrau}{gray}{0.5}
\definecolor{maroon}{rgb}{0.6901961,0.1882353,0.3764706}
\definecolor{dunkelgruen}{rgb}{0.015625,0.363281,0.109375}
\definecolor{dunkelrot}{rgb}{0.5450980392,0,0}
\definecolor{dunkelblau}{rgb}{0,0,0.5450980392}
\definecolor{blau}{rgb}{0,0,1}
\definecolor{newresult}{rgb}{0.6,0.6,0.6}
\definecolor{improvedresult}{rgb}{0.9,0.9,0.9}
\definecolor{hervorheben}{rgb}{0,0.9,0.7}
\definecolor{starkesblau}{rgb}{0.1019607843,0.3176470588,0.8156862745}
\definecolor{achtung}{rgb}{1,0.5,0.5}
\definecolor{frage}{rgb}{0.5,1,0.5}
\definecolor{schreibweise}{rgb}{0,0.7,0.9}
\definecolor{axiom}{rgb}{0,0.3,0.3}
\definecolor{drawing_light_grey}{gray}{0.85}
\definecolor{background_light_grey}{gray}{0.95}

\def\let@name#1#2{
    \expandafter\let\csname #1\expandafter\endcsname\csname #2\endcsname\relax
}
\DeclareRobustCommand\crfamily{\fontfamily{ccr}\selectfont}
\DeclareTextFontCommand{\textcr}{\crfamily}

\def\ifthenelseleer#1#2#3{\ifthenelse{\equal{#1}{}}{#2}{#1#3}}
\def\bedingtesspaceexpand#1#2#3{\ifthenelseleer{\csname #1\endcsname}{#3}{#2#3}}

\def\hraum{\null\hfill\null}

\def\nvraum{\@ifnextchar\bgroup{\nvraum@c}{\nvraum@bes}}
    \def\nvraum@c#1{\vspace*{-#1\baselineskip}}
    \def\nvraum@bes{\vspace*{-\baselineskip}}
\def\forceaddspace{\relax\ifmmode\else\@\xspace\fi}
\def\forceremovespace{\relax\ifmmode\else\expandafter\@gobble\fi}

\def\send@toaux#1{\@bsphack\protected@write\@auxout{}{\string#1}\@esphack}

\def\rlabel#1[#2]#3#4#5{#5\rlabel@aux{#1}[#2]{#3}{#4}{#5}}
    \def\rlabel@aux#1[#2]#3#4#5{%
        \send@toaux{\newlabel{#1}{{\@currentlabel}{\thepage}{{\unexpanded{#5}}}{#2.\csname the#2\endcsname}{}}}\relax%
    }

\def\tag@rawscheme#1#2[#3]#4#5{\@ifnextchar[{\tag@rawscheme@{#1}{#2}[#3]{#4}{#5}}{\tag@rawscheme@{#1}{#2}[#3]{#4}{#5}[*]}}
    \def\tag@rawscheme@#1#2[#3]#4#5[#6]{\@ifnextchar\bgroup{\tag@rawscheme@@{#1}{#2}[#3]{#4}{#5}[#6]}{\tag@rawscheme@@{#1}{#2}[#3]{#4}{#5}[#6]{}}}
    \def\tag@rawscheme@@#1#2[#3]#4#5[#6]#7{%
        \ifthenelse{\equal{#6}{*}}{%
            \ifthenelse{\equal{#7}{\boolleer}}{\refstepcounter{#3}#4\csname the#3\endcsname#5}{#4#7#5}%
        }{%
            \refstepcounter{#3}#4%
            \ifthenelse{\equal{#7}{\boolleer}}{\rlabel{#6}[#3]{#1}{#2}{\csname the#3\endcsname}}{\rlabel{#6}[#3]{#1}{#2}{#7}}%
            #5%
        }%
    }
\def\tag@scheme#1#2[#3]{\tag@rawscheme{#1}{#2}[#3]{\upshape(}{\upshape)}}

\def\eqtag@post#1{\makebox[0pt][r]{#1}}
\def\eqtag@pre{\tag@scheme{Eq}{Equation}[equation]}
\def\eqtag{\@ifnextchar[{\eqtag@}{\eqtag@[*]}}
    \def\eqtag@[#1]{\@ifnextchar\bgroup{\eqtag@@[#1]}{\eqtag@@[#1]{}}}
    \def\eqtag@@[#1]#2{\eqtag@post{\eqtag@pre[#1]{#2}}}

\def\eqcref#1{\text{(\ref{#1})}}
\def\punktlabel#1{\label{it:#1:\beweislabel}}
\def\punktcref#1{\eqcref{it:#1:\beweislabel}}

\def\Crefit#1#2{\Cref{#1}~\eqcref{it:#2:#1}}

\def\opfromto[#1]_#2^#3{\underset{#2}{\overset{#3}{#1}}}
\def\textoverset#1#2{\overset{\text{#1}}{#2}}

\def\eqcrefoverset#1#2{\textoverset{\eqcref{#1}}{#2}}

\def\mathclap#1{#1}
\def\oberunterset#1{\@ifnextchar^{\oberunterset@oben{#1}}{\oberunterset@unten{#1}}}
    \def\oberunterset@oben#1^#2_#3{\underset{\mathclap{#3}}{\overset{\mathclap{#2}}{#1}}}
    \def\oberunterset@unten#1_#2^#3{\underset{\mathclap{#2}}{\overset{\mathclap{#3}}{#1}}}
    \def\breitunderbrace#1_#2{\underbrace{#1}_{\mathclap{#2}}}
    \def\breitoverbrace#1^#2{\overbrace{#1}^{\mathclap{#2}}}
    \def\breitunderbracket#1_#2{\underbracket{#1}_{\mathclap{#2}}}
    \def\breitoverbracket#1^#2{\overbracket{#1}^{\mathclap{#2}}}

\def\generatenestedsecnumbering#1#2#3{%
    \expandafter\gdef\csname thelong#3\endcsname{%
        \expandafter\csname the#2\endcsname%
        \secnumberingpt%
        \expandafter\csname #1\endcsname{#3}%
    }%
    \expandafter\gdef\csname theshort#3\endcsname{%
        \expandafter\csname #1\endcsname{#3}%
    }%
}
\def\generatenestedthmnumbering#1#2#3{%
    \expandafter\gdef\csname the#3\endcsname{%
        \expandafter\csname the#2\endcsname%
        \thmnumberingpt%
        \expandafter\csname #1\endcsname{#3}%
    }%
    \expandafter\gdef\csname theshort#3\endcsname{%
        \expandafter\csname #1\endcsname{#3}%
    }%
}

\providecommand{\setcounternach}{}
\renewcommand{\setcounternach}[2]{\setcounter{#1}{#2}\addtocounter{#1}{-1}}
\providecommand{\textsubscript}{}
\renewcommand{\textsubscript}[1]{${}_{\textup{#1}}$}

\def\forcepunkt#1{#1\IfEndWith{#1}{.}{}{.}}

\def\matrix#1{\left(\begin{array}{#1}}
    \def\endmatrix{\end{array}\right)}
\def\smatrix{\left(\begin{smallmatrix}}
    \def\endsmatrix{\end{smallmatrix}\right)}

\def\multiargrekursiverbefehl#1#2#3#4#5#6#7#8{%
    \expandafter\gdef\csname#1\endcsname #2##1#4{\csname #1@anfang\endcsname##1#3\egroup}
    \expandafter\def\csname #1@anfang\endcsname##1#3{#5##1\@ifnextchar\egroup{\csname #1@ende\endcsname}{#7\csname #1@mitte\endcsname}}
    \expandafter\def\csname #1@mitte\endcsname##1#3{#6##1\@ifnextchar\egroup{\csname #1@ende\endcsname}{#7\csname #1@mitte\endcsname}}
    \expandafter\def\csname #1@ende\endcsname##1{#8}
}
\multiargrekursiverbefehl{svektor}{[}{;}{]}{\begin{smatrix}}{}{\\}{\\\end{smatrix}}
\multiargrekursiverbefehl{vektor}{[}{;}{]}{\begin{matrix}{c}}{}{\\}{\\\end{matrix}}
\multiargrekursiverbefehl{vektorzeile}{}{,}{;}{}{&}{}{}
\multiargrekursiverbefehl{matlabmatrix}{[}{;}{]}{\begin{smatrix}\vektorzeile}{\vektorzeile}{;\\}{;\end{smatrix}}

\def\BeweisRichtung[#1]{\@ifnextchar\bgroup{\@BeweisRichtung@c[#1]}{\@BeweisRichtung@bes[#1]}}
    \def\@BeweisRichtung@bes[#1]{{\bfseries (#1)}}
    \def\@BeweisRichtung@c[#1]#2#3{#2~#1~#3}
\def\erzeugeBeweisRichtungBefehle#1#2{
    \expandafter\gdef\csname #1text\endcsname##1##2{\BeweisRichtung[#2]{##1}{##2}}
    \expandafter\gdef\csname #1\endcsname{%
        \@ifnextchar\bgroup{\csname #1@\endcsname}{\csname #1text\endcsname{}{}}%
    }
    \expandafter\gdef\csname #1@\endcsname##1##2{%
        \csname #1text\endcsname{\punktcref{##1}}{\punktcref{##2}}%
    }
}
\erzeugeBeweisRichtungBefehle{hinRichtung}{$\Rightarrow$}
\erzeugeBeweisRichtungBefehle{herRichtung}{$\Leftarrow$}
\erzeugeBeweisRichtungBefehle{hinherRichtung}{$\Leftrightarrow$}

\def\mathfrak#1{\mbox{\usefont{U}{euf}{m}{n}#1}}

\def\rectangleblack{\text{\RectangleBold}}

\def\squareblack{\blacksquare}

\def\create@abbreviation#1#2{
    \expandafter\gdef\csname #1\endcsname{%
        #2\@ifnextchar.{%
            \relax\ifmmode\else\expandafter\@gobble\fi%
        }{%
            \relax\ifmmode\else\@\xspace\fi%
        }%
    }
}

\create@abbreviation{cf}{\emph{cf.}}
\create@abbreviation{Cf}{\emph{Cf.}}
\create@abbreviation{thatis}{\emph{i.e.}}
\create@abbreviation{idest}{\emph{i.e.}}
\create@abbreviation{Idest}{\emph{I.e.}}
\create@abbreviation{exempli}{\emph{e.g.}}
\create@abbreviation{Exempli}{\emph{E.g.}}
\create@abbreviation{etcetera}{\emph{etc.}}
\create@abbreviation{etAlia}{\emph{et al.}}
\create@abbreviation{andothers}{\emph{et al.}}
\create@abbreviation{initself}{\emph{per se}}
\create@abbreviation{perse}{\emph{per se}}
\create@abbreviation{namely}{\emph{viz.}}
\create@abbreviation{viz}{\emph{viz.}}

\create@abbreviation{Tfae}{T.\,f.\,a.\,e.}
\create@abbreviation{tfae}{t.\,f.\,a.\,e.}
\create@abbreviation{wrt}{wrt.}
\create@abbreviation{randomvar}{r.\,v.}
\create@abbreviation{iid}{i.\,i.\,d.}
\create@abbreviation{almostEverywhere}{a.\,e.}
\create@abbreviation{almostSurely}{a.\,s.}
\create@abbreviation{withoutlog}{w.\,l.\,o.\,g.}
\create@abbreviation{Withoutlog}{W.\,l.\,o.\,g.}
\create@abbreviation{resp}{resp.}
\create@abbreviation{posdef}{p.d.}

\def\crefname@full#1#2#3#4#5{%
    \crefname{#1}{#2}{#3}
    \Crefname{#1}{#4}{#5}
}
\def\crefname@fullmod#1#2#3#4#5{%
    \crefname@full{#1}{#2}{#3}{#4}{#5}
    \crefname@full{#1@basic}{#2}{#3}{#4}{#5}
    \crefname@full{#1@withName}{#2}{#3}{#4}{#5}
}
\crefname@full{chapter}{chapter}{chapters}{Chapter}{Chapters}
\crefname@full{appendix}{appendix}{appendices}{Appendix}{Appendices}
\crefname@full{section}{section}{sections}{Section}{Sections}
\crefname@full{subsection}{section}{sections}{Section}{Sections}
\crefname@full{subsubsection}{section}{sections}{Section}{Sections}
\crefname@full{subsubsubsection}{section}{sections}{Section}{Sections}
\crefname@full{table}{table}{tables}{Table}{Tables}
\crefname@full{figure}{figure}{figures}{Figure}{Figures}
\crefname@full{subfigure}{figure}{figures}{Figure}{Figures}

\crefname@fullmod{thm}{theorem}{theorems}{Theorem}{Theorems}
\crefname@fullmod{thmStar}{theorem}{theorems}{Theorem}{Theorems}
\crefname@fullmod{conj}{conjecture}{conjectures}{Conjecture}{Conjectures}
\crefname@fullmod{cor}{corollary}{corollaries}{Corollary}{Corollaries}
\crefname@fullmod{defn}{definition}{definitions}{Definition}{Definitions}
\crefname@fullmod{conv}{convention}{conventions}{Convention}{Conventions}
\crefname@fullmod{e.g.}{example}{examples}{Example}{Examples}
\crefname@fullmod{prop}{proposition}{propositions}{Proposition}{Propositions}
\crefname@fullmod{proof}{proof}{proofs}{Proof}{Proofs}
\crefname@fullmod{lemm}{lemma}{lemmata}{Lemma}{Lemmata}
\crefname@fullmod{problem}{problem}{problems}{Problem}{Problems}
\crefname@fullmod{qstn}{question}{questions}{Question}{Questions}
\crefname@fullmod{rem}{remark}{remarks}{Remark}{Remarks}

\def\qedEIGEN#1{\@ifnextchar[{\qedEIGEN@c{#1}}{\qedEIGEN@bes{#1}}}%]
\def\qedEIGEN@bes#1{%
    \parfillskip=0pt%            % so \par doesnt push \square to left
    \widowpenalty=10000%         % so we dont break the page before \square
    \displaywidowpenalty=10000%  % ditto
    \finalhyphendemerits=0%      % TeXbook exercise 14.32
    \leavevmode%                 % \nobreak means lines not pages
    \unskip%                     % remove previous space or glue
    \nobreak%                    % don’t break lines
    \hfil%                       % ragged right if we spill over
    \penalty50%                  % discouragement to do so
    \hskip.2em%                  % ensure some space
    \null%                       % anchor following \hfill
    \hfill%                      % push \square to right
    #1%                          % the end-of-proof mark
    \par%
}
\def\qedEIGEN@c#1[#2]{%
    \parfillskip=0pt%            % so \par doesnt push \square to left
    \widowpenalty=10000%         % so we dont break the page before \square
    \displaywidowpenalty=10000%  % ditto
    \finalhyphendemerits=0%      % TeXbook exercise 14.32
    \leavevmode%                 % \nobreak means lines not pages
    \unskip%                     % remove previous space or glue
    \nobreak%                    % don’t break lines
    \hfil%                       % ragged right if we spill over
    \penalty50%                  % discouragement to do so
    \hskip.2em%                  % ensure some space
    \null%                       % anchor following \hfill
    \hfill%                      % push \square to right
    {#1~{\small\bfseries\upshape (#2)}}%
    \par%
}
\def\qedVARIANT#1#2{
    \expandafter\def\csname ennde#1Sign\endcsname{#2}
    \expandafter\def\csname ennde#1\endcsname{\@ifnextchar[{\qedEIGEN@c{#2}}{\qedEIGEN@bes{#2}}} %]
}
\qedVARIANT{OfProof}{$\squareblack$}
\qedVARIANT{OfWork}{\rectangleblack}
\qedVARIANT{OfSomething}{$\lrcorner$}
\qedVARIANT{OnNeutral}{} % \lozenge \bigcirc \blacklozenge

\def\ra@pretheoremwork{
    \setlength{\theorempreskipamount}{\ownspaceabovethm}
    \setlength{\theorempostskipamount}{\ownspacebelowthm}
}
\def\rathmtransfer#1#2{
    \expandafter\def\csname #2\endcsname{\csname #1\endcsname}
    \expandafter\def\csname end#2\endcsname{\csname end#1\endcsname}
}

\def\ranewthm#1#2#3[#4]{
    \theoremstyle{\current@theoremstyle}
    \theoremseparator{\current@theoremseparator}
    \theoremprework{\ra@pretheoremwork}
    \@ifundefined{#1@basic}{\newtheorem{#1@basic}[#4]{#2}}{\renewtheorem{#1@basic}[#4]{#2}}
    \theoremstyle{\current@theoremstyle}
    \theoremseparator{\thmForceSepPt}
    \theoremprework{\ra@pretheoremwork}
    \@ifundefined{#1@withName}{\newtheorem{#1@withName}[#4]{#2}}{\renewtheorem{#1@withName}[#4]{#2}}
    \theoremstyle{nonumberplain}
    \theoremseparator{\thmForceSepPt}
    \theoremprework{\ra@pretheoremwork}
    \@ifundefined{#1@star@basic}{\newtheorem{#1@star@basic}[#4]{#2}}{\renewtheorem{#1@star@basic}[#4]{#2}}
    \theoremstyle{nonumberplain}
    \theoremseparator{\thmForceSepPt}
    \theoremprework{\ra@pretheoremwork}
    \@ifundefined{#1@star@withName}{\newtheorem{#1@star@withName}[#4]{#2}}{\renewtheorem{#1@star@withName}[#4]{#2}}
    \umbauenenv{#1}{#3}[#4]
    \umbauenenv{#1@star}{#3}[#4]
    \rathmtransfer{#1@star}{#1*}
}

\def\umbauenenv#1#2[#3]{%
    \expandafter\def\csname #1\endcsname{\relax%
        \@ifnextchar[{\csname #1@\endcsname}{\csname #1@\endcsname[*]}%
    }
    \expandafter\def\csname #1@\endcsname[##1]{\relax%
        \@ifnextchar[{\csname #1@@\endcsname[##1]}{\csname #1@@\endcsname[##1][*]}%
    }
    \expandafter\def\csname #1@@\endcsname[##1][##2]{%
        \ifx*##1%
            \def\enndeOfBlock{\csname end#1@basic\endcsname}
            \csname #1@basic\endcsname%
        \else%
            \def\enndeOfBlock{\csname end#1@withName\endcsname}
            \csname #1@withName\endcsname[##1]%
        \fi%
        \def\makelabel####1{%
            \gdef\beweislabel{####1}%
            \label{\beweislabel}%
        }%
        \ifx*##2%
            \def\enndeSymbol{\qedEIGEN{#2}}
        \else%
            \def\enndeSymbol{\qedEIGEN{#2}[##2]}
        \fi
    }
    \expandafter\gdef\csname end#1\endcsname{\enndeSymbol\enndeOfBlock}
}

    \def\current@theoremstyle{plain}
    \def\current@theoremseparator{\thmnumberingseppt}
    \theoremstyle{\current@theoremstyle}
    \theoremseparator{\current@theoremseparator}
    \theoremsymbol{}

\generatenestedthmnumbering{arabic}{section}{equation}

\generatenestedthmnumbering{arabic}{section}{X}
\generatenestedthmnumbering{Roman}{section}{Xsp}

    \theoremheaderfont{\upshape\bfseries}
    \theorembodyfont{\slshape}

\ranewthm{thm}{Theorem}{\enndeOfSomethingSign}[X]
\ranewthm{lemm}{Lemma}{\enndeOfSomethingSign}[X]
\ranewthm{cor}{Corollary}{\enndeOfSomethingSign}[X]
\ranewthm{prop}{Proposition}{\enndeOfSomethingSign}[X]

    \theorembodyfont{\upshape}

\ranewthm{defn}{Definition}{\enndeOfSomethingSign}[X]
\ranewthm{conv}{Convention}{\enndeOfSomethingSign}[X]
\ranewthm{e.g.}{Example}{\enndeOfSomethingSign}[X]
\ranewthm{fact}{Fact}{\enndeOfSomethingSign}[X]
\ranewthm{problem}{Problem}{\enndeOfSomethingSign}[X]
\ranewthm{qstn}{Question}{\enndeOfSomethingSign}[X]
\ranewthm{rem}{Remark}{\enndeOfSomethingSign}[X]

    \theoremheaderfont{\itshape}
    \theorembodyfont{\upshape}

\ranewthm{proof@tmp}{Proof}{\enndeOfProofSign}[Xdisplaynone]
\rathmtransfer{proof@tmp*}{proof}

\def\shortclaim@claim{%
    \iflanguage{british}{Claim}{%
    \iflanguage{english}{Claim}{%
    \iflanguage{ngerman}{Behauptung}{%
    \iflanguage{russian}{Утверждение}{%
    Claim%
    }}}}%
}
\def\shortclaim@pf@kurz{%
    \iflanguage{british}{Pf}{%
    \iflanguage{english}{Pf}{%
    \iflanguage{ngerman}{Bew}{%
    \iflanguage{russian}{Доказательство}{%
    Pf%
    }}}}%
}
\def\shortclaim{\@ifnextchar\bgroup{\shortclaim@c}{\shortclaim@bes}}
    \def\shortclaim@c#1{\item[{\bfseries \shortclaim@claim\forceaddspace #1.}]}
    \def\shortclaim@bes{\item[{\bfseries \shortclaim@claim.}]}
\def\proofofshortclaim{\item[{\bfseries\itshape\shortclaim@pf@kurz.}]}

\newdateformat{standardshort}{\oldstylenums{\THEYEAR}.\oldstylenums{\THEMONTH}.\oldstylenums{\THEDAY}}
\newdateformat{standardcompact}{\THEYEAR\twodigit{\THEMONTH}\twodigit{\THEDAY}}
\newdateformat{standardlong}{\THEYEAR\ \monthname\ \THEDAY}
\newcolumntype{\RECHTS}[1]{>{\raggedleft}p{#1}}
\newcolumntype{\LINKS}[1]{>{\raggedright}p{#1}}
\newcolumntype{m}{>{$}l<{$}}
\newcolumntype{C}{>{$}c<{$}}
\newcolumntype{L}{>{$}l<{$}}
\newcolumntype{R}{>{$}r<{$}}
\newcolumntype{0}{@{\hspace{0pt}}}
\newcolumntype{\LINKSRAND}{@{\hspace{\@totalleftmargin}}}
\newcolumntype{h}{@{\extracolsep{\fill}}}
\newcolumntype{i}{>{\itshape}}
\newcolumntype{t}{@{\hspace{\tabcolsep}}}
\newcolumntype{q}{@{\hspace{1em}}}
\newcolumntype{n}{@{\hspace{-\tabcolsep}}}
\newcolumntype{M}[2]{%
    >{\begin{minipage}{#2}\begin{math}}%
    {#1}%
    <{\end{math}\end{minipage}}%
}
\newcolumntype{T}[2]{%
    >{\begin{minipage}{#2}}%
    {#1}%
    <{\end{minipage}}%
}
\def\punkteumgebung@genbefehl#1#2#3{
    \punkteumgebung@genbefehl@{#1}{#2}{#3}{}{}
    \punkteumgebung@genbefehl@{multi#1}{#2}{#3}{
        \setlength{\columnsep}{10pt}%
        \setlength{\columnseprule}{0pt}%
        \begin{multicols}{\thecolumnanzahl}%
    }{\end{multicols}\nvraum{1}}
}
\def\punkteumgebung@genbefehl@#1#2#3#4#5{
    \expandafter\gdef\csname #1\endcsname{
        \@ifnextchar\bgroup{\csname #1@c\endcsname}{\csname #1@bes\endcsname}
    }%]
        \expandafter\def\csname #1@c\endcsname##1{
            \@ifnextchar[{\csname #1@c@\endcsname{##1}}{\csname #1@c@\endcsname{##1}[\z@]}
        }%]
        \expandafter\def\csname #1@c@\endcsname##1[##2]{
            \@ifnextchar[{\csname #1@c@@\endcsname{##1}[##2]}{\csname #1@c@@\endcsname{##1}[##2][\z@]}
        }%]
        \expandafter\def\csname #1@c@@\endcsname##1[##2][##3]{
            \let\alterlinkerRand\gesamtlinkerRand
            \let\alterrechterRand\gesamtrechterRand
            \addtolength{\gesamtlinkerRand}{##2}
            \addtolength{\gesamtrechterRand}{##3}
            \advance\linewidth -##2%
            \advance\linewidth -##3%
            \advance\@totalleftmargin ##2%
            \parshape\@ne \@totalleftmargin\linewidth%
            #4
            \begin{#2}[\upshape ##1]%
                \setlength{\parskip}{0.5\baselineskip}\relax%
                \setlength{\topsep}{\z@}\relax%
                \setlength{\partopsep}{\z@}\relax%
                \setlength{\parsep}{\parskip}\relax%
                \setlength{\itemsep}{#3}\relax%
                \setlength{\listparindent}{\z@}\relax%
                \setlength{\itemindent}{\z@}\relax%
        }
        \expandafter\def\csname #1@bes\endcsname{
            \@ifnextchar[{\csname #1@bes@\endcsname}{\csname #1@bes@\endcsname[\z@]}
        }%]
        \expandafter\def\csname #1@bes@\endcsname[##1]{
            \@ifnextchar[{\csname #1@bes@@\endcsname[##1]}{\csname #1@bes@@\endcsname[##1][\z@]}
        }%]
        \expandafter\def\csname #1@bes@@\endcsname[##1][##2]{
            \let\alterlinkerRand\gesamtlinkerRand
            \let\alterrechterRand\gesamtrechterRand
            \addtolength{\gesamtlinkerRand}{##1}
            \addtolength{\gesamtrechterRand}{##2}
            \advance\linewidth -##1%
            \advance\linewidth -##2%
            \advance\@totalleftmargin ##1%
            \parshape\@ne \@totalleftmargin\linewidth%
            #4
            \begin{#2}%
                \setlength{\parskip}{0.5\baselineskip}\relax%
                \setlength{\topsep}{\z@}\relax%
                \setlength{\partopsep}{\z@}\relax%
                \setlength{\parsep}{\parskip}\relax%
                \setlength{\itemsep}{#3}\relax%
                \setlength{\listparindent}{\z@}\relax%
                \setlength{\itemindent}{\z@}\relax%
        }
    \expandafter\gdef\csname end#1\endcsname{%
        \end{#2}#5
        \setlength{\gesamtlinkerRand}{\alterlinkerRand}
        \setlength{\gesamtlinkerRand}{\alterrechterRand}
    }
}

\def\ritempunkt{{\Large \textbullet}} % \textbullet, $\sqbullet$, $\blacktriangleright$
\setdefaultitem{\ritempunkt}{\ritempunkt}{\ritempunkt}{\ritempunkt}
\punkteumgebung@genbefehl{itemise}{compactitem}{\parskip}{}{}
\punkteumgebung@genbefehl{kompaktitem}{compactitem}{\z@}{}{}
\punkteumgebung@genbefehl{kompaktenum}{compactenum}{\z@}{}{}

\def\enumerate{%
    \@ifnextchar\bgroup{%
        \enumerate@legacyarg%
    }{%
        \@ifnextchar[{\enumerate@args}{\enumerate@noargs}
    }%
}
    \def\enumerate@spacing{
        \setlength{\parskip}{0.5\baselineskip}\relax%
        \setlength{\topsep}{\z@}\relax%
        \setlength{\partopsep}{\z@}\relax%
        \setlength{\parsep}{\parskip}\relax%
        \setlength{\itemsep}{\z@}\relax%
        \setlength{\listparindent}{\z@}\relax%
        \setlength{\itemindent}{\z@}\relax%
    }
    \def\enumerate@noargs{
        \begin{oldenumerate}
        \enumerate@spacing
    }
    \def\enumerate@args[#1]{
        \begin{oldenumerate}[#1]
        \enumerate@spacing
    }
    \def\enumerate@legacyarg#1{
        \begin{oldenumerate}[label=#1]
        \enumerate@spacing
    }
    \def\endenumerate{%
        \end{oldenumerate}
    }

\def\shorteqnarray{%
    \bgroup% <- NOTE: this might be unnecessary
    \setlength{\abovedisplayshortskip}{\aboveequation}%
    \setlength{\belowdisplayshortskip}{\belowequation}%
    \setlength{\abovedisplayskip}{\aboveequation - \baselineskip}%
    \setlength{\belowdisplayskip}{\belowequation}%
    \begin{eqnarray*}%
}
\def\endshorteqnarray{%
    \end{eqnarray*}%
    \egroup% <- NOTE: this might be unnecessary
}
\def\longeqnarray{%
    \bgroup%
    \allowdisplaybreaks%
    \setlength{\abovedisplayshortskip}{\aboveequation}%
    \setlength{\belowdisplayshortskip}{\belowequation}%
    \setlength{\abovedisplayskip}{\aboveequation - \baselineskip}%
    \setlength{\belowdisplayskip}{\belowequation}%
    \begin{eqnarray*}
}
\def\endlongeqnarray{%
    \end{eqnarray*}%
    \egroup%
}
\def\displayarray[#1]#2{
    \bgroup
    \everymath={\displaystyle}
    \begin{array}[#1]{#2}
}
\def\enddisplayarray{
    \end{array}
    \egroup
}

\def\matrix#1{\left(\begin{array}[mc]{#1}}
    \def\endmatrix{\end{array}\right)}
\def\smatrix{\left(\begin{smallmatrix}}
    \def\endsmatrix{\end{smallmatrix}\right)}

\def\multiargrekursiverbefehl#1#2#3#4#5#6#7#8{%
    \expandafter\gdef\csname#1\endcsname #2##1#4{\csname #1@anfang\endcsname##1#3\egroup}
    \expandafter\def\csname #1@anfang\endcsname##1#3{#5##1\@ifnextchar\egroup{\csname #1@ende\endcsname}{#7\csname #1@mitte\endcsname}}
    \expandafter\def\csname #1@mitte\endcsname##1#3{#6##1\@ifnextchar\egroup{\csname #1@ende\endcsname}{#7\csname #1@mitte\endcsname}}
    \expandafter\def\csname #1@ende\endcsname##1{#8}
}
\multiargrekursiverbefehl{svektor}{[}{;}{]}{\begin{smatrix}}{}{\\}{\\\end{smatrix}}
\multiargrekursiverbefehl{vektor}{[}{;}{]}{\begin{matrix}{c}}{}{\\}{\\\end{matrix}}
\multiargrekursiverbefehl{vektorzeile}{}{,}{;}{}{&}{}{}
\multiargrekursiverbefehl{matlabmatrix}{[}{;}{]}{\begin{smatrix}\vektorzeile}{\vektorzeile}{;\\}{;\end{smatrix}}

\def\underbracenodisplay#1{%
    \mathop{\vtop{\m@th\ialign{##\crcr
    $\hfil\displaystyle{#1}\hfil$\crcr
    \noalign{\kern3\p@\nointerlineskip}%
    \upbracefill\crcr\noalign{\kern3\p@}}}}\limits%
}

\def\changemargins{\@ifnextchar[{\indents@}{\indents@[\z@]}}%]
\def\indents@[#1]{\@ifnextchar[{\indents@@[#1]}{\indents@@[#1][\z@]}}%]
\def\indents@@[#1][#2]{%
    \begin{list}{}{\relax
        \setlength{\topsep}{\z@}\relax
        \setlength{\partopsep}{\z@}\relax
        \setlength{\parsep}{\parskip}\relax
        \setlength{\listparindent}{\z@}\relax
        \setlength{\itemindent}{\z@}\relax
        \setlength{\leftmargin}{#1}\relax
        \setlength{\rightmargin}{#2}\relax
        \let\alterlinkerRand\gesamtlinkerRand
        \let\alterrechterRand\gesamtrechterRand
        \addtolength{\gesamtlinkerRand}{#1}
        \addtolength{\gesamtrechterRand}{#2}
    }\relax
        \item[]\relax
}
    \def\endchangemargins{%
        \setlength{\gesamtlinkerRand}{\alterlinkerRand}
        \setlength{\gesamtlinkerRand}{\alterrechterRand}
        \end{list}%
    }

\def\indentonce{\begin{changemargins}[\rtab][\rtab]}
    \def\endindentonce{\end{changemargins}}

\def\restoremargins{\begin{changemargins}[-\gesamtlinkerRand][-\gesamtrechterRand]}
    \def\endrestoremargins{\end{changemargins}}

\def\programmiercode{
    \modulolinenumbers[1]
    \begin{changemargins}[\rtab][\rtab]%
    \begin{linenumbers}%
        \fontfamily{cmtt}\fontseries{m}\fontshape{u}\selectfont%
        \setlength{\parskip}{1\baselineskip}%
        \setlength{\parindent}{0pt}%
}
    \def\endprogrammiercode{
        \end{linenumbers}
        \end{changemargins}
    }

\def\schattiertebox@genbefehl#1#2#3{
    \expandafter\gdef\csname #1\endcsname{%
        \@ifnextchar[{\csname #1@args\endcsname}{\csname #1@args\endcsname[#3]}%]%
    }
        \expandafter\def\csname #1@args\endcsname[##1]{%
            \@ifnextchar[{\csname #1@args@l\endcsname[##1]}{\csname #1@args@n\endcsname[##1]}%]%
        }
        \expandafter\def\csname #1@args@l\endcsname[##1][##2]{%
            \@ifnextchar[{\csname #1@args@l@r\endcsname[##1][##2]}{\csname #1@args@l@n\endcsname[##1][##2]}%]%
        }
        \expandafter\def\csname #1@args@n\endcsname[##1]{%
            \let\boolinmdframed\boolwahr
            \begin{mdframed}[#2leftmargin=0,rightmargin=0,outermargin=0,innermargin=0,##1]
        }
        \expandafter\def\csname #1@args@l@n\endcsname[##1][##2]{%
            \let\boolinmdframed\boolwahr
            \begin{mdframed}[#2leftmargin=##2/2,rightmargin=##2/2,outermargin=##2/2,innermargin=##2/2,##1]
        }
        \expandafter\def\csname #1@args@l@r\endcsname[##1][##2][##3]{%
            \let\boolinmdframed\boolwahr
            \begin{mdframed}[#2leftmargin=##2,rightmargin=##3,outermargin=##2,innermargin=##3,##1]
        }
    \expandafter\gdef\csname end#1\endcsname{%
        \end{mdframed}
        \let\boolinmdframed\boolfalsch
    }
}
    \schattiertebox@genbefehl{schattiertebox}{
        splittopskip=0,%
        splitbottomskip=0,%
        frametitleaboveskip=0,%
        frametitlebelowskip=0,%
        skipabove=1\baselineskip,%
        skipbelow=1\baselineskip,%
        linewidth=2pt,%
        linecolor=black,%
        roundcorner=4pt,%
    }{
        backgroundcolor=leer,%
        nobreak=true,%
    }

    \schattiertebox@genbefehl{schattierteboxdunn}{
        splittopskip=0,%
        splitbottomskip=0,%
        frametitleaboveskip=0,%
        frametitlebelowskip=0,%
        skipabove=1\baselineskip,%
        skipbelow=1\baselineskip,%
        linewidth=1pt,%
        linecolor=black,%
        roundcorner=2pt,%
    }{
        backgroundcolor=leer,%
        nobreak=true,%
    }

    \schattiertebox@genbefehl{highlightbox}{
        splittopskip=0,%
        splitbottomskip=0,%
        frametitleaboveskip=0,%
        frametitlebelowskip=0,%
        skipabove=1\baselineskip,%
        skipbelow=1\baselineskip,%
        linewidth=0pt,%
        linecolor=black,%
        roundcorner=2pt,%
    }{
        backgroundcolor=background_light_grey,%
        nobreak=true,%
    }

    \schattiertebox@genbefehl{highlightboxWithBreaks}{
        splittopskip=0,%
        splitbottomskip=0,%
        frametitleaboveskip=0,%
        frametitlebelowskip=0,%
        skipabove=1\baselineskip,%
        skipbelow=1\baselineskip,%
        linewidth=0pt,%
        linecolor=black,%
        roundcorner=2pt,%
    }{
        backgroundcolor=background_light_grey,%
    }

\def\tikzsetzepfeil#1{%
    \begin{tikzpicture}[remember picture,overlay,>=latex]%
        \draw #1;%
    \end{tikzpicture}%
}

\def\tikzsetzekreise[#1]#2#3{%
    \tikzsetzepfeil{%
    [rounded corners,#1]%
        ([shift={(-\tabcolsep,0.75\baselineskip)}]#2)%
        rectangle%
        ([shift={(\tabcolsep,-0.5\baselineskip)}]#3)
    }%
}

\tikzset{
    >=stealth,
    auto,
    node distance=1cm,
    thick,
    main node/.style={
        circle,draw,font=\sffamily\Large\bfseries,minimum size=0pt
    },
    state/.style={minimum size=0pt}
    loop above right/.style={loop,out=30,in=60,distance=0.5cm},
    loop above left/.style={above left,out=150,in=120,loop},
    loop below right/.style={below right,out=330,in=300,loop},
    loop below left/.style={below left,out=240,in=210,loop},
    itria/.style={
        draw,dashed,shape border uses incircle,
        isosceles triangle,shape border rotate=90,yshift=-1.45cm
    },
    rtria/.style={
        draw,dashed,shape border uses incircle,
        isosceles triangle,isosceles triangle apex angle=90,
        shape border rotate=-45,yshift=0.2cm,xshift=0.5cm
    },
    ritria/.style={
        draw,dashed,shape border uses incircle,
        isosceles triangle,isosceles triangle apex angle=110,
        shape border rotate=-55,yshift=0.1cm
    },
    litria/.style={
        draw,dashed,shape border uses incircle,
        isosceles triangle,isosceles triangle apex angle=110,
        shape border rotate=235,yshift=0.1cm
    }
}

\renewenvironment{cases}[0]{\left\{\begin{array}[c]{0lcl}}{\end{array}\right.}

\renewenvironment{displaycases}[0]{%
    \left\{\relax%
    \bgroup\relax%
    \everymath={\displaystyle}\relax%
    \begin{array}[c]{0lcl}\relax%
}{%
    \end{array}\relax%
    \egroup\relax%
    \right.\relax%
}

\providecommand{\usesinglequotes}{}
\renewcommand{\usesinglequotes}[1]{`#1'}

\providecommand{\zeroone}{}
\renewcommand{\zeroone}[0]{\textup{0\=/1}\relax\ifmmode\else\@\xspace\fi}
\providecommand{\onetoone}{}
\renewcommand{\onetoone}[0]{\ensuremath{1\!\!:\!\!1}\relax\ifmmode\else\@\xspace\fi}
\providecommand{\First}{}
\renewcommand{\First}[0]{\text{I\textsuperscript{st}}\relax\ifmmode\else\@\xspace\fi}
\providecommand{\Second}{}
\renewcommand{\Second}[0]{\text{II\textsuperscript{nd}}\relax\ifmmode\else\@\xspace\fi}
\providecommand{\Third}{}
\renewcommand{\Third}[0]{\text{III\textsuperscript{rd}}\relax\ifmmode\else\@\xspace\fi}

\providecommand{\TextCStarAlg}{}
\renewcommand{\TextCStarAlg}[0]{\text{C\textsuperscript{\ensuremath{\ast}}\=/algebra}\relax\ifmmode\else\@\xspace\fi}
\providecommand{\TextCStarSubAlg}{}
\renewcommand{\TextCStarSubAlg}[0]{\text{C\textsuperscript{\ensuremath{\ast}}\=/subalgebra}\relax\ifmmode\else\@\xspace\fi}
\providecommand{\TextCStarAlgs}{}
\renewcommand{\TextCStarAlgs}[0]{\text{C\textsuperscript{\ensuremath{\ast}}\=/algebras}\relax\ifmmode\else\@\xspace\fi}
\providecommand{\TextCStarSubAlgs}{}
\renewcommand{\TextCStarSubAlgs}[0]{\text{C\textsuperscript{\ensuremath{\ast}}\=/subalgebras}\relax\ifmmode\else\@\xspace\fi}

\providecommand{\envPreMathsLong}{}
\renewcommand{\envPreMathsLong}[0]{%
    \bgroup\relax%
    \let\old@arraystretch\arraystretch\relax%
    \renewcommand\arraystretch{1.2}\relax\relax%
}
\providecommand{\envPostMathsLong}{}
\renewcommand{\envPostMathsLong}[0]{%
    \renewcommand\arraystretch{\old@arraystretch}\relax%
    \egroup\relax%
}
\providecommand{\id}{}
\renewcommand{\id}[0]{\mathrm{\textit{id}}}

\providecommand{\complex}{}
\renewcommand{\complex}[0]{\mathbb{C}}

\providecommand{\reals}{}
\renewcommand{\reals}[0]{\mathbb{R}}
\providecommand{\realsNonNeg}{}
\renewcommand{\realsNonNeg}[0]{\reals_{\geq 0}}

\providecommand{\integers}{}
\renewcommand{\integers}[0]{\mathbb{Z}}
\providecommand{\naturals}{}
\renewcommand{\naturals}[0]{\mathbb{N}}
\providecommand{\naturalsZero}{}
\renewcommand{\naturalsZero}[0]{\mathbb{N}_{0}}

\providecommand{\HilbertRaum}{}
\renewcommand{\HilbertRaum}[0]{\mathcal{H}}
\providecommand{\BanachRaum}{}
\renewcommand{\BanachRaum}[0]{\mathcal{E}}
\providecommand{\CStarAlg}{}
\renewcommand{\CStarAlg}[0]{\mathcal{A}}

\providecommand{\GenSet}{}
\renewcommand{\GenSet}[1]{\langle #1 \rangle}
\providecommand{\GenSetLong}{}
\renewcommand{\GenSetLong}[1]{\Big\langle #1 \Big\rangle}
\providecommand{\GenSetBy}{}
\renewcommand{\GenSetBy}[2]{\langle #1 \mathrel{\vert} #2 \rangle}
\providecommand{\GenSetByLong}{}
\renewcommand{\GenSetByLong}[2]{\Big\langle #1 \mathrel{\Big\vert} #2 \Big\rangle}

\providecommand{\topSOT}{}
\renewcommand{\topSOT}[0]{\text{\upshape\scshape sot}}

\providecommand{\card}{}
\renewcommand{\card}[1]{\lvert #1 \rvert}

\providecommand{\onefct}{}
\renewcommand{\onefct}[0]{1\!\!1}

\providecommand{\Gph}{}
\renewcommand{\Gph}[1]{\mathcal{G}\mathrm{ph}(#1)}

\providecommand{\iunit}{}
\renewcommand{\iunit}[0]{\imath}
\providecommand{\abs}{}
\renewcommand{\abs}[1]{\lvert #1 \rvert}

\providecommand{\Aut}{}
\renewcommand{\Aut}[1]{\mathrm{Aut}(#1)}

\providecommand{\linspann}{}
\renewcommand{\linspann}[0]{\textup{lin}}
\providecommand{\adjoint}{}
\renewcommand{\adjoint}[0]{\text{\upshape ad}}
\providecommand{\tr}{}
\renewcommand{\tr}[0]{\text{\upshape tr}}
\providecommand{\onematrix}{}
\renewcommand{\onematrix}[0]{\text{\upshape\bfseries I}}
\providecommand{\zeromatrix}{}
\renewcommand{\zeromatrix}[0]{\mathbf{0}}

\providecommand{\brkt}{}
\renewcommand{\brkt}[2]{\langle{}#1,\:#2{}\rangle}

\providecommand{\ketbra}{}
\renewcommand{\ketbra}[2]{\ket{#1}{}\bra{#2}}
\providecommand{\norm}{}
\renewcommand{\norm}[1]{\lVert #1 \rVert}
\providecommand{\normLong}{}
\renewcommand{\normLong}[1]{\Big\| #1 \Big\|}

\providecommand{\opDomain}{}
\renewcommand{\opDomain}[1]{\mathcal{D}(#1)}

\providecommand{\opResolventSet}{}
\renewcommand{\opResolventSet}[1]{\rho(#1)}

\providecommand{\BoundedOpsSymbol}{}
\renewcommand{\BoundedOpsSymbol}[0]{\mathfrak{L}}
\providecommand{\BaseVector}{}
\renewcommand{\BaseVector}[1]{\mathbf{e}_{#1}}
\providecommand{\ElementaryMatrix}{}
\renewcommand{\ElementaryMatrix}[2]{\mathbf{E}_{#1,#2}}
\providecommand{\Cnought}{}
\renewcommand{\Cnought}[0]{\mathcal{C}_{0}}

\providecommand{\Repr}{}
\renewcommand{\Repr}[2]{\mathrm{Repr}(#1 \!:\! #2)}

\providecommand{\dee}{}
\renewcommand{\dee}[0]{\mathop{\textup{d}}\!}

\providecommand{\letter}{}
\renewcommand{\letter}[1]{\mathbf{k}_{#1}}
\providecommand{\Irred}{}
\renewcommand{\Irred}[1]{\mathrm{Irr}(#1)}
\providecommand{\cover}{}
\renewcommand{\cover}[0]{\mathrm{cover}}

\def\Cts{\@ifnextchar_{\Cts@tief}{\Cts@tief_{}}}
    \def\Cts@tief_#1#2{\@ifnextchar\bgroup{\Cts@two_{#1}{#2}}{\Cts@one_{#1}{#2}}}
    \def\Cts@one_#1#2{C_{#1}\big(#2\big)}
    \def\Cts@two_#1#2#3{C_{#1}\big(#2,~#3\big)}

\def\BoundedOps#1{\@ifnextchar\bgroup{\BoundedOps@two{#1}}{\mathop{\BoundedOpsSymbol}(#1)}}
    \def\BoundedOps@two#1#2{\mathop{\BoundedOpsSymbol}(#1,#2)}
\def\BoundedOpsInv#1{\@ifnextchar\bgroup{\BoundedOps@two{#1}}{\mathop{\BoundedOpsSymbol}(#1)^{\times}}}
    \def\BoundedOpsInv@two#1#2{\mathop{\BoundedOpsSymbol}(#1,#2)^{\times}}

\def\restr#1{\vert_{#1}}
\def\without{\mathbin{\setminus}}

\def\eps{\varepsilon}
\let\altphi\phi
\let\altvarphi\varphi
    \def\phi{\altvarphi}
    \def\varphi{\altphi}
\def\quer#1{\overline{#1}}

\def\lim{\mathop{\ell\mathrm{im}}}
\def\supp{\mathop{\textup{supp}}}
\def\dim{\mathop{\textup{dim}}}

\def\ran{\mathop{\textup{ran}}}

\def\Re{\mathop{\mathfrak{R}\mathrm{e}}}

\def\interval{\mathcal{J}}

\def\Generate{\@ifnextchar[{\Generate@named}{\Generate@plain}}
    \def\Generate@named[#1]#2{\@ifnextchar\bgroup{\mathrm{#1}{}\GenSetBy{#2}}{\mathrm{#1}{}\GenSet{#2}}}
    \def\Generate@plain#1{\@ifnextchar\bgroup{\GenSetBy{#1}}{\GenSet{#1}}}
\def\GenerateLong{\@ifnextchar[{\GenerateLong@named}{\GenerateLong@plain}}
    \def\GenerateLong@named[#1]#2{\@ifnextchar\bgroup{\mathrm{#1}{}\GenSetByLong{#2}}{\mathrm{#1}{}\GenSetLong{#2}}}
    \def\GenerateLong@plain#1{\@ifnextchar\bgroup{\GenSetByLong{#1}}{\GenSetLong{#1}}}

\@ifundefined{setcitestyle}{%
}{%
    \setcitestyle{numeric-comp,open={[},close={]}}
}

\allowdisplaybreaks % <- allow long equations to overflow
\renewcommand{\arraystretch}{1}
\def\firstparagraph{\noindent}
\def\continueparagraph{\noindent}

\generatenestedsecnumbering{arabic}{section}{subsection}
\generatenestedsecnumbering{arabic}{subsection}{subsubsection}
\def\theunitnamesection{\thesection}

\def\sectionname{}

\let\appendix@orig\appendix
\def\appendix{%
    \appendix@orig%
    \let\boolinappendix\boolwahr
    \addcontentsline{toc}{part}{\appendixname}%
    \addtocontents{toc}{\protect\setcounter{tocdepth}{0}}
    \def\sectionname{Appendix}%
    \def\theunitnamesection{\Alph{section}}%
}
\def\notappendix{%
    \let\boolinappendix\boolfalse
    \addtocontents{toc}{\protect\setcounter{tocdepth}{1 }}
    \def\sectionname{}%
    \def\theunitnamesection{\arabic{section}}%
}

\def\@settitle{%
    \bgroup
    \LARGE
    \scshape
    \@title
    \egroup
}

\def\@seccntformat#1{%
    \protect\textup{%
        \protect\@secnumfont
        \expandafter\protect\csname format#1\endcsname%
        \csname the#1\endcsname
        \expandafter\protect\csname format#1@pt\endcsname%
        \space
    }%
}

\def\formatsection@text{\centering\Large\scshape}
\def\formatsection@pt{\secnumberingseppt}
\def\section{\@startsection{section}{1}{\z@}{.7\linespacing\@plus\linespacing}{.5\linespacing}{\formatsection@text}}

\def\formatsubsection@text{\flushleft\bfseries\scshape}
\def\formatsubsection@pt{\subsecnumberingseppt}
\def\subsection{\@startsection{subsection}{2}{\z@}{\z@}{\z@\hspace{1em}}{\formatsubsection@text}}

\renewcommand{\paragraph}[1]{%
    {\itshape #1}\:%
}

\DefineFNsymbols*{customAlphabet}{abcdefghijklmnopqrstuvwxyz}
\setfnsymbol{customAlphabet}

\def\footnotemark[#1]{\text{\textsuperscript{\getrefnumber{#1}}}}

\def\footnote@custom@period{24}
\providecommand{\footnote@ctr@prebump}{}
\renewcommand{\footnote@ctr@prebump}[1]{%
    \ifnum\value{#1}<\footnote@custom@period%
    \else\relax
        \setcounter{#1}{0}%
    \fi%
}
\providecommand{\footnoteref}{}
\renewcommand{\footnoteref}[1]{\protected@xdef\@thefnmark{\ref{#1}}\@footnotemark}
\let\@old@footnotetext\footnotetext
\def\footnotetext[#1]#2{%
    \footnote@ctr@prebump{footnote}%
    \addtocounter{footnote}{1}%
    \@old@footnotetext[\value{footnote}]{\label{#1}#2}%
}

\let\@old@footnote\footnote
\renewcommand{\footnote}[1]{%
    \footnote@ctr@prebump{footnote}%
    \@old@footnote{#1}%
}

\def\kopfzeiledefault{
    \lhead[]{}
    \lhead[]{}
    \chead[]{}
    \rhead[]{}
    \lfoot[]{}
    \cfoot{\footnotesize\thepage}
    \rfoot[]{}
}

\def\aktuellesfont{\csname lmodern\endcsname}
\def\documentfont{%
    \gdef\aktuellesfont{\csname lmodern\endcsname}%
    \fontfamily{lmr}\fontseries{m}\selectfont%
    \renewcommand{\sfdefault}{phv}%
    \renewcommand{\ttdefault}{pcr}%
    \renewcommand{\rmdefault}{cmr}% <— funktionieren nicht mit {ptm}
    \renewcommand{\bfdefault}{bx}%
    \renewcommand{\itdefault}{it}%
    \renewcommand{\sldefault}{sl}%
    \renewcommand{\scdefault}{sc}%
    \renewcommand{\updefault}{n}%
}

\allowdisplaybreaks

\def\startdocumentlayoutoptions{
    \selectlanguage{british}
    \setlength{\parskip}{0.25\baselineskip}
    \setlength{\parindent}{2em}
    \kopfzeiledefault
    \documentfont
    \normalsize
}

\providecommand{\highlightTerm}{}
\renewcommand{\highlightTerm}[1]{\emph{#1}}

\renewenvironment{highlightForReview}[0]{\color{blue}}{}

\def\addresseshere{%
  \bgroup
  \setlength{\parindent}{0pt}
  \enddoc@text
  \egroup
  \let\enddoc@text\relax
}

\makeatother

\begin{document}
\startdocumentlayoutoptions

%% FRONTMATTER:
\thispagestyle{plain}

%% ********************************************************************************
%% FILE: front/.index.tex
%% ********************************************************************************

%% ********************************************************************************
%% FILE: front/abstract.tex
%% ********************************************************************************

\def\abstractname{Abstract}
\begin{abstract}
    To generalise evolution families
    we consider systems
        $\{\phi(u, v)\}_{(u, v) \in E}$
    of contractions
    defined on the edges of a graph $\mathcal{G} = (\Omega, E)$.
    In this setup the Markov property,
    or \highlightTerm{divisibility},
    can be modelled via
        $\phi(u, v)\phi(v, w) = \phi(u, w)$
        for edges $(u, v), (v, w), (u, w) \in E$.
    We obtain results in three settings:
    1)~contractive Banach space operators;
    2)~positive unital maps on \TextCStarAlgs;
    and
    3)~completely positive trace-preserving (CPTP) maps on trace class operators on a Hilbert space.
    In the discrete setting, we are able to dilate
    possibly indivisible families of contractions
    to divisible families of operators with \usesinglequotes{nice} properties
    (\viz
        surjective isometries
        \resp {}\textsuperscript{\ensuremath{\ast}}\=/automorphisms
        \resp unitary representations%
    ).
    In the special case of linearly ordered graphs
    equipped with the order topology,
    we establish sufficient conditions
    for strongly continuous dilations
    of possibly indivisible families
    in the Banach space and \TextCStarAlg contexts.
    To achieve these results we work with string-rewriting systems,
    and make use of and extend dilation theorems of
        Stroescu \cite{Stroescu1973ArticleBanachDilations},
        Kraus \cite{%
            Kraus1971Article,%
            Kraus1983%
        },
        and vom~Ende--Dirr \cite{vomEnde2019unitaryDildiscreteCPsemigroups}.
\end{abstract}

%% ********** END OF FILE: front/abstract.tex **********

%% ********************************************************************************
%% FILE: front/title.tex
%% ********************************************************************************

\title[Dilations of non-Markovian dynamical systems on graphs]{%
    \hraum Dilations of non-Markovian\hraum%
    \large
    \newline%
    \hraum dynamical systems on graphs\hraum
}

\author{Raj Dahya}
\address{Fakult\"at f\"ur Mathematik und Informatik\newline
Universit\"at Leipzig, Augustusplatz 10, D-04109 Leipzig, Germany}
\email{raj\,[\!\![dot]\!\!]\,dahya\:[\!\![at]\!\!]\:web\,[\!\![dot]\!\!]\,de}

\def\subjclassname{Mathematics Subject Classification (2020)}
\subjclass{47A20, 47D03, 46L55, 47C15, 05C22}
\keywords{Operator families on graphs; dilations; evolution families; indivisible system; reduction systems.}

\maketitle

%% ********** END OF FILE: front/title.tex **********

%% ********** END OF FILE: front/.index.tex **********

%% HAUPTTEXT:
\setcounternach{section}{1}

%% ********************************************************************************
%% FILE: body/.index.tex
%% ********************************************************************************

%% ********************************************************************************
%% FILE: body/sec-1-introduction/.index.tex
%% ********************************************************************************

\section[Introduction]{Introduction}
\label{sec:introduction:sig:article-graph-raj-dahya}

%% ********************************************************************************
%% FILE: body/sec-1-introduction/sec-0-basic.tex
%% ********************************************************************************

%% ********************************************************************************
%% FILE: body/sec-1-introduction/sec-0-para-1-history.tex
%% ********************************************************************************

\firstparagraph
The notion of \highlightTerm{evolution families} (or: \highlightTerm{propagators}),
formally introduced by Howland
    \cite[\S{}1]{Howland1974Article}
and Evans
    \cite[Definition~1.4]{Evans1976ArticlePerturb},%
\footnote{%
    Howard and Evans treated the case of reversible systems
    in Hilbert \resp Banach space settings.
    For a more general treatment,
    see \exempli
        \cite[Chapter~5]{Pazy1983Book},
        \cite[\S{}3.1]{ChiconeLatushkin1999Book}.
}
traces its origins back to Kato \cite{Kato1953Article}
as a way to solve time-dependent partial differential equations of the form

\begin{displaymath}
    \left\{
    \begin{array}[m]{rcl}
        u^{\prime}(t) &= &A_{t}u(t) + g(t),\quad t \in \interval,\:t \geq s,\\
        u(s) &= &\xi,
    \end{array}
    \right.
\end{displaymath}

\continueparagraph
where $\interval \subseteq \reals$
is a connected subset of time points
(usually with $\min\interval = 0$),
    $s \in \interval$,
    each $A_{t}$ is the generator of
    some $\Cnought$\=/semigroup
        $T_{t} = \{T_{t}(\tau)\}_{\tau\in\realsNonNeg}$
        on a common Banach space $\BanachRaum$,
    $\xi\in\opDomain{A_{s}} \subseteq \BanachRaum$,
and
    ${g : \interval \to \BanachRaum}$
    represents the effects of an external force.
Under appropriate conditions
(see \exempli
    \cite[\S{}5.3]{Pazy1983Book}%
)
the evolution of such systems
can be expressed as

\begin{displaymath}
    u(t) = \mathcal{T}(t,s)u(s) + \text{non-homogenous term}
\end{displaymath}

\continueparagraph
for each $s,t \in \interval$ with $t \geq s$,
where $\{\mathcal{T}(t,s)\}_{t,s \in \interval,~t \geq s}$
is a family of bounded operators on $\BanachRaum$.
To model such operator families,
the following properties are considered:

\begin{enumerate}[
    label={\texttt{Ev\textsubscript{\arabic*}}},
    ref={\texttt{Ev\textsubscript{\arabic*}}},
]
    \item\label{ax:ev:cts:sig:article-graph-raj-dahya}
        (\highlightTerm{Continuity})
        The map
            ${
                E \ni (t,s)
                \mapsto
                    \mathcal{T}(t,s)
                    \in \BoundedOps{\BanachRaum}
            }$
        is strongly continuous.

    \item\label{ax:ev:id:sig:article-graph-raj-dahya}
        (\highlightTerm{Identity})
        $\mathcal{T}(t,t) = \onematrix$ for all $t \in \interval$.

    \item\label{ax:ev:id:sig:article-graph-raj-dahya}
        (\highlightTerm{Memorylessness})
        $\mathcal{T}(t', s') = \mathcal{T}(t, s)$
        for all $s, t, s', t' \in \interval$
        with $t' - s' = t - s \geq 0$.

    \item\label{ax:ev:weak-div:sig:article-graph-raj-dahya}
        (\highlightTerm{Weak divisibility})
        $\mathcal{T}(t, 0) = \mathcal{T}(t, s)\mathcal{T}(s, 0)$
        for all $s, t \in \interval$
        with $t \geq s \geq 0$.

    \item\label{ax:ev:div:sig:article-graph-raj-dahya}
        (\highlightTerm{Strong divisibility})
        $\mathcal{T}(t, r) = \mathcal{T}(t, s)\mathcal{T}(s, r)$
        for all $r, s, t \in \interval$
        with $t \geq s \geq r$.
\end{enumerate}

Our primary interest is in \eqcref{ax:ev:div:sig:article-graph-raj-dahya},
which we shall simply refer to as \highlightTerm{divisibility}.
Note that on the one hand,
divisibility in the literature typically refers to \eqcref{ax:ev:weak-div:sig:article-graph-raj-dahya},
and on the other,
\eqcref{ax:ev:div:sig:article-graph-raj-dahya}
is usually bundled with requirements on the operators
and referred to as \highlightTerm{Markovianity}.
As we shall always state the operator properties separately,
we generally avoid the latter terminology.

Now, the theory of $\Cnought$-semigroups allows us to treat
memoryless, divisible dynamical systems.
And the study of evolution families in the PDE-setting
drops memorylessness, but retains divisibility.
In recent times,
\highlightTerm{indivisible} processes,
\idest systems for which (weak \resp. strong) divisibility is not assumed,
have gathered interest in the
philosophy, mathematics, and applications of physics
(see \exempli
    \cite{%
        RivasHuelgaPlenio2014ArticleNonMarkovian,%
        MilzKimPollock2019ArticleDivisibility,%
        BarandesKagan2020MinimalModalInterpretation,%
        MilzModi2021QSPNonMarkovian,%
        Barandes2025ArticleStochQuantum%
    }%
).
In foundational work,
indivisibility is seen as a means to deal with the category problem
for \usesinglequotes{measurements} in quantum mechanics
(\cf
    \cite[\S{}1]{Barandes2025MiscIndivis}%
).
Non-Markovianity also appears to provide promising approaches
for modern applications
such as
    metrology,
    state preparation in quantum computing,
    noise handling in information processing,
\etcetera
(see \exempli
    \cite[\S{}6]{RivasHuelgaPlenio2014ArticleNonMarkovian},
    \cite[\S{}VII]{MilzKimPollock2019ArticleDivisibility}%
).

%% ********** END OF FILE: body/sec-1-introduction/sec-0-para-1-history.tex **********

%% ********************************************************************************
%% FILE: body/sec-1-introduction/sec-0-para-2-scope.tex
%% ********************************************************************************

To contribute to the foundational picture,
the present work demonstrates how such processes
can be embedded (or: \highlightTerm{dilated})
into strongly divisible ones.
More abstractly, we work with families of bounded operators
    $\{\phi(u,v)\}_{(u,v) \in E} \subseteq \BoundedOps{\BanachRaum}$
on a Banach space $\BanachRaum$
defined on the edges of a graph $\mathcal{G} = (\Omega, E)$,
which can be axiomatised in a natural way
to capture the above properties,
in particular divisibility
(see \S{}\ref{sec:intro:definitions:sig:article-graph-raj-dahya} below).
It shall also be fruitful to work with systems
of the form
    $\phi = \{e^{A(u,v)}\}_{(u,v) \in E}$.
In particular, we shall see how
commutativity of the family
    $\{A(u,v)\}_{(u,v) \in E} \subseteq \BoundedOps{\BanachRaum}$
of bounded generators
plays a crucial role in determining the divisibility of $\phi$
(see \S{}\ref{sec:examples:hamiltonian:sig:article-graph-raj-dahya} below).

Considering the quantum setting,
\viz
families
    $\{\Phi_{(t,s)}(\cdot)\}_{(t,s) \in E}$
    of CPTP\=/operators (defined below)
    with $E = \{(t,s) \in \realsNonNeg^{2} \mid t \geq s\}$,
the sought after dilations are divisible unitary evolutions.
In \exempli
    \cite[\S{}V.C.]{MilzModi2021QSPNonMarkovian},
    \cite[\S{}3.4]{Barandes2025ArticleStochQuantum},
    \cite[\S{}4.2]{Barandes2025MiscIndivis},
weakly indivisible systems%
\footnote{%
    \idest (in)divisibility in the sense of \eqcref{ax:ev:weak-div:sig:article-graph-raj-dahya}.
}
are considered and their $1$\=/parameter subfamilies
    $\{\Phi_{(t, 0)}\}_{t\in\realsNonNeg}$
are dilated to unitary evolutions.
But since $1$\=/parameter dilations are necessarily memoryless,
these cannot be pieced together
to obtain meaningful dilations of $2$\=/parameter families
(\cf \Cref{rem:1-parameter-decomp:sig:article-graph-raj-dahya} below).
Others treat the full $2$\=/parameter families,
stating dilation results under the assumption of strong divisibility
\cite[\S{}3.3.2]{RivasHuelgaPlenio2014ArticleNonMarkovian},
as well as partial results without this assumption
\cite[\S{}2--3]{RybarFilippovZiman2012Article}
by relying on so-called \highlightTerm{collision models}.
We further mention the work of Wolf and Cirac \cite{WolfCirac2008Article}
especially \S{}VI and Theorem 16 of this reference,
which provides a thorough treatment of \highlightTerm{infinite divisibility}
of channels in the finite dimensional setting.

In the present paper, our first main result
(see \Cref{thm:result:graph-dilations:discrete:sig:article-graph-raj-dahya})
establishes $2$\=/parameter dilations rigorously
and under no assumptions of divisibility.
And our final results
(see \Cref{%
    thm:result:graph-dilations:cts:divisible:sig:article-graph-raj-dahya,%
    thm:result:graph-dilations:cts:indivisible:sig:article-graph-raj-dahya%
})
ensure the continuity of such dilations
under modest requirements.
The work in this paper further differs from existing literature
due to the abstract setting of systems defined on graphs,
as well as the algebraic framework we devise to derive our results
(see also \Cref{rem:literature-algebraic:sig:article-graph-raj-dahya}).

%% ********** END OF FILE: body/sec-1-introduction/sec-0-para-2-scope.tex **********

%% ********** END OF FILE: body/sec-1-introduction/sec-0-basic.tex **********

%% ********************************************************************************
%% FILE: body/sec-1-introduction/sec-1-definitions.tex
%% ********************************************************************************

\subsection[Terminology for dynamical systems on graphs]{Terminology for dynamical systems on graphs}
\label{sec:intro:definitions:sig:article-graph-raj-dahya}

\firstparagraph
Consider a graph $\mathcal{G} = (\Omega, E)$,
where $\Omega$ denotes a non-empty set of \highlightTerm{nodes}
and $E \subseteq \Omega \times \Omega$ a non-empty set of \highlightTerm{edges}.
The physical systems we aim to model are defined
by families
    $\{\phi(u,v)\}_{(u,v) \in E} \subseteq \BoundedOps{\BanachRaum}$
of bounded operators on a Banach space $\BanachRaum$.%
\footnote{%
    throughout we shall alternate between this way of writing
    and simply considering $\phi$ as a map
    ${\phi : E \to \BoundedOps{\BanachRaum}}$.
}
We consider the following axioms:

\begin{enumerate}[
    label={\texttt{Dyn\textsubscript{\arabic*}}},
    ref={\texttt{Dyn\textsubscript{\arabic*}}},
]
    \item\label{ax:dyn:cts:sig:article-graph-raj-dahya}
        (\highlightTerm{Continuity})
        If $\Omega$ is endowed with a topology,
        endow $E$ with the subspace topology of the product space $\Omega \times \Omega$.
        Letting $\tau$ be any topology on $\BoundedOps{\BanachRaum}$,
        we say that the family
            $\phi$ is $\tau$\=/continuous
        if the map
            ${E \ni (u,v) \mapsto \phi(u,v) \in \BoundedOps{\BanachRaum}}$
        is continuous \wrt $\tau$.

    \item\label{ax:dyn:id:sig:article-graph-raj-dahya}
        (\highlightTerm{Identity})
        $\phi(u,u) = \onematrix$
        for $u \in \Omega$,
        provided $(u,u) \in E$.

    \item\label{ax:dyn:div:sig:article-graph-raj-dahya}
        (\highlightTerm{Divisibility})
        $\phi(u,w) = \phi(u,v)\phi(v,w)$
        for $u,v,w \in \Omega$,
        provided $(u,v), (v,w), (u, w) \in E$.
\end{enumerate}

And for the \usesinglequotes{generators}
in families of the form
    $\{e^{A(u,v)}\}_{(u,v) \in E}$,
the following are considered:

\begin{enumerate}[
    label={\texttt{Gen\textsubscript{\arabic*}}},
    ref={\texttt{Gen\textsubscript{\arabic*}}},
]
    \item\label{ax:gen:cts:sig:article-graph-raj-dahya}
        (\highlightTerm{Continuity})
        As \eqcref{ax:dyn:cts:sig:article-graph-raj-dahya}
        but for the map
            ${E \ni (u,v) \mapsto A(u,v) \in \BoundedOps{\BanachRaum}}$.

    \item\label{ax:gen:id:sig:article-graph-raj-dahya}
        (\highlightTerm{Identity})
        $A(u,u) = \zeromatrix$
        for $u \in \Omega$,
        provided $(u,u) \in E$.

    \item\label{ax:gen:add:sig:article-graph-raj-dahya}
        (\highlightTerm{Additivity})
        $A(u,w) = A(u,v) + A(v,w)$
        for edges $(u,v), (v,w), (u, w) \in E$.
\end{enumerate}

\continueparagraph
Observe that
additivity \eqcref{ax:gen:add:sig:article-graph-raj-dahya}
clearly implies
the identity axiom \eqcref{ax:gen:id:sig:article-graph-raj-dahya},
so we shall not need to demand the latter separately.

Our main definition is as follows:

\begin{defn}
    Let $\mathcal{G} = (\Omega, E)$ be a graph
    and
        $\phi = \{\phi(u,v)\}_{(u,v) \in E} \subseteq \BoundedOps{\BanachRaum}$
        a family of bounded operators
        on a Banach space $\BanachRaum$.
    We say that $(\mathcal{G}, \phi)$ or simply $\phi$
    is a
    (norm/strongly/\etcetera continuous)
    \highlightTerm{divisible dynamical system on graph $\mathcal{G}$}
    if it satisfies
    (the appropriate variant of \eqcref{ax:dyn:cts:sig:article-graph-raj-dahya} and)
    the identity axiom \eqcref{ax:dyn:id:sig:article-graph-raj-dahya}
    and
    the divisibility axiom \eqcref{ax:dyn:div:sig:article-graph-raj-dahya}.
\end{defn}

\begin{conv}
    Throughout this paper,
    \highlightTerm{indivisibility}
    for concrete systems shall mean that the divisibility axiom fails
    and
    for classes of operator families merely that the divisibility axiom is not assumed.%
    \footnote{%
        in case of ambiguity, we shall use the terminology
        \highlightTerm{possibly indivisible} for the latter.
    }
\end{conv}

\begin{rem}[Path independence]
\makelabel{rem:path-independence:sig:article-graph-raj-dahya}
    Let $\mathcal{G} = (\Omega, E)$ be an arbitrary graph
    and suppose that an operator family
        $\{\phi(u, v)\}_{(u,v) \in E}$
    satisfies
    the divisibility axiom \eqcref{ax:dyn:div:sig:article-graph-raj-dahya}.
    For $u, v \in \Omega$
    let $\mathrm{Path}(u,v)$
    denote the set of all finite sequences
        $\pi = \{u_{k}\}_{k=0}^{n} \subseteq \Omega$
    with $n \in \naturalsZero$
    and $(u_{k-1},u_{k}) \in E$ for all $k \in \{1,2,\ldots,n\}$.
    Define
        $\phi(\pi) \coloneqq \prod_{k=1}^{n}\phi(u_{k-1},u_{k})$
    for each such walk $\pi$.
    Then if $(u, v) \in E$
        by the divisibility axiom
        $\phi(\pi) = \phi(\pi')$
        for all $\pi,\pi' \in \mathrm{Path}(u, v)$.
    The divisibility axiom thus implies a kind of \emph{path independence}.
    This suggests that if we model networks via such dynamical systems,
    indivisibility is generally unavoidable
    (\cf \S{}\ref{sec:examples:networks:sig:article-graph-raj-dahya}).
\end{rem}

In order to treat (norm-)continuity,
we shall make use of the following geometric conditions.
Let $\mathcal{G} = (\Omega, E)$ be a graph
where $\Omega$ is a topological space
and $E \subseteq \Omega \times \Omega$
is endowed with the subspace topology.

\begin{defn}[Length functions]
    Say that a function
        ${\ell : E \to [0,\:\infty)}$
    is an
    \highlightTerm{%
        additive
        (\resp subadditive \resp superadditive)
        length function%
    }
    on (the edges of) $\mathcal{G}$,
    if
        $\ell$ is continuous \wrt the topology on $E$;
        $\ell(u, u) = 0$
        for all $u \in \Omega$
        for which $(u,u) \in E$;
    and
        $\ell(u, w) = \ell(u, v) + \ell(v, w)$
        \resp
        $\ell(u, w) \leq \ell(u, v) + \ell(v, w)$
        \resp
        $\ell(u, w) \geq \ell(u, v) + \ell(v, w)$
    for all $u,v,w\in\Omega$
    for which $(u,v), (v, w), (u, w) \in E$.
\end{defn}

\begin{defn}[Geometric growth of operator families]
\makelabel{defn:geom-growth:op-family:sig:article-graph-raj-dahya}
    We shall say that a family
        $\{\phi(u,v)\}_{(u,v) \in E} \subseteq \BoundedOps{\BanachRaum}$
    of bounded operators
    has \highlightTerm{geometric growth}
    if
        $\norm{\phi(u, v) - \onematrix} \leq \ell(u, v)$
    for all edges $(u, v) \in E$
    and some superadditive length function $\ell$ on $\mathcal{G}$.
\end{defn}

\begin{defn}[Geometric growth of generators]
\makelabel{defn:geom-growth:generators:sig:article-graph-raj-dahya}
    A family
        $\{A(u, v)\}_{(u,v) \in E} \subseteq \BoundedOps{\BanachRaum}$
    of bounded generators
    has \highlightTerm{geometric growth}
    if
        $\norm{A(u, v)} \leq \ell(u, v)$
    for all edges $(u, v) \in E$
    and some superadditive length function $\ell$ on $\mathcal{G}$.
\end{defn}

%% ********** END OF FILE: body/sec-1-introduction/sec-1-definitions.tex **********

%% ********************************************************************************
%% FILE: body/sec-1-introduction/sec-2-notation.tex
%% ********************************************************************************

\subsection[General notation]{General notation}
\label{sec:intro:notation:sig:article-graph-raj-dahya}

\firstparagraph
Throughout this paper we use the following notation

\begin{itemize}
    \item
        $\naturals = \{1,2,\ldots\}$,
        $\naturalsZero = \{0,1,2,\ldots\}$,
        $\realsNonNeg = \{r\in\reals \mid r\geq 0\}$,
        and
        to distinguish from indices $i$
        we use $\iunit$ for the imaginary unit $\sqrt{-1}$.

    \item
        For arbitrary groups $G$ or monoids $M$,
        we let $1$ denote the neutral element.
        In particular,
        for monoids consisting of words over an alphabet,
        we also use $1$ to denote the empty word,
        since this is the neutral element
        \wrt the monoidal operation of concatenation.

    \item
        We use
            $\HilbertRaum$ and $H$
        to denote Hilbert spaces,
            $\BanachRaum$
        for Banach spaces,
        and
            $\CStarAlg$
        for \TextCStarAlgs.

    \item
        For any Hilbert or Banach space
        $\onematrix$ shall denote the identity operator.
        For \TextCStarAlgs, $\id$ shall denote the identity map
        in order to avoid confusion with the element
            $\onematrix$,
        in the case of concretely represented unital \TextCStarAlgs.
        In ambivalent circumstances we use subscripts to denote the space
        on which an identity operator lives.

    \item
        Within the matrix algebra $M_{n}(\complex)$ for $n\in\naturals$,
        we let $\ElementaryMatrix{i}{j}$
        denote the $(i,j)$-th elementary operator on $\complex^{n}$
        defined \wrt the standard basis
        $\{\BaseVector{i}\}_{i=1}^{n}$.

    \item
        Given a linear map ${\Phi : \CStarAlg_{1} \to \CStarAlg_{2}}$
        between \TextCStarAlgs
        and $n\in\naturals$,
        the map
            ${\Phi \otimes \id_{n} : \CStarAlg_{1} \otimes M_{n}(\complex) \to \CStarAlg_{2} \otimes M_{n}(\complex)}$
        is defined via
            $
                (\Phi \otimes \id_{n})(\sum_{ij} a_{ij} \otimes \ElementaryMatrix{i}{j})
                = \sum_{ij}\Phi(a_{ij}) \otimes \ElementaryMatrix{i}{j}
            $
        for $\{a_{ij}\}_{i,j=1}^{n} \subseteq \CStarAlg$.

    \item
        The map $\Phi$ is called
        \highlightTerm{unital}
        if $\Phi(1) = 1$ (assuming $\CStarAlg_{1}$, $\CStarAlg_{2}$ are unital \TextCStarAlgs);
        \highlightTerm{self-adjoint}
        if $\Phi(a)$ is self-adjoint for all self-adjoint elements $a \in \CStarAlg_{1}$;
        \highlightTerm{positive}
        if $\Phi(a)$ is positive for all positive elements $a \in \CStarAlg_{1}$;
        \highlightTerm{$n$-positive}
        if $\Phi \otimes \id_{n}$ is positive for a given $n \in \naturals$;
        \highlightTerm{completely positive}
        if $\Phi$ is $n$-positive for each $n \in \naturals$;
        and a \highlightTerm{Schwarz} map,
        if it satisfies the \highlightTerm{Schwarz-inequality}
        $
            \Phi(a^{\ast}a)
            \geq \Phi(a)^{\ast}\Phi(a)
        $
        for all $a \in \CStarAlg_{1}$.
        Note that
            completely positive
            $\Rightarrow$
            $2$-positive
            $\Rightarrow$
            Schwarz
            $\Rightarrow$
            positive
            $\Rightarrow$
            self-adjoint
        (see \exempli \cite[Chapter~1 and Corollary~1.3.2]{Stoermer2013BookPosOps}).
        But the reverse implications fail in general.

    \item
        Given a Hilbert space $\HilbertRaum$,
        $L^{1}(\HilbertRaum) \subseteq \BoundedOps{\HilbertRaum}$
        denotes the \highlightTerm{trace class} operators,
        \idest the class of operators
        $T$ for which $\tr(\abs{T}) < \infty$
        (\cf
            \cite[\S{}2.4]{Murphy1990},
            \cite[\S{}3.4]{Pedersen1989analysisBook}%
        ).

    \item
        The above properties of
            positivity,
            $n$-positivity,
            and
            completely positivity
        are similarly defined for linear maps
            ${\Phi : L^{1}(H_{1}) \to L^{1}(H_{2})}$
        between spaces of trace class operators.
        We say that $\Phi$ is a
        \highlightTerm{completely positive trace-preserving (CPTP) map},
        if it is completely positive
        and $\tr(\Phi(s)) = \tr(s)$
        for all $s \in L^{1}(H_{1})$.
        (In \S{}\ref{sec:dilation:kraus:examples:sig:article-graph-raj-dahya}
        some examples shall be considered.)

    \item
        For an element $u$ of a \TextCStarAlg $\CStarAlg$,
        the \highlightTerm{adjoint}
        is defined by
            $\adjoint_{u}(a) = u\:a\:u^{\ast}$
        for $a \in \CStarAlg$.
        For a linear map $\Psi$ on $\CStarAlg$,
        the \highlightTerm{dissipation map}
        is defined by
            $D_{\Psi}(a, b) = \Psi(b^{\ast}a) - (\Psi(b)^{\ast}a + b^{\ast}\Psi(a))$
        for $a, b \in \CStarAlg$.
        The \highlightTerm{commutator} $[\cdot,\:\cdot]$
        and \highlightTerm{anti-commutator} $\{\cdot,\:\cdot\}$
        on any ring $R$
        are defined by
            $[x,\:y] = xy - yx$
            and
            $\{x,\:y\} = xy + yx$
        for $x, y \in R$.

    \item
        In \S{}\ref{sec:dilation:kraus:sig:article-graph-raj-dahya}
        and \S{}\ref{sec:dilation:vom-ende:sig:article-graph-raj-dahya},
        it shall be convenient to work with the notation
            $\ketbra{\xi}{\eta}$,
        which,
            for vectors $\xi$, $\eta$ in a Hilbert space $\HilbertRaum$,
        denotes
        the rank\=/$1$ operator
        defined by
            ${
                \HilbertRaum \ni x
                \mapsto
                \brkt{x}{\eta} \: \xi \in \linspann\{\xi\}
            }$.
        In particular one has
            $
                \tr(\ketbra{\xi}{\eta})
                = \brkt{\xi}{\eta}
            $.

    \item
        We use $T\:\xi$
        to denote the action of a linear operator $T$ on vectors $\xi$ of a Banach space.
        For \TextCStarAlgs,
        it is more standard to use $\Phi(a)$
        to denote the action of a linear operator $\Phi$ on elements $a$ of the \TextCStarAlg,
        even though this is itself a Banach space.
        This allows us to write $\Phi(a)\:\xi$,
        to denote the action of the \TextCStarAlg element $\Phi(a)$
        on the vector $\xi$ of a Hilbert space.
        As such, we denote families of Banach spaces \resp \TextCStarAlg operators
        parameterised say by a group $G$
        as
            $\{\phi(g)\}_{g \in G}$
            \resp
            $\{\Phi_{g}(\cdot)\}_{g \in G}$
        and their actions on elements of the underlying spaces
        as $\phi(g)\:\xi$ \resp $\Phi_{g}(a)$.
\end{itemize}

%% ********** END OF FILE: body/sec-1-introduction/sec-2-notation.tex **********

%% ********************************************************************************
%% FILE: body/sec-1-introduction/sec-3-results.tex
%% ********************************************************************************

\subsection[Statement of results]{Statement of results}
\label{sec:intro:aims:sig:article-graph-raj-dahya}

\firstparagraph
Our first result provides discrete dilations for possibly indivisible systems.
The terminology of
dissipative operators and partial traces
are presented in
    \S{}\ref{sec:examples:dissipative:sig:article-graph-raj-dahya},
    and
    \S{}\ref{sec:dilation:kraus:examples:sig:article-graph-raj-dahya}
below.

%% ********************************************************************************
%% FILE: body/sec-1-introduction/sec-3-para-1-results.tex
%% ********************************************************************************

\begin{highlightboxWithBreaks}
\begin{thm}[Discrete dilations of indivisible systems]
\makelabel{thm:result:graph-dilations:discrete:sig:article-graph-raj-dahya}
    Let $\mathcal{G} = (\Omega, E)$ be an arbitrary graph
    and $\{\phi(u, v)\}_{(u, v) \in E}$ be a family of contractions
    on a Banach space $\BanachRaum$.
    Suppose that $\phi$
    satisfies the identity axiom \eqcref{ax:dyn:id:sig:article-graph-raj-dahya}.%
    \footnoteref{ft:1:\beweislabel}
    Then the following hold:

    \begin{enumerate}[
        label={\bfseries (\alph*)},
        ref={\alph*},
        left=\rtab,
    ]
        \item\punktlabel{banach}
            There exists
                a Banach space $\tilde{\BanachRaum}$,
                a divisible dynamical system
                    $\{U(u,v)\}_{(u,v) \in E}$
                on the graph $\mathcal{G}$
                consisting of surjective isometries on $\tilde{\BanachRaum}$,
            as well as
                a surjective contraction
                    ${j : \tilde{\BanachRaum} \to \BanachRaum}$
                and
                a linear isometry
                    ${r : \BanachRaum \to \tilde{\BanachRaum}}$
                satisfying $j \circ r = \onematrix$,
            such that
                $j\:U(u, v)\:r = \phi(u, v)$
            for all $(u, v) \in E$.

        \item\punktlabel{cstar}
            Suppose $\BanachRaum$ is a (commutative) unital \TextCStarAlg $\CStarAlg$
            and $\phi = \{\Phi_{(u, v)}\}_{(u, v) \in E}$
            is a family of positive unital linear operators on $\CStarAlg$.
            Then there exist
                a (commutative) unital \TextCStarAlg $\tilde{\CStarAlg}$,
                a divisible dynamical system
                    $\{U(u,v)\}_{(u,v) \in E}$
                on the graph $\mathcal{G}$
                consisting of {}\textsuperscript{\ensuremath{\ast}}\=/automorphisms on $\tilde{\CStarAlg}$,
            as well as
                a surjective unital {}\textsuperscript{\ensuremath{\ast}}\=/homomorphism
                    ${j : \tilde{\CStarAlg} \to \CStarAlg}$
                and an isometric positive unital linear map
                    ${r : \CStarAlg \to \tilde{\CStarAlg}}$
                satisfying $j \circ r = \id_{\CStarAlg}$
            such that
                $j\:U(u, v)\:r = \Phi_{(u, v)}$
            for all $(u, v) \in E$.

        \item\punktlabel{cptp}
            Suppose $\BanachRaum$ is the space $L^{1}(\HilbertRaum)$
            of trace class operators on a Hilbert space $\HilbertRaum$,
            and $\phi = \{\Phi_{(u, v)}\}_{(u, v) \in E}$
            is a family of CPTP\=/maps on $L^{1}(\HilbertRaum)$.
            Then there exist
                an auxiliary Hilbert space $\tilde{\HilbertRaum}$,
                a divisible dynamical system
                    $\{U(u,v)\}_{(u,v) \in E}$
                on the graph $\mathcal{G}$
                consisting of unitaries on $\HilbertRaum \otimes \tilde{\HilbertRaum}$,
            as well as
                a pure state $\omega \in L^{1}(\tilde{\HilbertRaum})$
            such that
                $\Phi_{(u,v)}(s) = \tr_{2}(\adjoint_{U(u, v)}(s \otimes \omega))$
            for all $(u, v) \in E$
            and $s \in L^{1}(\HilbertRaum)$.
    \end{enumerate}

    \nvraum{1}

\end{thm}
\end{highlightboxWithBreaks}

\footnotetext[ft:1:\beweislabel]{%
    but not necessarily the divisibility axiom \eqcref{ax:dyn:div:sig:article-graph-raj-dahya}
}

To obtain continuous counterparts,
we restrict our attention to linearly ordered graphs
endowed with the natural order topology
(\cf \S{}\ref{sec:examples:hamiltonian:sig:article-graph-raj-dahya}),
and make use of special conditions.

\begin{highlightboxWithBreaks}
\begin{thm}[Continuous dilations of divisible systems]
\makelabel{thm:result:graph-dilations:cts:divisible:sig:article-graph-raj-dahya}
    Let $\mathcal{G} = (\Omega, E)$
    be a graph where $E$ is a reflexive linear ordering.
    Endow $\Omega$ with the order topology
    and $E \subseteq \Omega \times \Omega$
    with the relative topology of the product topology.
    Let
        $\{\phi(u, v)\}_{(u, v) \in E}$
        be a divisible dynamical system on the graph $\mathcal{G}$
        consisting of contractions
        on a Banach space $\BanachRaum$.
    Suppose further that $\phi$ has geometric growth.
    Then the following hold:

    \begin{enumerate}[
        label={\bfseries (\alph*)},
        ref={\alph*},
        left=\rtab,
    ]
        \item\punktlabel{banach}
            There exists
                a Banach space $\tilde{\BanachRaum}$,
                a \uline{strongly continuous} divisible dynamical system
                    $\{U(u,v)\}_{(u,v) \in E}$
                on the graph $\mathcal{G}$
                consisting of surjective isometries on $\tilde{\BanachRaum}$,
            as well as
                a surjective contraction
                    ${j : \tilde{\BanachRaum} \to \BanachRaum}$
                and
                a linear isometry
                    ${r : \BanachRaum \to \tilde{\BanachRaum}}$
                satisfying $j \circ r = \onematrix$,
            such that
                $j\:U(u, v)\:r = \phi(u, v)$
            for all $(u, v) \in E$.

        \item\punktlabel{cstar}
            Suppose $\BanachRaum$ is a (commutative) unital \TextCStarAlg $\CStarAlg$
            and $\phi = \{\Phi_{(u, v)}\}_{(u, v) \in E}$
            is a family of positive unital linear operators on $\CStarAlg$.
            Then there exist
                a (commutative) unital \TextCStarAlg $\tilde{\CStarAlg}$,
                a \uline{strongly continuous} divisible dynamical system
                    $\{U(u,v)\}_{(u,v) \in E}$
                on the graph $\mathcal{G}$
                consisting of {}\textsuperscript{\ensuremath{\ast}}\=/automorphisms on $\tilde{\CStarAlg}$,
            as well as
                a surjective unital {}\textsuperscript{\ensuremath{\ast}}\=/homomorphism
                    ${j : \tilde{\CStarAlg} \to \CStarAlg}$
                and an isometric positive unital linear map
                    ${r : \CStarAlg \to \tilde{\CStarAlg}}$
                satisfying $j \circ r = \id_{\CStarAlg}$
            such that
                $j\:U(u, v)\:r = \Phi_{(u, v)}$
            for all $(u, v) \in E$.
    \end{enumerate}

    \nvraum{1}

\end{thm}
\end{highlightboxWithBreaks}

\begin{highlightboxWithBreaks}
\begin{thm}[Continuous dilations of indivisible systems]
\makelabel{thm:result:graph-dilations:cts:indivisible:sig:article-graph-raj-dahya}
    Let $\mathcal{G}$ be as in \Cref{thm:result:graph-dilations:cts:divisible:sig:article-graph-raj-dahya}
    and consider a family
        $\phi = \{e^{A(u,v)}\}_{(u,v) \in E} \subseteq \BoundedOps{\BanachRaum}$
    of contractions on a Banach space $\BanachRaum$,
    where $\{A(u, v)\}_{(u,v) \in E} \subseteq \BoundedOps{\BanachRaum}$
        is a norm-continuous additive family of
        (not necessarily commuting)
        bounded operators on $\BanachRaum$
        with geometric growth.
    Then the following hold:

    \begin{enumerate}[
        label={\bfseries (\alph*)},
        ref={\alph*},
        left=\rtab,
    ]
        \item\punktlabel{banach}
            If each $A(u, v)$ is dissipative
            and $\{A(u, v)\}_{(u,v) \in E}$ has geometric growth,
            then the claim in
            \Cref{thm:result:graph-dilations:cts:divisible:sig:article-graph-raj-dahya}~%
            \eqcref{it:banach:thm:result:graph-dilations:cts:divisible:sig:article-graph-raj-dahya}
            holds.

        \item\punktlabel{cstar}
            Suppose $\BanachRaum$ is a (commutative) unital \TextCStarAlg $\CStarAlg$
            and that each $A(u, v) = L_{(u, v)}$,
            where
                $L_{(u, v)}$ is a self-adjoint map
            satisfying
                $L_{(u,v)}(1) = \zeromatrix$
            and
                $D_{L_{(u, v)}}(a, a) \geq \zeromatrix$
            for all $a \in \CStarAlg$.
            If $\{L_{(u, v)}\}_{(u,v) \in E}$ has geometric growth,
            then the conclusion of the claim in
            \Cref{thm:result:graph-dilations:cts:divisible:sig:article-graph-raj-dahya}~%
            \eqcref{it:cstar:thm:result:graph-dilations:cts:divisible:sig:article-graph-raj-dahya}
            holds.
    \end{enumerate}

    \nvraum{1}
\end{thm}
\end{highlightboxWithBreaks}

%% ********** END OF FILE: body/sec-1-introduction/sec-3-para-1-results.tex **********

%% ********************************************************************************
%% FILE: body/sec-1-introduction/sec-3-para-2-remarks.tex
%% ********************************************************************************

\begin{rem}[$1$\=/parameter factorisations]
\makelabel{rem:1-parameter-decomp:sig:article-graph-raj-dahya}
    One immediate advantage of such dilations
    is the following decomposition.
    Consider a graph $\mathcal{G} = (\Omega, E)$
    for which $E$ is a reflexive linear ordering.
    Suppose
        $\{U(t,s)\}_{(t,s) \in E}$
    is a family of (strongly continuous) surjective isometries
    which dilates a family of contractions
        $\{\phi(t, s)\}_{(t, s) \in E}$.
    Fix any point $t_{0} \in \Omega$
    and define

        \begin{shorteqnarray}
            U(t)
                \coloneqq
                    \begin{cases}
                        U(t, t_{0})
                            &: &(t, t_{0}) \in E\\
                        U(t_{0}, t)^{-1}
                            &: &(t_{0}, t) \in E
                    \end{cases}
        \end{shorteqnarray}

    \continueparagraph
    for all $t \in \Omega$.
    Then $\{U(t)\}_{t \in \Omega}$
    is a (strongly continuous) family of surjective isometries.
    Following Howland
        \cite[Lemma~1]{Howland1974Article}
    and Evans
        \cite[Lemma~6.2]{Evans1976ArticlePerturb},
    one can then readily prove that

        \begin{shorteqnarray}
            U(t, s) = U(t)U(s)^{-1}
        \end{shorteqnarray}

    \continueparagraph
    holds for all $(t, s) \in E$.
    This identity entails that
        $\{U(t)\}_{t \in \Omega}$
    is unique upto right-multiplication via a surjective isometry.
    Finally, considering the case of $(\Omega, E) = (\interval, \geq)$,
    where $\interval \subseteq \realsNonNeg$ is a connected subset,
    we note that the family of surjective isometries
        $\{U(t)\}_{t \in \interval}$
    does \uline{not} in general satisfy the semigroup law,
    otherwise $\{U(t,s)\}_{(t,s) \in E}$
    and thus $\{\phi(t, s)\}_{(t,s) \in E}$
    would be memoryless processes.
\end{rem}

%% ********** END OF FILE: body/sec-1-introduction/sec-3-para-2-remarks.tex **********

%% ********************************************************************************
%% FILE: body/sec-1-introduction/sec-3-para-3-goals.tex
%% ********************************************************************************

In order to achieve our results,
we make use of known non-classical dilation theorems.
These in turn all rely on operator families parameterised by (topological) groups,
which poses the main challenge of this paper,
as our operator families live on (the edges of) graphs.
The roadmap of this paper can be summarised as follows:

\begin{enumerate}[
    label={\bfseries Goal~{\Roman*})},
    ref={Goal~{\Roman*}},
    left=0pt,
]
    \item\label{goal:embed:sig:article-graph-raj-dahya}
        Embed the set of graph edges $E$
        into an appropriately defined group $G = G_{\mathcal{G}}$,
        in a way that loops are mapped to the identity,
        and products of (the images of) edges
        is coherent with traversal in the graph.
        To this end, we apply the theory of
        \highlightTerm{reduction} and \highlightTerm{string-rewriting systems},
        as well as presented groups
        (see \S{}\ref{sec:algebra:sig:article-graph-raj-dahya}).

    \item\label{goal:ext:sig:article-graph-raj-dahya}
        Letting ${\iota : E \to G}$ denote the embedding
        from \ref{goal:embed:sig:article-graph-raj-dahya},
        construct \usesinglequotes{natural} extensions
        of $\{\phi(u,v)\}_{(u,v) \in E}$
        to operator families
            $\{\quer{\phi}(x)\}_{x \in G}$
        such that
            $
                \quer{\phi}(\iota(u,v))
                = \phi(u,v)
            $
        for all $(u,v) \in E$
        (see \S{}\ref{sec:ext:sig:article-graph-raj-dahya}).

    \item\label{goal:cts:sig:article-graph-raj-dahya}
        Determine sufficient conditions
        for certain continuity conditions of
            $\quer{\phi}$
        in the case that $\Omega$ is topologised
        (see \S{}\ref{sec:dilation:stroescu:cts:sig:article-graph-raj-dahya}).

    \item\label{goal:dilation:sig:article-graph-raj-dahya}
        Recall and extend appropriate dilation results for our context
        (see \S{}\ref{sec:dilation:sig:article-graph-raj-dahya}).
\end{enumerate}

\continueparagraph
These dilations
shall then be applied to the extensions
to obtain \Cref{%
    thm:result:graph-dilations:discrete:sig:article-graph-raj-dahya,%
    thm:result:graph-dilations:cts:divisible:sig:article-graph-raj-dahya,%
    thm:result:graph-dilations:cts:indivisible:sig:article-graph-raj-dahya%
}
(see \S{}\ref{sec:results:sig:article-graph-raj-dahya}).

%% ********** END OF FILE: body/sec-1-introduction/sec-3-para-3-goals.tex **********

%% ********** END OF FILE: body/sec-1-introduction/sec-3-results.tex **********

%% ********** END OF FILE: body/sec-1-introduction/.index.tex **********

%% ********************************************************************************
%% FILE: body/sec-2-examples/.index.tex
%% ********************************************************************************

\section[Examples]{Examples}
\label{sec:examples:sig:article-graph-raj-dahya}

\firstparagraph
Before proceeding with the above roadmap,
we recall some basic facts about
semigroups and their generators.
We then present classes of examples
arising from dynamical systems
to which the main results apply.
In particular consider natural ways
in which indivisibility arises,
both in the continuous setting of evolution families
as well as the discrete setting of networks.

%% ********************************************************************************
%% FILE: body/sec-2-examples/sec-1-dissipative.tex
%% ********************************************************************************

\subsection[Dissipativity of generators]{Dissipativity of generators}
\label{sec:examples:dissipative:sig:article-graph-raj-dahya}

\firstparagraph
In order to obtain operator families of contractions in our examples,
we first recall some basic notions from classical semigroup theory
and mathematical physics.
Consider a Banach space $\BanachRaum$
with dual $\BanachRaum^{\prime}$
and associated bilinear evaluation map
    ${\brkt{\cdot}{\cdot} : \BanachRaum \times \BanachRaum^{\prime} \to \complex}$.
A (not necessarily linear) isometric function
    ${J : \BanachRaum \to \BanachRaum^{\prime}}$
is called a \highlightTerm{duality section}
if
    $\brkt{\xi}{J(\xi)} = \norm{\xi}^{2}$
for all $\xi\in\BanachRaum$.
By the Hahn Banach theorem, duality sections always exist.
A densely defined linear operator $A$ on $\BanachRaum$
is called \highlightTerm{dissipative \wrt $J$}
if $\Re\brkt{A\xi}{J(\xi)} \leq 0$
for all $\xi \in \opDomain{A}$.%
\footnote{%
    \cf
    \cite[\S{}I.3]{Goldstein1985semigroups}.
}
By the Lumer--Phillips form of the Hille--Yosida theorem,%
\footnote{%
    see
    \cite[Theorem~I.3.3 and Corollary~I.3.5]{Goldstein1985semigroups}.
}
the following statements are equivalent:%
\footnote{%
    here $\opResolventSet{A}$ denotes the resolvent set
    $\{\lambda \in \complex \mid (\lambda - A)^{-1} \in \BoundedOps{\BanachRaum}\}$.
}

\begin{enumerate}[
    label={\bfseries{\arabic*}.},
    ref={\arabic*},
    left=\rtab,
]
    \item
        $A$ is densely defined
        with $\opResolventSet{A} \supseteq (0,\:\infty)$
        and is dissipative \wrt all duality sections;
    \item
        $A$ is densely defined
        with $\opResolventSet{A} \cap (0,\:\infty) \neq \emptyset$
        and is dissipative \wrt some duality section;
    \item
        $A$ is the generator of a contractive semigroup.
\end{enumerate}

In the case of Hilbert spaces, standard examples include
operators of the form $A = \iunit H$
for a self-adjoint operator $H$.%
\footnote{%
    (unbounded) dissipative operators thus generalise the role of Hamiltonians.
}
In particular, for a bounded operator $A \in \BoundedOps{\BanachRaum}$,
since $\lambda \in \opResolventSet{A}$ for all $\lambda \in \complex$ with $\abs{\lambda} > \norm{A}$,
one has that $e^{\alpha A}$ is a contraction for all $\alpha > 0$
if and only if $A$ is dissipative \wrt some duality section
if and only if $A$ is dissipative \wrt all duality sections.
A bounded operator $A$ is called \highlightTerm{dissipative}
if it is dissipative \wrt some (equivalently: any) duality section.%
\footnote{%
    this equivalence continues to hold in the unbounded case
    (see
        \cite[Corollary~I.3.5]{Goldstein1985semigroups}%
    ).
}
In particular, the class of bounded dissipative operators
is closed under positive linear combinations
and weak limits.

As we shall later make use of operators of the form $e^{X}$,
where $X$ is a bounded dissipative operator,
it shall be useful to determine upper bounds for perturbations.
The next result is well-known and arises
in the process of deriving the celebrated
Baker--Campbell--Hausdorff formula
(\cf
    \cite[\S{}2 and \S{}4]{Wilcox1967ArticleExp},
    \cite[Theorem~5.4 and (5.15)]{Hall2015Book}%
).
As it is typically proved in the literature
for either Hilbert spaces or finite dimensional spaces,
we present an elementary proof without any such restrictions
for the reader's convenience.

\begin{prop}[Derivative of the operator exponential]
\makelabel{prop:bch-exp:sig:article-graph-raj-dahya}
    Let $X, Y \in \BoundedOps{\BanachRaum}$
    be bounded operators on an arbitrary Banach space $\BanachRaum$.
    Then the derivative of the analytic function
        ${f : \reals \ni t \mapsto e^{X + t\:Y} \in \BoundedOps{\BanachRaum}}$
    is given by

        \begin{restoremargins}
        \begin{equation}
        \label{eq:bch-exp:sig:article-graph-raj-dahya}
            f^{\prime}(t)
            = \int_{s \in [0,\:1]}
                    e^{(1 - s)\:(X + t\:Y)}
                    \:Y
                    \:e^{s\:(X + t\:Y)}
                \:\dee s
        \end{equation}
        \end{restoremargins}

    \continueparagraph
    for all $t \in \reals$,
    where the integrals are computed strongly via Bochner-integrals.
\end{prop}

    \begin{proof}
        Let $t \in \reals$ be arbitrary
        and set $Z_{t} \coloneqq X + t\:Y$.
        First observe that by analyticity of
        ${\reals \ni s \mapsto e^{(1-s)\:Z_{t}}\:Y\:e^{s\:Z_{t}} \in \BoundedOps{\BanachRaum}}$,
        the integral in \eqcref{eq:bch-exp:sig:article-graph-raj-dahya}
        exists and can be computed via the strong limit of
            $
                I_{n}(t)
                \coloneqq
                \sum_{k=1}^{n}
                \tfrac{1}{n}
                e^{(1 - \frac{k}{n})\:Z_{t}}
                \:
                Y
                \:
                e^{\frac{k}{n}\:Z_{t}}
            $
        as ${n \longrightarrow \infty}$.
        For $n \in \naturals$ the following trick can be used

            \begin{shorteqnarray}
                \tfrac{\dee}{\dee t}e^{Z_{t}}
                = \tfrac{\dee}{\dee t}
                        (e^{\frac{1}{n}\:Z_{t}})^{n}
                =
                    \sum_{k=1}^{n}
                        (e^{\frac{1}{n}\:Z_{t}})^{n - k}
                        \:
                        (
                            \tfrac{\dee}{\dee t}
                            e^{\frac{1}{n}\:Z_{t}}
                        )
                        \:
                        (e^{\frac{1}{n}\:Z_{t}})^{k - 1},
            \end{shorteqnarray}

        \continueparagraph
        which allows us to compute

            \begin{restoremargins}
            \begin{equation}
            \label{eq:1:\beweislabel}
            \everymath={\displaystyle}
            \begin{array}[m]{rcl}
                \normLong{
                    \tfrac{\dee}{\dee t}e^{Z_{t}}
                    -
                    I_{n}(t)
                }
                &\leq
                    &\tfrac{1}{n}
                    \sum_{k=1}^{n}
                        \norm{e^{\frac{n - k}{n}\:Z_{t}}}
                        \:\norm{
                            n
                            \cdot
                            \tfrac{\dee}{\dee t}
                            e^{\frac{1}{n}\:Z_{t}}
                            -
                            Y
                        }
                        \:\norm{e^{\frac{k - 1}{n}\:Z_{t}}}
                    \\
                    &&+
                    \tfrac{1}{n}
                    \sum_{k=1}^{n}
                        \norm{e^{(1 - \frac{k}{n})\:Z_{t}}}
                        \:\norm{Y}
                        \:\underbrace{
                            \norm{
                                e^{\frac{k}{n}\:Z_{t}}
                                -
                                e^{\frac{k-1}{n}\:Z_{t}}
                            }
                        }_{
                            =
                            \norm{e^{\frac{k-1}{n}\:Z_{t}}}
                            \norm{
                                e^{\frac{1}{n}\:Z_{t}}
                                -
                                \onematrix
                            }
                        }
                    \\
                &\leq
                    &e^{\norm{Z_{t}}}
                    \Big(
                        \norm{
                            n
                            \cdot
                            \tfrac{\dee}{\dee t}
                            e^{\frac{1}{n}\:Z_{t}}
                            -
                            Y
                        }
                        +
                        \norm{Y}
                        \norm{
                            e^{\frac{1}{n}\:Z_{t}}
                            -
                            \onematrix
                        }
                    \Big).
            \end{array}
            \end{equation}
            \end{restoremargins}

        \continueparagraph
        To estimate the penultimate term,
        the analyticity of
            ${\reals \ni t \mapsto e^{\frac{1}{n}\:Z_{t}} \in \BoundedOps{\BanachRaum}}$,
        allows us to compute

            \begin{restoremargins}
            \begin{equation}
            \label{eq:2:\beweislabel}
            \everymath={\displaystyle}
            \begin{array}[m]{rcl}
                n \cdot \tfrac{\dee}{\dee t}e^{\frac{1}{n}\:Z_{t}}
                - Y
                &=
                    &\Big(
                        \sum_{j=0}^{\infty}
                        \tfrac{1}{j!n^{j - 1}}
                        \:
                        \tfrac{\dee}{\dee t}
                        Z_{t}^{j}
                    \Big)
                    - Y
                    \\
                &=
                    &\Big(
                        \sum_{j=1}^{\infty}
                        \tfrac{1}{j!n^{j - 1}}
                        \sum_{(k,l) \in \naturalsZero^{2}}
                            \delta_{k + l, j - 1}
                            \:Z_{t}^{k}
                            \:(\tfrac{\dee}{\dee t}\:Z_{t})
                            \:Z_{t}^{l}
                    \Big)
                    - Y
                    \\
                &=
                    &\sum_{(k,l) \in \naturalsZero^{2} \without \{(0,0)\}}
                        \tfrac{1}{(k + l + 1)!n^{k + l}}
                        Z_{t}^{k}
                        \:Y
                        \:Z_{t}^{l}
            \end{array}
            \end{equation}
            \end{restoremargins}

        \continueparagraph
        and thus

            \begin{shorteqnarray}
                \norm{
                    n \cdot \tfrac{\dee}{\dee t}e^{\frac{1}{n}\:Z_{t}}
                    -
                    Y
                }
                &\leq
                    &\sum_{(k,l) \in \naturalsZero^{2} \without \{(0,0)\}}
                        \tfrac{1}{(k + l + 1)!n^{k + l}}
                            \norm{Z_{t}}^{k}
                            \norm{Y}
                            \norm{Z_{t}}^{l}
                    \\
                &\leq
                    &\sum_{(k,l) \in \naturalsZero^{2} \without \{(0,0)\}}
                        \tfrac{1}{(k + l + 1)!n^{k + l}}
                            (\norm{X} + t\norm{Y})^{k}
                            \norm{Y}
                            (\norm{X} + t\norm{Y})^{l}
                    \\
                &\overset{(\ast)}{=}
                    &n \cdot \tfrac{\dee}{\dee t} e^{\frac{1}{n}(\norm{X}+ t\norm{Y})} - \norm{Y}
                \\
                &=
                    &\norm{Y} e^{\frac{1}{n}(\norm{X}+ t\norm{Y})} - \norm{Y},
            \end{shorteqnarray}

        \continueparagraph
        where ($\ast$) can be derived analogously to \eqcref{eq:2:\beweislabel}.
        The expression in \eqcref{eq:1:\beweislabel} thus simplifies to

            \begin{shorteqnarray}
                \normLong{
                    \tfrac{\dee}{\dee t}e^{Z_{t}}
                    -
                    I_{n}(t)
                }
                \leq
                e^{\norm{Z_{t}}}
                \norm{Y} (
                    e^{\frac{1}{n}(\norm{X}+ t\norm{Y})} - 1
                    +
                    \norm{e^{\frac{1}{n}Z_{t}} - \onematrix}
                ),
            \end{shorteqnarray}

        \continueparagraph
        which converges to $0$
        as ${n \longrightarrow \infty}$.
        Since $\{I_{n}(t)\}_{n \in \naturals}$ converges strongly
        to the integral in \eqcref{eq:bch-exp:sig:article-graph-raj-dahya},
        this completes the proof.
    \end{proof}

\begin{prop}[Perturbations of the operator exponential]
\makelabel{prop:bch-exp-perturbation:sig:article-graph-raj-dahya}
    Let $\BanachRaum$ be an arbitrary Banach space.
    Then

        \begin{restoremargins}
        \begin{equation}
        \label{eq:bch-exp-perturbation:sig:article-graph-raj-dahya}
            \norm{e^{X + Y} - e^{X}}
            \leq \norm{Y}
        \end{equation}
        \end{restoremargins}

    \continueparagraph
    for all (not necessarily commuting)
    bounded dissipative operators $X, Y \in \BoundedOps{\BanachRaum}$.
\end{prop}

    \begin{proof}
        This follows from
        the fundamental theorem for Banach space derivatives,
        \Cref{prop:bch-exp:sig:article-graph-raj-dahya},
        and
        the dissipativity of positive linear combinations of dissipative operators,
        since

            \begin{shorteqnarray}
                \norm{e^{X + Y} - e^{X}}
                &= &\normLong{
                        \int_{t \in [0,\:1]}
                            \tfrac{\dee}{\dee t}
                            e^{X + t\:Y}
                        \:\dee t
                    }
                \\
                &\eqcrefoverset{eq:bch-exp:sig:article-graph-raj-dahya}{=}
                    &\normLong{
                        \int_{t \in [0,\:1]}
                        \int_{s \in [0,\:1]}
                            e^{(1 - s)\:(X + t\:Y)}
                            \:Y
                            \:e^{s\:(X + t\:Y)}
                        \:\dee s
                        \:\dee t
                    }
                \\
                &\leq &\sup_{s, t \in [0,\:1]}
                    \underbrace{
                        \norm{e^{(1 - s)\:(X + t\:Y)}}
                    }_{\leq 1}
                    \norm{Y}
                    \underbrace{
                        \norm{e^{s\:(X + t\:Y)}}
                    }_{\leq 1}
                \leq \norm{Y}.
            \end{shorteqnarray}

    \end{proof}

%% ********** END OF FILE: body/sec-2-examples/sec-1-dissipative.tex **********

%% ********************************************************************************
%% FILE: body/sec-2-examples/sec-2-linblad.tex
%% ********************************************************************************

\subsection[Linblad-correspondences]{Linblad-correspondences}
\label{sec:examples:linblad:sig:article-graph-raj-dahya}

\firstparagraph
In the narrower context of a unital \TextCStarAlg $\CStarAlg$,
as a simple example of dissipative generators
one considers operators of the form
$\iunit [h,\:\cdot]$,
where
$h \in \CStarAlg$
is a self-adjoint element,
noting that
$
    e^{\iunit t[h,\:\cdot]}
    = \adjoint_{e^{\iunit t h}}
$
for $t \in \reals$,%
\footnote{%
    Let $a \in \CStarAlg$ be arbitrary.
    One can easily observe by induction that
    $
        [a,\:\cdot]^{n}(x)
        = \sum_{k = 0}^{n}
            \binom{n}{k}
            a^{k}\:x\:(-a)^{n-k}
    $
    for all $x \in \CStarAlg$, $n\in\naturalsZero$.
    From this one obtains the Lie algebra correspondence
    $e^{t[a,\:\cdot]} = \adjoint_{e^{t a}}$
    for all $t \in \reals$.
}
which is a norm-continuous semigroup
consisting of contractions (in fact surjective isometries).
More generally, the following holds:

\begin{thm}[Linblad, 1976]
\makelabel{thm:linblad:ucp:sig:article-graph-raj-dahya}
    Let $\CStarAlg$ be a unital \TextCStarAlg.
    Let $L \in \BoundedOps{\CStarAlg}$
    be the bounded generator
    of a norm-continuous semigroup
    $\{\Phi_{t}\}_{t \geq \realsNonNeg}$.
    Consider the following statements:

    \begin{enumerate}[
        label={(\bfseries{\alph*})},
        ref={\alph*},
        left=\rtab,
    ]
        \item\punktlabel{1}
            Each $\Phi_{t}$ is a unital completely positive map.

        \item\label{repr:linblad:ucp:sig:article-graph-raj-dahya}
            $L = \iunit[h,\:\cdot] + \Psi - \frac{1}{2}\{\Psi(1),\cdot\}$
            for some self-adjoint element $h \in \CStarAlg$
            and some ultra-weakly continuous
            completely positive map
            $\Psi \in \BoundedOps{\CStarAlg}$.
    \end{enumerate}

    \continueparagraph
    If $\CStarAlg$ is a hyperfinite von~Neumann factor over a separable Hilbert space,%
    \footnoteref{ft:1:\beweislabel}
    then \punktcref{1} $\Leftrightarrow$ \eqcref{repr:linblad:ucp:sig:article-graph-raj-dahya}.
    And for unital \TextCStarAlgs in general,
        \eqcref{repr:linblad:ucp:sig:article-graph-raj-dahya} $\Rightarrow$ \punktcref{1}.
\end{thm}

\footnotetext[ft:1:\beweislabel]{%
    see \cite[Definition~XIV.1.4]{Takesaki2003BookIII},
    \cite[\S{}II.3]{Takesaki2002BookI}.
}

For a proof, see
\cite[Corollary~1, Proposition~5, and Theorem~3]{Lindblad1976Article}.
Note that by considering strong derivatives, $L(1)=\zeromatrix$
is a necessary (and sufficient)
requirement for each $\Phi_{t}$ to be unital.
As an important step in the proof
of \Cref{thm:linblad:ucp:sig:article-graph-raj-dahya},
Linblad established the following correspondence
(see
    \cite[\S{}3 and Proposition 4]{Lindblad1976Article}
for a proof).

\begin{lemm}[Schwarz-correspondence]
\makelabel{lemm:linblad:schwarz:sig:article-graph-raj-dahya}
    Let $\CStarAlg$ be a unital \TextCStarAlg.
    Let $L \in \BoundedOps{\CStarAlg}$
    be the bounded generator
    of a norm-continuous semigroup
    $\{\Phi_{t}\}_{t \geq \realsNonNeg}$.
    Then \tfae:

    \begin{enumerate}[
        label={\bfseries (\alph*)},
        ref={\alph*},
        left=\rtab,
    ]
        \item\punktlabel{1}
            Each $\Phi_{t}$ is a unital map
            satisfying the Schwarz-inequality.

        \item\label{repr:linblad:schwarz:sig:article-graph-raj-dahya}
            $L$ is self-adjoint
            with
            $L(1) = \zeromatrix$
            and
            $D_{L}(a, a) \geq \zeromatrix$
            for all $a \in \CStarAlg$.%
            \footnoteref{ft:1:\beweislabel}
    \end{enumerate}

    \nvraum{1}

\end{lemm}

\footnotetext[ft:1:\beweislabel]{%
    recall that the dissipation map
    associated to $L$
    is given by
    $D_{L}(a, b) = L(b^{\ast}a) - (L(b)^{\ast}a + b^{\ast}L(a))$
    for $a, b \in \CStarAlg$.
}

\begin{rem}[Physical interpretation]
\makelabel{rem:physical-interpretation-linblad:sig:article-graph-raj-dahya}
    In
        \Cref{thm:linblad:ucp:sig:article-graph-raj-dahya}~%
        \eqcref{repr:linblad:ucp:sig:article-graph-raj-dahya},
    the $\iunit [h,\:\cdot]$ part of $L$ may be referred to as the \highlightTerm{Hamiltonian part}
    and the remainder $\Psi - \frac{1}{2}\{\Psi(1),\:\cdot\}$
    as the \highlightTerm{purely dissipative part}.
    Whilst this decomposition is not unique
    (\cf
        \cite[\S{}5]{Lindblad1976Article}%
    ),
    it nonetheless admits an interesting physical interpretation.
    If $L$ consists entirely of a Hamiltonian part
    $\iunit [h,\:\cdot]$,
    then by the discussion at the beginning of this subsection,
    the semigroup takes the form
        $
            \Phi
            = \{e^{\iunit t[h,\:\cdot]}\}_{t \in \realsNonNeg}
            = \{\adjoint_{e^{\iunit t h}}\}_{t \in \realsNonNeg}
        $,
    and the converse clearly holds.
    However, if the system is not governed by unitary evolution,
    then it is understood to be subject to interference from the \usesinglequotes{environment},%
    \footnote{%
        \idest, the system, which can always be being initially
        separably coupled to a surrounding environment,
        becomes entangled with the state of the environment
        (\cf \Cref{rem:physical-interpretation-cptp:sig:article-graph-raj-dahya}).
    }
    and the dissipative part witnesses this.
    In this case, the operators in $\Phi$
    are referred to in the literature as
    \highlightTerm{noisy quantum channels}
    (under the \usesinglequotes{Heisenberg picture}).
\end{rem}

\begin{rem}
\makelabel{rem:class-of-linblad-schwarz:sig:article-graph-raj-dahya}
    Observe that since
    the maps
        ${L \mapsto L(1)}$
    and
        $
            L
            \mapsto
            D_{L}(a, a)
            = L(a^{\ast}a) - a^{\ast}L(a) - L(a)^{\ast}a
            = L(a^{\ast}a) - a^{\ast}L(a) - L(a^{\ast})a
        $
    are linear in $L$,
    the class of self-adjoint generators
    satisfying \eqcref{repr:linblad:schwarz:sig:article-graph-raj-dahya}
    is clearly closed under positive linear combinations.
\end{rem}

Since by the Russo--Dye theorem,
unital positive (\exempli completely positive or Schwarz)
maps are necessarily contractions
(see
    \cite[Corollary~1]{RussoDye1966Article},
    \cite[Corollary~2.9]{Paulsen2002book},
    \cite[Theorem~1.3.3]{Stoermer2013BookPosOps}%
),
by the Linblad theorem and the Schwarz-correspondence
one immediately obtains:

\begin{cor}[Dissipativity of Linbladian generators]
\makelabel{cor:linblad:dissipative:sig:article-graph-raj-dahya}
    Let $\CStarAlg$ be a unital \TextCStarAlg.
    Then
        $\iunit[h,\:\cdot] + \Psi - \frac{1}{2}\{\Psi(1),\cdot\}$
    is a dissipative operator
    for all self-adjoint elements $h \in \CStarAlg$
    and all completely positive (not necessarily unital)
    maps $\Psi \in \BoundedOps{\CStarAlg}$.
    And more generally,
    each $L \in \BoundedOps{\CStarAlg}$ is dissipative
    for which
        $L(1) = \zeromatrix$
        and
        $D_{L}(a, a) \geq \zeromatrix$ for all $a \in \CStarAlg$.
\end{cor}

%% ********** END OF FILE: body/sec-2-examples/sec-2-linblad.tex **********

%% ********************************************************************************
%% FILE: body/sec-2-examples/sec-3-lo.tex
%% ********************************************************************************

\subsection[Dynamical systems on linear orderings]{Dynamical systems on linear orderings}
\label{sec:examples:lo:sig:article-graph-raj-dahya}

\firstparagraph
We now observe some basic properties of dynamical systems
defined on linearly ordered graphs.
To this end, let $\mathcal{G} = (\Omega, E)$
be a graph for which $E$ is a reflexive linear ordering.
For convenience we let $\mathrel{\prec}$
denote the irreflexive part of $E$,
so that $E = \{(u, v) \in \Omega \mid u \preceq v\}$.
Endow $\Omega$ with the order topology,
\idest the topology generated by sets of the form
    $\{u \in \Omega \mid u \prec v\}$
    and
    $\{u \in \Omega \mid v \prec u\}$
for $v \in \Omega$,
and $E$ with the subspace topology of $\Omega \times \Omega$.
We first observe that geometric growth entails norm-continuity.

\begin{prop}[Norm-continuity under geometric growth]
\makelabel{prop:geom-implies-norm-cts:dyn:sig:article-graph-raj-dahya}
    Let $\{\phi(u,v)\}_{(u,v) \in E} \subseteq \BoundedOps{\BanachRaum}$
    be a uniformly bounded family of bounded operators on a Banach space $\BanachRaum$.
    If $\phi$ satisfies
    the divisibility axiom \eqcref{ax:dyn:div:sig:article-graph-raj-dahya}
    and has geometric growth,
    then $\phi$ necessarily satisfies
    the identity axiom \eqcref{ax:dyn:id:sig:article-graph-raj-dahya}
    and is norm-continuous.
\end{prop}

    \begin{proof}
        Due to geometric growth one has
            $
                \norm{\phi(u, u) - \onematrix}
                \leq
                \ell(u, u)
                = 0
            $,
        whence $\phi(u, u) = \onematrix$
        for all $u \in \Omega$.
        To establish norm-continuity, let
            $
                C \coloneqq \sup_{(u, v) \in E} \norm{\phi(u, v)} < \infty
            $
        and let $\ell : E \to [0,\:\infty)$
        be a superadditive length function
        witnessing the geometric growth of $\phi$.
        For
            $(u, v),(u', v') \in E$
            with
            $u' \leq u$
            and
            $v \leq v'$
        one computes

            \begin{shorteqnarray}
                \norm{\phi(u',v') - \phi(u, v)}
                &= &\norm{\phi(u', u)\phi(u, v)\phi(v, v') - \phi(u, v)}\\
                &\leq
                    &\norm{\phi(u', u)\phi(u, v)\phi(v, v') - \phi(u', u)\phi(u, v)}\\
                    &&+ \norm{\phi(u', u)\phi(u, v) - \phi(u, v)}
                \\
                &\leq
                    &\norm{\phi(u', u)\phi(u, v)}\norm{\phi(v, v') - \onematrix}\\
                    &&+ \norm{\phi(u', u) - \onematrix}\norm{\phi(u, v)}
                \\
                &\leq &C \cdot (\ell(u', u) + \ell(v, v')).
            \end{shorteqnarray}

        From this one obtains for $(u, v), (u', v') \in E$
        the estimation
            $
                \norm{\phi(u, v) - \phi(u',v')}
                \leq
                    \norm{\phi(u, v) - \phi(\min\{u',u\},\max\{v',v\})}
                    +
                    \norm{\phi(u', v') - \phi(\min\{u',u\},\max\{v',v\})}
                \leq
                    C \cdot (
                        \ell(\min\{u',u\}, u)
                        + \ell(v, \max\{v',v\})
                        + \ell(\min\{u',u\}, u')
                        + \ell(v', \max\{v',v\})
                    )
            $.
        Together with the continuity of $\min$, $\max$, and $\ell$ \wrt the order topology,
        this implies the norm-continuity of $\phi$.
    \end{proof}

\begin{prop}[Norm-continuity under generator continuity]
\makelabel{prop:geom-implies-norm-cts:gen:sig:article-graph-raj-dahya}
    Let
        $\{A(u, v)\}_{(u,v) \in E} \subseteq \BoundedOps{\BanachRaum}$
    be a norm-continuous family of dissipative operators
    on a Banach space $\BanachRaum$,
    and
        $\phi = \{e^{A(u,v)}\}_{(u,v) \in E}$,
    which is a family of contractions.
    If $\{A(u, v)\}_{(u,v) \in E}$ satisfies
    the additivity axiom \eqcref{ax:gen:add:sig:article-graph-raj-dahya},
    then $\phi$ necessarily satisfies
    the identity axiom \eqcref{ax:dyn:id:sig:article-graph-raj-dahya}
    and is norm-continuous.
\end{prop}

    \begin{proof}
        By dissipativity, $\phi$ is a family of contractions.
        Since $A(u, u) = \zeromatrix$ for all $u \in \Omega$ for which $(u, u) \in E$,
        one clearly has that $\phi$ satisfies
        the identity axiom \eqcref{ax:dyn:id:sig:article-graph-raj-dahya}.
        For edges
            $(u,v), (u', v') \in E$
        for which $(u',u),(v,v') \in E$,
        additivity as well as the general estimate in
            \Cref{prop:bch-exp-perturbation:sig:article-graph-raj-dahya}
        for dissipative operators yields

            \begin{restoremargins}
            \begin{equation}
            \label{eq:bch-exp-perturbation:graph:sig:article-graph-raj-dahya}
            \everymath={\displaystyle}
            \begin{array}[m]{rcl}
                \norm{\phi(u', v') - \phi(u, v)}
                &= &\norm{
                        e^{A(u', u) + A(u, v) + A(v, v')}
                        -
                        e^{A(u, v)}
                    }
                    \\
                &= &\norm{
                        e^{A(u, v) + A(u', u) + A(v, v')}
                        -
                        e^{A(u, v)}
                    }
                    \\
                &\eqcrefoverset{eq:bch-exp-perturbation:sig:article-graph-raj-dahya}{\leq}
                    &\norm{A(u', u) + A(v, v')}
                    \\
                &\leq
                    &\norm{A(u',u)} + \norm{A(v,v')}.
            \end{array}
            \end{equation}
            \end{restoremargins}

        For arbitrary edges
            $(u,v), (u', v') \in E$
        one has
            $\quer{u} \coloneqq \min\{u,u'\} \preceq \max\{v,v'\} \eqqcolon \quer{v}$,
        whence
            $(\quer{u}, u), (\quer{u}, u'), (v, \quer{v}), (v', \quer{v}) \in E$.
        So the above estimate yields

            \begin{restoremargins}
            \begin{equation}
            \label{eq:bch-exp-perturbation:lo:sig:article-graph-raj-dahya}
            \everymath={\displaystyle}
            \begin{array}[m]{rcl}
                \norm{\phi(u', v') - \phi(u, v)}
                    &\leq
                        &\norm{\phi(\quer{u}, \quer{v}) - \phi(u', v')}
                        + \norm{\phi(\quer{u}, \quer{v}) - \phi(u, v)}
                    \\
                    &\eqcrefoverset{eq:bch-exp-perturbation:graph:sig:article-graph-raj-dahya}{\leq}
                        &\norm{A(\quer{u}, u)}
                        + \norm{A(v,\quer{v})}
                        + \norm{A(\quer{u}, u')}
                        + \norm{A(v',\quer{v})},
            \end{array}
            \end{equation}
            \end{restoremargins}

        \continueparagraph
        which converges to $0$
        as $(u',v') \longrightarrow (u,v)$ inside $E$,
        since $\max$ and $\min$ are always continuous \wrt the order topology
        and $A(u,u) = \zeromatrix = A(v,v)$.
    \end{proof}

This allows us to obtain the following general examples:

\begin{e.g.}
\makelabel{e.g.:abstract-hamiltonian:sig:article-graph-raj-dahya}
    Let $\mathcal{G} = (\Omega, E)$
    be a graph
    where $E$ is a reflexive linear ordering.
    Endow $\Omega$ with the order topology
    and $E \subseteq \Omega \times \Omega$
    with the relative topology.
    Let
        $\{A(u,v)\}_{(u,v) \in E} \subseteq \BoundedOps{\BanachRaum}$
    be a norm-continuous additive family of dissipative operators
    on a Banach space $\BanachRaum$
    and let $\alpha > 0$.
    Then by \Cref{prop:geom-implies-norm-cts:gen:sig:article-graph-raj-dahya},
        $\phi \coloneqq \{e^{\alpha A(u,v)}\}_{(u,v) \in E}$
    is a norm-continuous family of contractions
    satisfying
    the identity axiom \eqcref{ax:dyn:id:sig:article-graph-raj-dahya}.
    %% Commutative case
    If $\{A(u, v)\}_{(u,v) \in E}$ is a commuting family,
    then

        \begin{restoremargins}
        \begin{equation}
        \label{eq:1:\beweislabel}
            \phi(u, w)
            = e^{\alpha A(u, w)}
            = e^{\alpha\:(A(u, v) + A(v, w))}
            = e^{\alpha A(u, v)} e^{\alpha A(v, w)}
            = \phi(u, v)\phi(v, w)
        \end{equation}
        \end{restoremargins}

    \continueparagraph
    for all $u,v,w \in \Omega$
    for which $(u,v), (v,w) \in E$.
    Thus $\phi$ satisfies
    the divisibility axiom \eqcref{ax:dyn:div:sig:article-graph-raj-dahya}.

    %% Non-commutative case
    Suppose however, that
        $A(u_{0},v_{0})$, $A(v_{0},w_{0})$
    fail to commute for some edges
        $(u_{0},v_{0}), (v_{0},w_{0}) \in E$.
    Then one can choose $\alpha > 0$ such that
    the simplification in \eqcref{eq:1:\beweislabel} fails.%
    \footnote{%
        for arbitrary $X, Y \in \BoundedOps{\BanachRaum}$,
        if $e^{\alpha(X + Y)} = e^{\alpha X}e^{\alpha Y}$
        for all $\alpha > 0$,
        then since the map $\alpha \mapsto e^{\alpha Z}$
        is analytic
        for bounded operators $Z$ on $\BanachRaum$,
        taking 2nd derivatives at $\alpha = 0$ yields
            $(X + Y)^{2} = X^{2} + 2XY + Y^{2}$
        and thus $YX = XY$.
    }
    Thus $\phi$
    fails to satisfy
    the divisibility axiom \eqcref{ax:dyn:div:sig:article-graph-raj-dahya}.
\end{e.g.}

%% ********** END OF FILE: body/sec-2-examples/sec-3-lo.tex **********

%% ********************************************************************************
%% FILE: body/sec-2-examples/sec-4-hamiltonian.tex
%% ********************************************************************************

\subsection[Evolution families]{Evolution families}
\label{sec:examples:hamiltonian:sig:article-graph-raj-dahya}

\firstparagraph
Using the general observations in the previous subsection,
we obtain more concrete classes of examples.
Consider the graph
    $\mathcal{G} = (\Omega, E) \coloneqq (\interval, \geq)$
for some connected subset $\interval \subseteq \reals$.

\begin{e.g.}[Divisible systems]
\makelabel{e.g.:hamiltonian:divisible:sig:article-graph-raj-dahya}
    Let
        $\{A_{\tau}\}_{\tau \in \interval} \subseteq \BoundedOps{\BanachRaum}$
    be a uniformly bounded family of dissipative operators
    on a Banach space $\BanachRaum$
    for which $\{A_{\tau}\}_{\tau \in [s\:t]}$
    is strongly measurable
    for all $(t, s) \in E$.%
    \footnote{%
        This holds for example if
            ${\interval \ni \tau \mapsto A_{\tau}}$
        is strongly continuous,
        since in particular
            $\ran(A\restr{[s,\:t]}(\cdot)\xi) \subseteq \BanachRaum$
        is norm-compact and thus separable
        for all $\xi \in \BanachRaum$
        and $(t, s) \in E$.
    }
    Set
        $C \coloneqq \sup_{\tau \in \interval}\norm{A_{\tau}} < \infty$
    and define
        $
            A(t,s)
            \coloneqq
            \int_{\tau \in [s,\:t]}
                A_{\tau}
            \:\dee\tau
        $
    for $(t,s) \in E$,
    where the integrals are computed strongly via Bochner-integrals.
    Assuming that $\{A_{\tau}\}_{\tau \in \interval}$ is a commuting family,
    then $\{A(t,s)\}_{(t,s) \in E}$
    is a Lipschitz-continuous additive family of commuting dissipative operators.
    Let $\alpha > 0$.
    As in \Cref{e.g.:abstract-hamiltonian:sig:article-graph-raj-dahya},
        $\{\phi(t,s) \coloneqq e^{\alpha A(t,s)}\}_{(t,s) \in E}$
    defines a norm-continuous divisible dynamical system on $\mathcal{G}$
    consisting of contractions.
    Moreover, by the estimate in \eqcref{eq:bch-exp-perturbation:lo:sig:article-graph-raj-dahya},
    for each $(t,s),(t',s') \in E$,
    setting $\quer{t} \coloneqq \min_{E}\{t,t'\} = \max\{t,t'\}$
    and $\quer{s} \coloneqq \max_{E}\{s,s'\} = \min\{s,s'\}$,
    one has

    \begin{restoremargins}
    \begin{equation}
    \label{eq:bch-exp-perturbation:ev:sig:article-graph-raj-dahya}
    \everymath={\displaystyle}
    \begin{array}[m]{rcl}
        \norm{\phi(t',s') - \phi(t,s)}
        &\leq
            &\alpha C \cdot (
                (\quer{t} - t')
                + (\quer{t} - t)
                + (\quer{s} - s')
                + (\quer{s} - s)
            )
        \\
        &= &\alpha C \cdot (\abs{t' - t}  + \abs{s' - s}),
    \end{array}
    \end{equation}
    \end{restoremargins}

    \continueparagraph
    whence $\phi$ is Lipschitz-continuous.
    By \eqcref{eq:bch-exp-perturbation:ev:sig:article-graph-raj-dahya}
    one obtains
        $
            \norm{\phi(t, s) - \onematrix}
            = \norm{\phi(t, s) - \phi(s, s)}
            \leq \alpha C \cdot (\abs{t - s} + \abs{s - s})
            = \alpha C \cdot (t - s)
            \eqqcolon \ell(t, s)
        $
    for all $(t, s) \in E$.
    Clearly, $\ell$ is a continuous additive function,
    whence $\phi$ has geometric growth.%
    \footnote{%
        see \Cref{defn:geom-growth:op-family:sig:article-graph-raj-dahya}.
    }
\end{e.g.}

Note that by taking derivatives,
for $\{A(t,s)\}_{(t,s) \in E}$ to be commuting,
it is necessary for
    $\{A_{\tau}\}_{\tau \in \interval}$
to be a commuting family.
To obtain indivisible systems,
we thus need to consider non-commuting families.
%% Construction
Consider for example the following setup:
Assume $\interval \supseteq [0,\:t_{\max}]$
for some $t_{\max} > 0$.
Let $\BanachRaum$ be any non-commutative \TextCStarAlg $\CStarAlg$.
Furthermore, assume that $\CStarAlg$
is \uline{not} Lie-nilpotent of class $2$.
Then we can find non-commuting self-adjoint elements
    $h_{1}, h_{2} \in \CStarAlg$
for which $\hat{h} \coloneqq [h_{1},\:h_{2}]$
is not in the centraliser of $\CStarAlg$.%
\footnote{%
    \idest $\hat{h}$ does not commute with every element in $\CStarAlg$.
}
Set
    $\Psi_{i} \coloneqq \iunit[h_{i},\cdot]$
for $i \in \{1,2\}$
and let

    \begin{displaymath}
        A_{\tau}
        \coloneqq
            \tfrac{\tau}{t_{\max}^{2}}\Psi_{1}
            +
            \tfrac{(t_{\max} - \tau)}{t_{\max}^{2}}\Psi_{2}
    \end{displaymath}

\continueparagraph
for each $\tau \in \interval$.
By \Cref{cor:linblad:dissipative:sig:article-graph-raj-dahya},
    $\{A_{\tau}\}_{\tau \in \interval}$
is a (clearly norm-continuous) family of dissipative operators.
Computing integrals yields

    \begin{shorteqnarray}
        A(t, s)
        &=
            &\tfrac{t^{2} - s^{2}}{2 t_{\max}^{2}}
            \Psi_{1}
            +
            \tfrac{2 t_{\max} (t-s) - (t^{2} - s^{2})}{2 t_{\max}^{2}}
            \Psi_{2}
            \\
        &=
            &\tfrac{t-s}{t_{\max}}
            \Big(
                \tfrac{t + s}{2 t_{\max}}
                \Psi_{1}
                +
                (1 - \tfrac{t + s}{2 t_{\max}})
                \Psi_{2}
            \Big)
    \end{shorteqnarray}

\continueparagraph
for $(t,s) \in E$.
Choosing
    $r_{0} \coloneqq 0$,
    $s_{0} \coloneqq \tfrac{1}{2}t_{\max}$,
    and
    $t_{0} \coloneqq t_{\max}$,
one has
    $
        A(t_{0}, s_{0})
            = \tfrac{1}{2}
                (
                    \tfrac{3}{4}
                    \Psi_{1}
                    +
                    \tfrac{1}{4}
                    \Psi_{2}
                )
            = \tfrac{1}{8}(3\Psi_{1} + \Psi_{2})
    $
and similarly
    $
        A(s_{0}, r_{0})
            = \tfrac{1}{8}(\Psi_{1} + 3\Psi_{2})
    $.
Basic operations with commutators
yields
    $
        [\Psi_{1},\:\Psi_{2}]
        = [\iunit[h_{1},\:\cdot], \iunit[h_{2},\:\cdot]]
        = -[h_{1},\:[h_{2},\:\cdot]]
          + [h_{2},\:[h_{1},\:\cdot]]
        = -[[h_{1}, h_{2}],\:\cdot]
    $
and thus

    \begin{shorteqnarray}
        [A(t_{0},s_{0}),\:A(s_{0},r_{0})]
            =
                [
                    \tfrac{1}{8}(3\Psi_{1} + \Psi_{2}),
                    \:\tfrac{1}{8}(\Psi_{1} + 3\Psi_{2})
                ]
            = \tfrac{3^{2}-1}{8^{2}}[\Psi_{1},\Psi_{2}]
            = -\tfrac{1}{8}[[h_{1}, h_{2}],\:\cdot],
    \end{shorteqnarray}

\continueparagraph
which is not the $\zeromatrix$-operator,
since by construction $[h_{1}, h_{2}]$
is not in the centraliser of $\CStarAlg$.

\begin{e.g.}[Indivisible systems]
\makelabel{e.g.:hamiltonian:indivisible:sig:article-graph-raj-dahya}
    Suppose in \Cref{e.g.:hamiltonian:divisible:sig:article-graph-raj-dahya}
    there are times $t_{0} \geq s_{0} \geq r_{0}$ in $\interval$
    for which
        $A(t_{0},s_{0})$, $A(s_{0}, r_{0})$
    fail to commute (see \exempli the above construction).
    Then as in \Cref{e.g.:abstract-hamiltonian:sig:article-graph-raj-dahya},
        $\phi \coloneqq \{e^{\alpha A(t,s)}\}_{(t,s) \in E}$
    fails to satisfy
    the divisibility axiom \eqcref{ax:dyn:div:sig:article-graph-raj-dahya}
    for some $\alpha > 0$.
    Everything else in \Cref{e.g.:hamiltonian:divisible:sig:article-graph-raj-dahya}
    however continues to hold:
        $\phi$
    is a norm-continuous
    family of contractions on $\mathcal{G}$
    satisfying
    the identity axiom \eqcref{ax:dyn:id:sig:article-graph-raj-dahya}
    but \uline{not}
    the divisibility axiom \eqcref{ax:dyn:div:sig:article-graph-raj-dahya}.
    This family again has Lipschitz-continuity and geometric growth.
\end{e.g.}

%% ********** END OF FILE: body/sec-2-examples/sec-4-hamiltonian.tex **********

%% ********************************************************************************
%% FILE: body/sec-2-examples/sec-5-networks.tex
%% ********************************************************************************

\subsection[Dynamical systems on networks]{Dynamical systems on networks}
\label{sec:examples:networks:sig:article-graph-raj-dahya}

\firstparagraph
We now briefly consider structures more adequate
to model semantics in automata theory and propagation in neural networks.
Consider a finite acyclic directed graph
    $\mathcal{G} = (\Omega, E)$
whose edges are weighted by a family
    $\{w_{u,v}\}_{(u,v) \in E} \subseteq \BoundedOps{\BanachRaum}$
of bounded operators on a Banach space $\BanachRaum$.%
\footnote{%
    $\mathcal{G}$ is acyclic there are no finite sequences
    of edges $(u_{0},u_{1}), (u_{1},u_{2}), \ldots, (u_{n-1}, u_{n}) \in E$
    with $n \geq 1$ and $u_{0} = u_{n}$.
}
Note that in automata theory, it is more common to allow cycles, including loops.
However, our goal here is to demonstrate how,
even under simple assumptions,
indivisible dynamical systems can arise.

As in \Cref{rem:path-independence:sig:article-graph-raj-dahya},
for each $u, v \in \Omega$
we let $\mathrm{Path}(u,v)$
denote the set of all finite sequences
    $\pi = \{u_{k}\}_{k=0}^{n} \subseteq \Omega$
with $n \in \naturalsZero$
and $(u_{k-1},u_{k}) \in E$ for all $k \in \{1,2,\ldots,n\}$.%
\footnote{%
    such sequences are technically referred to as \highlightTerm{walks}.
}
By the assumptions on $\mathcal{G}$,
each such set is finite.
We may thus define an operator family
    $\phi = \{\phi(u,v)\}_{(u,v) \in \Omega\times\Omega} \subseteq \BoundedOps{\BanachRaum}$
for the graph
    $(\Omega,\Omega \times \Omega)$
via

    \begin{restoremargins}
    \begin{equation}
    \label{eq:networks:sum:sig:article-graph-raj-dahya}
        \phi(u, v)
            \coloneqq
                \sum_{\pi \in \mathrm{Path}(u,v)}
                    w_{\pi}
    \end{equation}
    \end{restoremargins}

\continueparagraph
for $(u, v) \in E$, where

    \begin{restoremargins}
    \begin{equation}
    \label{eq:networks:path:sig:article-graph-raj-dahya}
        w_{\pi}
            \coloneqq
                \prod_{k=1}^{n}
                    w_{u_{k-1},u_{k}}
    \end{equation}
    \end{restoremargins}

\continueparagraph
for all $\pi = \{u_{k}\}_{k=0}^{n} \in \mathrm{Path}(u, v)$,
$n \in \naturalsZero$.

Since $\mathcal{G}$ is acyclic, it contains no loops.
So for each $u \in \Omega$
the set $\mathrm{Path}(u,u)$ only contains one walk,
\viz the trivial walk $\pi = \{u_{k}\}_{k=0}^{n}$ with $n=0$ and $u_{0} = u$.
Since the product expression in \eqcref{eq:networks:path:sig:article-graph-raj-dahya} is empty,
one has $\phi(u, u) = w_{\pi} = \onematrix$.
So $\phi$ satisfies
the identity axiom \eqref{ax:dyn:id:sig:article-graph-raj-dahya}.

However, $\phi$ fails to be divisible in general:
Consider edges $u,v,w \in \Omega$.
Combining \eqcref{eq:networks:sum:sig:article-graph-raj-dahya}
and \eqcref{eq:networks:path:sig:article-graph-raj-dahya}
yields

    \begin{shorteqnarray}
        \phi(u, w) - \phi(u, v)\phi(v, w)
            =
                \sum_{\pi \in \mathrm{Path}(u,w) \without \mathrm{Path}(u,w;\{v\})}
                    w_{\pi},
    \end{shorteqnarray}

\continueparagraph
where $\mathrm{Path}(u,w;\{v\})$
is the set of all $\{u_{k}\}_{k=0}^{n} \in \mathrm{Path}(u,w)$,
$n \in \naturalsZero$,
for which $v \in \{u_{0},u_{1},\ldots,u_{n}\}$.
Provided there is some sequence of edges $\pi$ from $u$ to $w$
in $\mathcal{G}$
which does not cross the node $v$,
we can organise \exempli for all the weights to be positive multiples of the identity,
and thereby ensure that
    $\phi(u, w) - \phi(u, v)\phi(v, w) \neq \zeromatrix$.

%% ********** END OF FILE: body/sec-2-examples/sec-5-networks.tex **********

%% ********** END OF FILE: body/sec-2-examples/.index.tex **********

%% ********************************************************************************
%% FILE: body/sec-3-algebra/.index.tex
%% ********************************************************************************

\section[Algebraic structures on graphs]{Algebraic structures on graphs}
\label{sec:algebra:sig:article-graph-raj-dahya}

\firstparagraph
Towards \ref{goal:embed:sig:article-graph-raj-dahya},
we note that the axioms
    \eqcref{ax:dyn:id:sig:article-graph-raj-dahya}
    and
    \eqcref{ax:dyn:div:sig:article-graph-raj-dahya}
contain the germs of an algebraic structure.
We thus turn our attention to the theory of reduction systems
and quotients of free groups,
to construct out of the edges of graphs
an appropriate algebraic structure.

%% ********************************************************************************
%% FILE: body/sec-3-algebra/sec-1-algebra.tex
%% ********************************************************************************

\subsection[Algebraic definitions]{Algebraic definitions}
\label{sec:algebra:defn:sig:article-graph-raj-dahya}

\firstparagraph
The following definitions and results can be found in the literature,
\exempli
    \cite[Chapter~1--2]{BookOtto1993BookRewritingSystems},
    \cite[\S{}12.2--3]{HoltEickObrien2005Book}.

%% ********************************************************************************
%% FILE: body/sec-3-algebra/sec-1-para-1-reduction.tex
%% ********************************************************************************

\subsubsection[Reduction systems]{Reduction systems}
\label{sec:algebra:defn:reduction:sig:article-graph-raj-dahya}

\firstparagraph
A \highlightTerm{reduction system} $(X, \to)$
is any set $X$ endowed with a binary relation
    $
        \mathbin{\to}
        \subseteq
        X \times X
    $.
Let $x,y \in X$
We write $x \overset{\ast}{\to} y$,
if there is a \highlightTerm{walk}
in the graph $(X, \to)$
from $x$ to $y$,
\idest
if there exists $n \in \naturalsZero$
and $\{x_{k}\}_{k=0}^{n} \subseteq X$,
such that
    $x=x_{0} \to x_{1} \to \cdots \to x_{n} = y$.
Defining
    $
        \mathbin{\leftrightarrow}
        \coloneqq
            \mathbin{\to}
            \cup
            \mathbin{\to}^{-1}
        \subseteq
            X \times X
    $,
one similarly defines
    $x \overset{\ast}{\leftrightarrow} y$
if there is a walk in the graph $(X, \leftrightarrow)$.
In other words, $\overset{\ast}{\to}$
is the reflexive and transitive hull of $\to$
and $\overset{\ast}{\leftrightarrow}$
is the smallest equivalence relation extending $\to$.

We shall say that $x \in X$ is \highlightTerm{irreducible},
if there is no $y \in X$ with $x \to y$,
and we let $\Irred{X} \subseteq X$
denote the subset of irreducible elements.
We shall also say that $x, y \in X$
is \highlightTerm{compatible} written $x \parallel y$,
if there exists $z \in X$
with
    $x \overset{\ast}{\to} z$
    and
    $y \overset{\ast}{\to} z$.
A reduction system $(X, \to)$

\begin{enumerate}[
    label={\texttt{RS\textsubscript{\arabic*}})},
    ref={\texttt{RS\textsubscript{\arabic*}}},
]
\item\label{ax:rs:noeth:sig:article-graph-raj-dahya}
    is \highlightTerm{Noetherian}
    if there is no infinite sequence
    $\{x_{n}\}_{n\in\naturals} \subseteq X$
    with $x_{1} \to x_{2} \to x_{3} \to \ldots$;%
    \footnote{%
        In set theoretic terms,
        under the axiom of choice,
        this is equivalent
        to $(X, \to)$ being a \highlightTerm{well-founded} relation,
        \idest for all non-empty $A \subseteq X$
        there exists a \usesinglequotes{minimal} element $x \in A$
        in the sense that there be no $y \in A$ with $x \to y$.
        In fact, this holds under weaker background axioms
        than choice, see \exempli \cite[Lemma~5.5(ii)]{Jech2003}.
    }
\item\label{ax:rs:church-rosser:sig:article-graph-raj-dahya}
    has the \highlightTerm{Church--Rosser property}
    if
        $x \overset{\ast}{\leftrightarrow} y$
    implies $x \parallel y$
    for all $x, y \in X$;
\item\label{ax:rs:conf:sig:article-graph-raj-dahya}
    is \highlightTerm{confluent}
    if
        $w \overset{\ast}{\to} x$
        and
        $w \overset{\ast}{\to} y$
    implies $x \parallel y$
    for all $w, x, y \in X$;
\item\label{ax:rs:loc-conf:sig:article-graph-raj-dahya}
    is \highlightTerm{locally confluent}
    if
        $w \to x$
        and
        $w \to y$
    implies $x \parallel y$
    for all $w, x, y \in X$.
\end{enumerate}

\continueparagraph
Clearly,
    \eqcref{ax:rs:church-rosser:sig:article-graph-raj-dahya}
    $\Rightarrow$
    \eqcref{ax:rs:conf:sig:article-graph-raj-dahya}
    $\Rightarrow$
    \eqcref{ax:rs:loc-conf:sig:article-graph-raj-dahya}.
In fact one has:

\begin{prop}
\makelabel{prop:cr-iff-conf:sig:article-graph-raj-dahya}
    A reduction system $(X, \to)$ has the Church--Rosser property
    if and only if it is confluent.
\end{prop}

\begin{prop}
\makelabel{prop:conf-iff-loc-conf:sig:article-graph-raj-dahya}
    Suppose that a given reduction system $(X, \to)$ is Noetherian.
    Then $(X, \to)$ is confluent if and only if it is locally confluent.
\end{prop}

Thus under \eqcref{ax:rs:noeth:sig:article-graph-raj-dahya},
the axioms
    \eqcref{ax:rs:church-rosser:sig:article-graph-raj-dahya},
    \eqcref{ax:rs:conf:sig:article-graph-raj-dahya},
    and
    \eqcref{ax:rs:loc-conf:sig:article-graph-raj-dahya}
are equivalent.
See \cite[Lemma~1.1.7 and Theorem~1.1.13]{BookOtto1993BookRewritingSystems}
for a proof.
\Cref{prop:conf-iff-loc-conf:sig:article-graph-raj-dahya}
is referred to as the \highlightTerm{diamond lemma},
attributed to Newman
\cite[\S{}7]{Newman1942DiamondLemma}.%
\footnote{%
    not to be confused with the Bergman \highlightTerm{diamond lemma}
    \cite[\S{}7]{Bergman1978DiamondLemma},
    which is similar but different.
}
One further has:

\begin{prop}
\makelabel{prop:rs-normal-form:sig:article-graph-raj-dahya}
    Suppose that a given reduction system $(X, \to)$
    is Noetherian and confluent.
    Then every element $x \in X$
    possesses a unique \highlightTerm{normal form},
    that is an irreducible element $\hat{x} \in X$
    for which $x \overset{\ast}{\leftrightarrow} \hat{x}$.
    In fact $x \overset{\ast}{\to} \hat{x}$
    for all $x \in X$ with normal form $\hat{x}$.
    For any $x' \in X$ with $x' \overset{\ast}{\leftrightarrow} x$
    one has that $x' \overset{\ast}{\to} \hat{x}$.
\end{prop}

    \begin{proof}
        Proof of the existence and uniqueness is a simple exercise
        (see \cite[Theorem~1.1.12]{BookOtto1993BookRewritingSystems}).
        Towards the penultimate claim,
        since $x \overset{\ast}{\leftrightarrow} \hat{x}$
        and since by \Cref{prop:cr-iff-conf:sig:article-graph-raj-dahya}
        confluence is equivalent to the Church--Rosser property,
        there exists $y \in X$
        such that
            $x \overset{\ast}{\to} y$
            and
            $\hat{x} \overset{\ast}{\to} y$.
        Since $\hat{x}$ is irreducible,
        it follows that $\hat{x} = y$
        and thus $x \overset{\ast}{\to} y = \hat{x}$.

        Towards the final claim,
        for any $x' \in X$ with $x' \overset{\ast}{\leftrightarrow} x$,
        since $x \overset{\ast}{\leftrightarrow} \hat{x}$,
        one has $x' \overset{\ast}{\leftrightarrow} \hat{x}$.
        Arguing as above,
        by the Church--Rosser property and irreducibility,
        one again obtains that $x' \overset{\ast}{\to} \hat{x}$.
    \end{proof}

The existence and above properties of such normal forms
make it desirable to work with reduction systems
which are Noetherian and confluent ($\equiv$ locally confluent).

%% ********** END OF FILE: body/sec-3-algebra/sec-1-para-1-reduction.tex **********

%% ********************************************************************************
%% FILE: body/sec-3-algebra/sec-1-para-2-rewrite.tex
%% ********************************************************************************

\subsubsection[String-rewriting systems]{String-rewriting systems}
\label{sec:algebra:defn:rewrite:sig:article-graph-raj-dahya}

\firstparagraph
We now describe reduction systems that occur in the study of free algebraic structures,
\viz \highlightTerm{string-rewriting systems}.
Let $\Gamma$ be an arbitrary non-empty set of elements,
which shall be referred to as an \highlightTerm{alphabet}.
We refer to the elements of $\Gamma$ as \highlightTerm{letters}.
For each $n \in \naturalsZero$
the set of $n$\=/tuples over $\Gamma$ is denoted
as $\Gamma^{n}$,
where in particular $\Gamma^{0}$ contains a single element,
\namely the \highlightTerm{empty word},
which we shall denote via $1$.%
\footnote{%
    The empty word is typically denoted $\eps$ in the literature.
    We use $1$ to avoid conflicting terminology.
}
The disjoint union
$
    \Gamma^{\ast}
    \coloneqq
    \bigcup_{n \in \naturalsZero}
        \Gamma^{n}
$,
referred to as the \highlightTerm{set of all words over $\Gamma$},
constitutes a monoid under the operation of (tuple) concatenation,
with $1$ as its neutral element.
It shall also be useful to define
$
    \Gamma^{+}
    \coloneqq
    \bigcup_{n \in \naturals}
        \Gamma^{n}
$,
which is the set of all non-empty words.
For $n\in\naturalsZero$ and $w \in \Gamma^{n}$,
we define $\card{w} \coloneqq n$
to the \highlightTerm{word length} of $w$.

Consider now a relation of the form
    $
        R \subseteq (\Gamma^{\ast})^{+} \times \Gamma^{\ast}
    $,
\idest
$R$ associates arbitrary $n$\=/tuples of words over $\Gamma$,
where $n\in\naturals$,
to words over $\Gamma$.
We shall refer to the elements of $R$ as \highlightTerm{rules}.
We associate to
$(\Gamma, R)$
the reduction system
$(\Gamma^{\ast}, \to_{R})$
defined via

\begin{restoremargins}
\begin{equation}
\label{eq:context-free:sig:article-graph-raj-dahya}
    x \to_{R} y
    :\Leftrightarrow
    \begin{array}[t]{l}
        \exists{w, w' \in \Gamma^{\ast}:~}
        \exists{n\in\naturals:~}
        \exists{(\{u_{k}\}_{k=1}^{n}, v) \in R:~}\\
        \quad
        x = w \cdot u_{1} \cdot u_{2} \cdot \ldots \cdot u_{n} \cdot w'
        ~\text{and}\\
        \quad
        y = w \cdot v \cdot w'
    \end{array}
\end{equation}
\end{restoremargins}

\continueparagraph
for $x, y \in \Gamma^{\ast}$,
and refer to this structure as a \highlightTerm{string-rewriting system}.
Such binary reduction rules defined on words are referred to as \highlightTerm{context-free grammars}.
We now demonstrate sufficient conditions
for the associated reduction system
to be Noetherian and confluent.

\begin{prop}
\makelabel{prop:rew-sys-noeth:sig:article-graph-raj-dahya}
    Suppose that
    $R$ is \highlightTerm{strictly monotone decreasing}
    in the sense that
        $
            \card{y}
            < \sum_{k=1}^{n}
                \card{x_{i}}
        $
    for all $n\in\naturals$
    and $(\{x_{k}\}_{k=1}^{n}, y) \in R$.
    Then $(\Gamma^{\ast}, \to_{R})$
    is Noetherian.
\end{prop}

    \begin{proof}
        If $\{x_{n}\}_{n\in\naturals} \subseteq \Gamma^{\ast}$
        were an infinite sequence with
            $x_{1} \to_{R} x_{2} \to_{R} \ldots$,
        then $\card{x_{1}} > \card{x_{2}} > \ldots$,
        which is impossible.
    \end{proof}

By considering the following axioms,
which we shall refer to as \highlightTerm{pre-algebraic axioms},
one can obtain sufficient conditions for local confluence:

\begin{enumerate}[
    label={\texttt{PA}\textsubscript{\arabic*}},
    ref={\texttt{PA}\textsubscript{\arabic*}},
    left=\rtab,
]
    \item\label{ax:pre-alg:neutral:sig:article-graph-raj-dahya}
        (\highlightTerm{Pre-Identity})
        For all rules
            $((x,y),z) \in R$,
        if
            $(x, 1) \in R$
        then $y = z$
        and
        if
            $(y, 1) \in R$
        then $x = z$.

    \item\label{ax:pre-alg:associative:sig:article-graph-raj-dahya}
        (\highlightTerm{Pre-Associativity})
        For all rules
            $((x,y),z)$
            and
            $((x',y'),z')$
            in $R$,
        if $y = x'$,
        there exist rules
            $((x, z'), w)$
            and
            $((z, y'), w)$
            in $R$
        for some $w \in \Gamma^{\ast}$.
\end{enumerate}

\begin{prop}
\makelabel{prop:rew-sys-loc-conf:sig:article-graph-raj-dahya}
    Suppose that
        $
            R
            \subseteq
            (\Gamma^{1})^{1} \times \Gamma^{0}
            \cup
            (\Gamma^{1})^{2} \times \Gamma^{1}
        $,
    \idest $R$
    consists of rules which reduce
    certain letters to the empty word
    and
    certain pairs of letters to single letters.
    Suppose further that $R$
    satisfies the pre-algebraic axioms
        \eqcref{ax:pre-alg:neutral:sig:article-graph-raj-dahya}
        and
        \eqcref{ax:pre-alg:associative:sig:article-graph-raj-dahya}.
    Then $(\Gamma^{\ast}, \to_{R})$
    is locally confluent.
\end{prop}

    \begin{proof}
        For $w,u,u'\in \Gamma^{\ast}$
        with $w \to_{R} u$
        and $w \to_{R} u'$
        there are two trivial and three non-trivial cases
        to consider:

        \begin{enumerate}[
            label={\itshape Case~\arabic*.},
            ref={\itshape Case~\arabic*},
            left=\rtab,
        ]
            \item
                $u = u'$.
                Then
                $u \overset{\ast}{\to}_{R} u$
                and
                $u' \overset{\ast}{\to}_{R} u' = u$.
            \item
                The reductions
                $w \to_{R} u$
                and
                $w \to_{R} u'$
                are obtained by reducing disjoint parts of $w$.
                Then due to the context-free nature of the reduction,
                one may apply these reductions successively in any order
                to obtain
                    $w \to_{R} u \to_{R} \tilde{w}$
                and
                    $w \to_{R} u' \to_{R} \tilde{w}$
                for some common word $\tilde{w} \in \Gamma^{\ast}$.

            \item
                There exist $x,y,z \in \Gamma$
                with
                    $(x, 1) \in R$
                    and $((x,y), z) \in R$,
                such that
                    $w = \alpha x y \beta$,
                    $u = \alpha y \beta$,
                    and
                    $u' = \alpha z \beta$
                    (or vice versa),
                for some words $\alpha, \beta \in \Gamma^{\ast}$.
                By \eqcref{ax:pre-alg:neutral:sig:article-graph-raj-dahya}
                one has that $y = z$
                and thus $u = u'$.
                Hence
                    $u \overset{\ast}{\to}_{R} u$
                    and
                    $u' \overset{\ast}{\to}_{R} u' = u$.

            \item
                There exist $x,y,z \in \Gamma$
                with
                    $(y, 1) \in R$
                    and $((x,y), z) \in R$,
                such that
                    $w = \alpha x y \beta$,
                    $u = \alpha x \beta$,
                    and
                    $u' = \alpha z \beta$
                    (or vice versa),
                for some words $\alpha, \beta \in \Gamma^{\ast}$.
                By \eqcref{ax:pre-alg:neutral:sig:article-graph-raj-dahya}
                one has that $x = z$
                and thus $u = u'$.
                Hence
                    $u \overset{\ast}{\to}_{R} u$
                    and
                    $u' \overset{\ast}{\to}_{R} u' = u$.

            \item
                There exist rules
                    $((x,y),z)$
                    and
                    $((x',y'),z')$
                    in $R$
                with $y = x'$
                such that
                    $
                        w = \alpha x y y' \beta
                        = \alpha x x' y' \beta
                    $,
                    $u = \alpha z y' \beta$,
                and $u' = \alpha x z' \beta$
                for some words $\alpha, \beta \in \Gamma^{\ast}$.
                By \eqcref{ax:pre-alg:associative:sig:article-graph-raj-dahya}
                    $((x, z'), s)$
                    and
                    $((z, y'), s)$
                    in $R$
                for some $s \in \Gamma^{\ast}$
                (in fact $s \in \Gamma$).
                It follows that
                    $u = \alpha z y' \beta \to_{R} \alpha s \beta$
                    and
                    $u' = \alpha x z' \beta \to_{R} \alpha s \beta$.
        \end{enumerate}

        So in all cases,
        some common word $\tilde{w} \in \Gamma^{\ast}$
        exists such that
            $u \overset{\ast}{\to}_{R} \tilde{w}$
        and
            $u' \overset{\ast}{\to}_{R} \tilde{w}$.
        Since $w, u, u'$ were arbitrarily chosen,
        it follows that $(\Gamma^{\ast},\to_{R})$
        is locally confluent.
    \end{proof}

Combining
\Cref{%
    prop:conf-iff-loc-conf:sig:article-graph-raj-dahya,%
    prop:rew-sys-noeth:sig:article-graph-raj-dahya,%
    prop:rew-sys-loc-conf:sig:article-graph-raj-dahya%
}
yields

\begin{prop}
\makelabel{prop:rew-sys-convergent:sig:article-graph-raj-dahya}
    Let $\Gamma$ be a non-empty alphabet
    and let
        $
            R
            \subseteq
            (\Gamma^{1})^{1} \times \Gamma^{0}
            \cup
            (\Gamma^{1})^{2} \times \Gamma^{1}
        $
    be a set of rules satisfying
    \eqcref{ax:pre-alg:neutral:sig:article-graph-raj-dahya}
    and
    \eqcref{ax:pre-alg:associative:sig:article-graph-raj-dahya}.
    Then the string-rewriting system
    $(\Gamma^{\ast},\to_{R})$
    associated to $(\Gamma, R)$
    is Noetherian and confluent.
\end{prop}

%% ********** END OF FILE: body/sec-3-algebra/sec-1-para-2-rewrite.tex **********

%% ********************************************************************************
%% FILE: body/sec-3-algebra/sec-1-para-3-free.tex
%% ********************************************************************************

\subsubsection[Quotient monoid of a string-rewriting system]{Quotient monoid of a string-rewriting system}
\label{sec:algebra:defn:free:sig:article-graph-raj-dahya}

\firstparagraph
Let $(\Gamma, R)$ be an alphabet and set of rules
and consider the associated string-rewriting system
    $(\Gamma^{\ast}, \to_{R})$.
As in \S{}\ref{sec:algebra:defn:reduction:sig:article-graph-raj-dahya},
the reduction system
    $(\Gamma^{\ast}, \to_{R})$
induces an equivalence relation
    $(\Gamma^{\ast}, \overset{\ast}{\leftrightarrow}_{R})$.
For $x \in \Gamma^{\ast}$
we let $[x]_{R}$ or simply $[x]$
denote the equivalence class of $x$
within this structure.

Let $u,v,x,y \in \Gamma^{\ast}$ be arbitrary words over $\Gamma$.
Relying on the definition in \eqcref{eq:context-free:sig:article-graph-raj-dahya},
it is a simple exercise to show that
    ${x \to_{R} y}$
    implies ${u x v \to_{R} u y v}$.
And by induction one may obtain that
    ${x \overset{\ast}{\to}_{R} y}$
    implies
    ${u x v \overset{\ast}{\to}_{R} u y v}$.
In a similar way, one may arrive at the fact that
    ${x \overset{\ast}{\leftrightarrow}_{R} y}$
    implies
    ${u x v \overset{\ast}{\leftrightarrow}_{R} u y v}$.
The equivalence relation
    $\overset{\ast}{\leftrightarrow}_{R}$
is thus a \highlightTerm{congruence}
on the monoid $(\Gamma^{\ast},\cdot,1)$.

Since the equivalence relation is a congruence,
it follows that the operation
    $[x]_{R} [y]_{R} \coloneqq [x y]_{R}$
for $x, y \in \Gamma^{\ast}$
is well-defined.%
\footnote{%
    since for $x,x',y,y'\in\Gamma^{\ast}$
    with
        $x \overset{\ast}{\leftrightarrow}_{R} x'$
    and
        $y \overset{\ast}{\leftrightarrow}_{R} y'$,
    the congruence yields
        $
            x y
            \overset{\ast}{\leftrightarrow}_{R}
            x' y
            \overset{\ast}{\leftrightarrow}_{R}
            x' y'
        $
    and thus $[x y]_{R} = [x' y']_{R}$,
    so that the operation does not depend on the choice of representatives.
}
One may thus define the (quotient) \highlightTerm{monoid associated to} $(\Gamma, R)$
as
    $
        \Generate[Mon]{\Gamma}{R}
        \coloneqq \{[x]_{R} \mid x \in \Gamma^{\ast}\}
    $
endowed with the operation
    $
        [x]_{R} [y]_{R} = [x y]_{R}
    $
for $x, y \in \Gamma^{\ast}$
and the neutral element $1 \coloneqq [1]_{R}$.
Observe in particular that

\begin{restoremargins}
\begin{equation}
\label{eq:epimorphism-free-monoid:sig:article-graph-raj-dahya}
\everymath={\displaystyle}
\begin{array}[m]{rcl}
        \Gamma^{\ast} &\to &\Generate[Mon]{\Gamma}{R}\\
        x &\mapsto &[x]_{R}
\end{array}
\end{equation}
\end{restoremargins}

\continueparagraph
is a surjective morphism between the monoids.

Finally note that one can similarly associate a group
to $(\Gamma^{\ast},\to_{R})$.
Fix a bijection ${\theta : \Gamma \to \Gamma^{-1}}$
between $\Gamma$
and a disjoint copy $\Gamma^{-1}$ of the alphabet.
Define
    $
        R_{0}
        \coloneqq
        \{
            ((a, a^{-1}), 1)
            \mid a \in \Gamma
        \}
        \cup
        \{
            ((a^{-1}, a), 1)
            \mid a \in \Gamma
        \}
    $
where
    $a^{-1} \coloneqq \theta(a) \in \Gamma^{-1}$
    for $a \in \Gamma$.
One can then define
    $
        \Generate[Gp]{\Gamma}{R}
        \coloneqq
        \Generate[Mon]{\Gamma \cup \Gamma^{-1}}{R_{0} \cup R}
    $,
and verify that this defines a group.
We shall however, not need these constructions in the present paper.

%% ********** END OF FILE: body/sec-3-algebra/sec-1-para-3-free.tex **********

%% ********** END OF FILE: body/sec-3-algebra/sec-1-algebra.tex **********

%% ********************************************************************************
%% FILE: body/sec-3-algebra/sec-2-graphs.tex
%% ********************************************************************************

\subsection[Groups induced by graphs]{Groups induced by graphs}
\label{sec:algebra:graphs:sig:article-graph-raj-dahya}

\firstparagraph
The definitions of the previous subsection shall now be applied to graphs.
We start by creating a string-rewriting system
on an alphabet whose letters arise from edges.
By transitively closing the set of edges under traversal,
we obtain a Noetherian confluent system.
By further closing under symmetry,
the quotient monoid associated to the string-rewriting system
provides us with a group.

%% ********************************************************************************
%% FILE: body/sec-3-algebra/sec-2-para-1-rewrite.tex
%% ********************************************************************************

\subsubsection[String-rewriting system for graphs]{String-rewriting system for graphs}
\label{sec:algebra:graphs:rewrite:sig:article-graph-raj-dahya}

\firstparagraph
Let $\mathcal{G} = (\Omega, E)$ be an arbitrary graph,
where $\Omega$ is a non-empty set of nodes
and $E \subseteq \Omega \times \Omega$ a (possibly empty) set of edges.
For $F \subseteq \Omega \times \Omega$
let $\Generate{F} \subseteq \Omega \times \Omega$
denote its transitive hull.
Set
    $
        \breve{E}
        \coloneqq
        \Generate{E \cup E^{-1} \cup \Gph{\id_{\Omega}}}
    $
where
    $E^{-1} = \{(v, u) \mid (u, v) \in E\}$
    and
    $\Gph{\id_{\Omega}} = \{(u, u) \mid u \in \Omega\}$.
Then
    $\breve{E}$
is the smallest equivalence relation on $\Omega$ containing $E$.
For each pair $(u, v) \in \Omega \times \Omega$
define a distinct \usesinglequotes{letter}
    $\letter{(u,v)}$.
We shall refer to

\begin{restoremargins}
\begin{equation}
\label{eq:alphabet-graph:sig:article-graph-raj-dahya}
    \Gamma_{\mathcal{G}}
        \coloneqq \{
            \letter{(u,v)}
            \mid
            (u, v) \in \breve{E}
        \}
\end{equation}
\end{restoremargins}

\continueparagraph
as the \highlightTerm{edge alphabet},%
\footnote{%
    One can of course choose the set of letters
    to simply be $E$ itself
    with $\letter{e} \coloneqq e$
    for each edge $e \in E$.
    The choice of notation here is simply a visual
    aid to distinguish between the graph-theoretic
    and algebraic setting.
}
and the set of rules

\begin{restoremargins}
\begin{equation}
\label{eq:rules-graph:sig:article-graph-raj-dahya}
\everymath={\displaystyle}
\begin{array}[m]{rcl}
    R_{\mathcal{G}}
    &\coloneqq
    &\{
        (
            (\letter{(u,v)}, \letter{(v,w)}),
            \letter{(u,w)}
        )
        \mid
        (u, v, w) \in \Omega^{3},
        (u, v), (v, w), (u, w)
        \in \breve{E}
    \}\\
    &&\cup
    \{
        (\letter{(u,u)}, 1)
        \mid
        u \in \Omega
    \}
\end{array}
\end{equation}
\end{restoremargins}

\continueparagraph
as the \highlightTerm{edge reduction rules}.
Consider now the string-rewriting system
$(\Gamma_{\mathcal{G}}^{\ast}, \to_{R_{\mathcal{G}}})$
associated to
$(\Gamma_{\mathcal{G}}, R_{\mathcal{G}})$,
as defined in \eqcref{eq:context-free:sig:article-graph-raj-dahya}.
For brevity, we shall write $\to_{\mathcal{G}}$ instead of $\to_{R_{\mathcal{G}}}$
and
    $[x]_{\mathcal{G}}$
    or simply $[x]$
for the equivalence class of each word $x \in \Gamma_{\mathcal{G}}^{\ast}$
in $(\Gamma_{\mathcal{G}}^{\ast}, \overset{\ast}{\leftrightarrow}_{\mathcal{G}})$.

\begin{prop}
\makelabel{prop:rew-sys-on-graph-is-noeth+conf:sig:article-graph-raj-dahya}
    The rewriting-system
    $(\Gamma_{\mathcal{G}}^{\ast}, \to_{\mathcal{G}})$
    is Noetherian and confluent.
\end{prop}

    \begin{proof}
        Since $R_{\mathcal{G}}$
        is a subset of
            $
                (\Gamma_{\mathcal{G}}^{1})^{1} \times \Gamma_{\mathcal{G}}^{0}
                \cup
                (\Gamma_{\mathcal{G}}^{1})^{2} \times \Gamma_{\mathcal{G}}^{1}
            $,
        by \Cref{prop:rew-sys-convergent:sig:article-graph-raj-dahya}
        it suffices to prove that
            $R_{\mathcal{G}}$
        satisfies the pre-algebraic axioms
        \eqcref{ax:pre-alg:neutral:sig:article-graph-raj-dahya}
        and
        \eqcref{ax:pre-alg:associative:sig:article-graph-raj-dahya}
        in \S{}\ref{sec:algebra:defn:rewrite:sig:article-graph-raj-dahya}.

        Towards \eqcref{ax:pre-alg:neutral:sig:article-graph-raj-dahya},
        let $((x, y), z) \in R_{\mathcal{G}}$.
        Then there exist
            $(u, v, w) \in \Omega^{3}$
        such that
            $(u, v), (v, w), (u, w) \in \breve{E}$
        and
            $x = \letter{(u, v)}$,
            $y = \letter{(v, w)}$,
            and
            $z = \letter{(u, w)}$.
        If $(x, 1) \in R_{\mathcal{G}}$.
        then there exists $s \in \Omega$
        such that $x = \letter{(s, s)}$,
        whence $u = s = v$,
        which implies
            $z = \letter{(u, w)} = \letter{(v, w)} = y$.
        And if $(y, 1) \in R_{\mathcal{G}}$,
        then there exists $s \in \Omega$
        such that $y = \letter{(s, s)}$,
        whence $v = s = w$,
        which implies
            $z = \letter{(u, w)} = \letter{(u, v)} = x$.
        Thus \eqcref{ax:pre-alg:neutral:sig:article-graph-raj-dahya}
        is fulfilled by $R_{\mathcal{G}}$.

        Towards \eqcref{ax:pre-alg:associative:sig:article-graph-raj-dahya},
        let
            $((x, y), z), ((x', y'), z') \in R_{\mathcal{G}}$
        with $y = x'$.
        Then there exist
            $(u, v, w), (u', v', w') \in \Omega^{3}$
        such that
            $(u, v), (v, w), (u, w), (u', v'), (v', w'), (u', w') \in \breve{E}$
        and
            $x = \letter{(u, v)}$,
            $y = \letter{(v, w)}$,
            $z = \letter{(u, w)}$,
            $x' = \letter{(u', v')}$,
            $y' = \letter{(v', w')}$,
            and
            $z' = \letter{(u', w')}$.
        Since $y = x'$
        one has $(v, w) = (u', v')$
        and thus
            $
                z'
                = \letter{(u', w')}
                = \letter{(v, w')}
            $
        and
            $
                y'
                = \letter{(v', w')}
                = \letter{(w, w')}
            $.
        Since $\breve{E}$ is transitive
        and contains
            $(u, v)$ and $(u', w') = (v, w')$,
        one has that $(u, w') \in \breve{E}$.
        By construction, it follows that
            $
                ((x, z'), \letter{(u, w')})
                = ((\letter{(u, v)}, \letter{(v, w')}), \letter{(u, w')})
            $
            and
            $
                ((z, y'), \letter{(u, w')})
                = ((\letter{(u, w)}, \letter{(w, w')}), \letter{(u, w')})
            $
        are rules in $R_{\mathcal{G}}$.
        Thus \eqcref{ax:pre-alg:associative:sig:article-graph-raj-dahya}
        is fulfilled by $R_{\mathcal{G}}$.
    \end{proof}

%% ********** END OF FILE: body/sec-3-algebra/sec-2-para-1-rewrite.tex **********

%% ********************************************************************************
%% FILE: body/sec-3-algebra/sec-2-para-2-free.tex
%% ********************************************************************************

\subsubsection[Groups associated to graphs]{Groups associated to graphs}
\label{sec:algebra:graphs:free:sig:article-graph-raj-dahya}

\firstparagraph
Consider now the monoid
$
    \Generate[Mon]{\Gamma_{\mathcal{G}}}{R_{\mathcal{G}}}
$
associated to
    $(\Gamma_{\mathcal{G}}, R_{\mathcal{G}})$,
as defined in \S{}\ref{sec:algebra:defn:free:sig:article-graph-raj-dahya}.
Due to the epimorphism in \eqcref{eq:epimorphism-free-monoid:sig:article-graph-raj-dahya},
since the monoid
    $\Gamma_{\mathcal{G}}^{\ast}$
is generated by the subset
    $\Gamma_{\mathcal{G}}$,
the monoid
    $\Generate[Mon]{\Gamma_{\mathcal{G}}}{R_{\mathcal{G}}}$
is generated by the subset
    $
        S \coloneqq \{
            [x]
            \mid
            x \in \Gamma_{\mathcal{G}}
        \}
    $.
Let $x \in \Gamma_{\mathcal{G}}$ be arbitrary.
Then $x = \letter{(u, v)}$ for some edge $(u, v) \in \breve{E}$.
Since $\breve{E}$ is an equivalence relation on $\Omega$,
$(v, u), (u, u), (v, v) \in \breve{E}$.
Setting
    $x' \coloneqq \letter{(v, u)}$,
one has that
    $((x', x), \letter{(v, v)})$,
    $((x, x'), \letter{(u, u)})$,
    $(\letter{(u, u)}, 1)$,
    and
    $(\letter{(v, v)}, 1)$
are rules in $R_{\mathcal{G}}$.
Thus
    $
        x' x
        \to_{\mathcal{G}}
        \letter{(v, v)}
        \to_{\mathcal{G}}
        1
    $
and
    $
        x x'
        \to_{\mathcal{G}}
        \letter{(u, u)}
        \to_{\mathcal{G}}
        1
    $.
So
    $
        [x'] [x]
        = [1]
        = 1
    $
and
    $
        [x] [x']
        = [1]
        = 1
    $.
Hence every element in
    $S$
has an inverse.
It follows that
    $G_{\mathcal{G}} \coloneqq \Generate[Mon]{\Gamma_{\mathcal{G}}}{R_{\mathcal{G}}} = \Generate{S}$
is a group,
which we shall refer to as the
\highlightTerm{group associated with the edges of the graph $\mathcal{G}$}
or simply the \highlightTerm{edge group}.

\begin{rem}
    Note that in \Cref{prop:rew-sys-on-graph-is-noeth+conf:sig:article-graph-raj-dahya}
    one could replace $\breve{E}$
    by the transitive hull of $E \cup \Gph{\id_{\Omega}}$.
    We only needed $\breve{E}$ to be symmetric,
    in order for the associated monoid
    $\Generate[Mon]{\Gamma_{\mathcal{G}}}{R_{\mathcal{G}}}$
    to constitute a group.
\end{rem}

%% ********** END OF FILE: body/sec-3-algebra/sec-2-para-2-free.tex **********

%% ********************************************************************************
%% FILE: body/sec-3-algebra/sec-2-para-3-normal.tex
%% ********************************************************************************

\subsubsection[Normal forms]{Normal forms}
\label{sec:algebra:graphs:normal:sig:article-graph-raj-dahya}

\firstparagraph
Now since by \Cref{prop:rew-sys-on-graph-is-noeth+conf:sig:article-graph-raj-dahya}
the rewriting-system
    $(\Gamma_{\mathcal{G}}^{\ast}, \to_{\mathcal{G}})$
is Noetherian and confluent,
by \Cref{prop:rs-normal-form:sig:article-graph-raj-dahya}
for each word $x \in \Gamma_{\mathcal{G}}^{\ast}$,
there exists a unique normal form
    $\hat{x} \in \Irred{\Gamma_{\mathcal{G}}^{\ast}}$,
\idest an irreducible element,
such that $x' \overset{\ast}{\to}_{\mathcal{G}} \hat{x}$
for all $x' \in [x]$.
The map

\begin{restoremargins}
\begin{equation}
\label{eq:graph-normal-form:sig:article-graph-raj-dahya}
\everymath={\displaystyle}
\begin{array}[m]{rcccl}
    N_{\mathcal{G}}
        &: &G_{\mathcal{G}} &\to &\Irred{\Gamma_{\mathcal{G}}^{\ast}}\\
        &&[x]
            &\mapsto
            &\text{the}~\hat{x}~\text{with $x \overset{\ast}{\to}_{\mathcal{G}} \hat{x}$}
\end{array}
\end{equation}
\end{restoremargins}

\continueparagraph
is thus well-defined.
Observe in particular, that $N_{\mathcal{G}}$ is a \highlightTerm{choice function},
\idest $\hat{x} = N_{\mathcal{G}}([x]) \in [x]$
for all $x \in \Gamma_{\mathcal{G}}^{\ast}$.
Using this map we may prove the following:

\begin{prop}
\makelabel{prop:normal-form-for-single-letters-in-graph-alphabet:sig:article-graph-raj-dahya}
    Let ${\iota : \breve{E} \to G_{\mathcal{G}}}$
    be defined by
    $\iota(u, v) = [\letter{(u, v)}]$
    for $(u, v) \in \breve{E}$.
    Then
        $\iota(u, u) = [1]$
        and
        $N_{\mathcal{G}}(\iota(u, u)) = 1$
        for $u \in \Omega$
    and
        $N_{\mathcal{G}}(\iota(u, v)) = \letter{(u, v)}$
    for all $(u, v) \in \breve{E} \without \Gph{\id_{\Omega}}$.
    In particular,
        $\iota\restr{\breve{E} \without \Gph{\id_{\Omega}}}$
    is injective.
\end{prop}

    \begin{proof}
        For $u \in \Omega$
        it holds that
        $\letter{(u, u)} \to_{\mathcal{G}} 1$
        and thus
        $
            \iota(u, u)
            = [\letter{(u, u)}]
            = [1]
        $.
        Since $1 \in \Gamma_{\mathcal{G}}^{\ast}$
        is clearly irreducible,
        it follows that
            $
                N_{\mathcal{G}}(\iota((u, u)))
                = N_{\mathcal{G}}([1])
                = 1
            $.
        Now consider $(u, v) \in \breve{E} \without \Gph{\id_{\Omega}}$
        and let
            $
                x \coloneqq [\letter{(u, v)}]
            $.
        Using the normal form we have
            $
                \letter{(u, v)}
                \overset{\ast}{\to}_{\mathcal{G}} N_{\mathcal{G}}(x)
            $.
        Since the word $\letter{(u, v)}$ consists of a single letter,
        the only reductions under the context-free grammar
        induced by $\to_{\mathcal{G}}$
        are
            $\letter{(u, v)} \overset{0}{\to}_{\mathcal{G}} \letter{(u, v)}$
            and
            $\letter{(u, v)} \overset{1}{\to}_{\mathcal{G}} 1$.
        By construction of $R_{\mathcal{G}}$,
        the latter is only possible if $u = v$.
        Since $(u, v) \notin \Gph{\id_{\Omega}}$,
        it follows that $N_{\mathcal{G}}(x) = \letter{(u, v)}$.
    \end{proof}

For further work, it shall be useful to provide exact descriptions of irreducible elements.

\begin{defn}
    For $n\in\naturalsZero$
    say that a sequence of pairs
        $\{(u_{i},v_{i})\}_{i=1}^{n} \subseteq \Omega \times \Omega$,
    is \highlightTerm{non-coalescent}
    if
        $u_{i} \neq v_{i}$
    for $i \in \{1,2,\ldots,n\}$
    and
        $v_{i} \neq u_{i + 1}$
    for all $i \in \{1,2,\ldots,n-1\}$.
\end{defn}

Working through the construction of
the rules $R_{\mathcal{G}}$
it is a simple exercise to obtain:

\begin{prop}
\makelabel{prop:graph-form-irred-elements:sig:article-graph-raj-dahya}
    Let $x \in \Gamma_{\mathcal{G}}^{\ast}$.
    Then $x$ is irreducible if and only if
        $
            x = \prod_{i=1}^{n}
                \letter{(u_{i},v_{i})}
        $
    for some $n\in\naturalsZero$
    and a sequence of non-coalescent edges
        $\{(u_{i},v_{i})\}_{i=1}^{n} \subseteq \breve{E} \without \Gph{\id_{\Omega}}$.
\end{prop}

%% ********** END OF FILE: body/sec-3-algebra/sec-2-para-3-normal.tex **********

%% ********** END OF FILE: body/sec-3-algebra/sec-2-graphs.tex **********

%% ********************************************************************************
%% FILE: body/sec-3-algebra/sec-3-continuity.tex
%% ********************************************************************************

\subsection[Continuity]{Continuity}
\label{sec:algebra:continuity:sig:article-graph-raj-dahya}

\firstparagraph
Finally, we provide a notion of continuity,
which shall later be instrumental to obtain continuous dilations.
Let $G$ be a group, $E$ a topological space,
and ${\iota : E \to G}$ an arbitrary map.

\begin{defn}[Embedded uniform continuity]
\makelabel{defn:embedded-cts:sig:article-graph-raj-dahya}
    Let
        $\{\psi(g)\}_{g \in G} \subseteq \BoundedOps{\BanachRaum}$
    be a family of operators on a Banach space $\BanachRaum$.
    We shall say that $\psi$ has \highlightTerm{embedded uniform strong continuity}
    \wrt the map $\iota$ if

        \begin{restoremargins}
        \begin{equation}
        \label{eq:embedded-sot-continuity:sig:article-graph-raj-dahya}
            \sup_{g, h \in G_{\mathcal{G}}}
                \norm{
                    (
                        \psi(g\:\iota(e')\:h)
                        -
                        \psi(g\:\iota(e)\:h)
                    )
                    \xi
                }
            \longrightarrow 0
        \end{equation}
        \end{restoremargins}

    \continueparagraph
    for ${E \ni e' \longrightarrow e}$
    and for each $e \in E$
    and $\xi \in \BanachRaum$.
\end{defn}

In the context of the present paper,
we shall take $E$ to be the set of edges of a graph,
whose set of nodes $\Omega$ forms a topological space,
and $\iota$ shall be the embedding
${e \mapsto [\letter{e}] \in G_{(\Omega,E)}}$.

%% ********** END OF FILE: body/sec-3-algebra/sec-3-continuity.tex **********

%% ********** END OF FILE: body/sec-3-algebra/.index.tex **********

%% ********************************************************************************
%% FILE: body/sec-4-extensions/.index.tex
%% ********************************************************************************

\section[Algebraic extensions of systems on graphs]{Algebraic extensions of systems on graphs}
\label{sec:ext:sig:article-graph-raj-dahya}

\firstparagraph
We now focus on
\ref{goal:ext:sig:article-graph-raj-dahya}
and
\ref{goal:cts:sig:article-graph-raj-dahya}
expressed at the start of this paper.
Let
    $\{\phi(u,v)\}_{(u,v) \in E} \subseteq \BoundedOps{\BanachRaum}$
a family of operators on the edges $E$
of a graph $\mathcal{G} = (\Omega, E)$
where $\BanachRaum$ is a Banach space.
As in the previous section,
we consider the equivalence relation
    $\breve{E} = \Generate{E \cup E^{-1} \cup \Gph{\id_{\Omega}}}$
generated by $E$,
the edge alphabet
    $\Gamma_{\mathcal{G}} = \{\letter{(u, v)} \mid (u, v) \in \breve{E}\}$
defined in \eqcref{eq:alphabet-graph:sig:article-graph-raj-dahya},
the reduction rules
    $
        R_{\mathcal{G}}
        \subseteq
        (\Gamma_{\mathcal{G}}^{\ast})^{\ast} \times \Gamma_{\mathcal{G}}^{\ast}
    $
defined in \eqcref{eq:rules-graph:sig:article-graph-raj-dahya},
and
the edge group
    $
        G_{\mathcal{G}}
        = \Gamma_{\mathcal{G}}^{\ast}/\overset{\ast}{\leftrightarrow}_{\mathcal{G}}
        = \{
            [x]
            \mid
            x \in \Gamma_{\mathcal{G}}^{\ast}
        \}
    $
associated with the graph $\mathcal{G}$,
where
    $
        [x]
        = [x]_{\mathcal{G}}
        = \{
            x' \in \Gamma_{\mathcal{G}}^{\ast}
            \mid
            x \overset{\ast}{\leftrightarrow}_{\mathcal{G}} x'
        \}
    $
for $x \in \Gamma_{\mathcal{G}}^{\ast}$.
Our goal is to find suitable operator families on $G_{\mathcal{G}}$
which are in some sense compatible with $\phi$.

%% ********************************************************************************
%% FILE: body/sec-4-extensions/sec-1-general.tex
%% ********************************************************************************

\subsection[Discrete extensions for general graphs]{Discrete extensions for general graphs}
\label{sec:ext:general:sig:article-graph-raj-dahya}

\firstparagraph
We first consider the case of operator families defined on general graphs
with no particular properties.
Observing that
    $(\BoundedOps{\BanachRaum},\circ,\onematrix)$
is a monoid,
we obtain the following abstract results.

\begin{highlightboxWithBreaks}
\begin{lemm}
\makelabel{lemm:extension-normal-form:sig:article-graph-raj-dahya}
    Let $(M,\cdot,1)$ be any monoid,
    \exempli the set of contractions on a Banach space under operator multiplication.
    Let ${\phi : E \to M}$
    satisfy the identity axiom \eqcref{ax:dyn:id:sig:article-graph-raj-dahya},
    \idest $\phi(u, u) = 1$
    for $u \in \Omega$ provided $(u, u) \in E$.
    Then there exists a unique map
        ${\quer{\phi} : G_{\mathcal{G}} \to M}$
    with $\ran(\quer{\phi}) \subseteq \Generate{\ran(\phi) \cup \{1\}}$,%
    \footnoteref{ft:algebraic-generator:sig:article-graph-raj-dahya}
    satisfying
        $\quer{\phi}(1) = 1$
        and
        $\quer{\phi}([\letter{(u, v)}]) = \phi(u, v)$
    for all $(u, v) \in E$.
\end{lemm}
\end{highlightboxWithBreaks}

\footnotetext[ft:algebraic-generator:sig:article-graph-raj-dahya]{%
    $\Generate{A}$ denotes the closure of a subset $A \subseteq M$ under multiplication.
}

    \begin{proof}
        We first extend $\phi$
        to the equivalence relation $\breve{E}$ generated by the edge set $E$.
        Define
            $\phi_{0} : \breve{E} \to M$
        via
            $\phi_{0}(u, v) \coloneqq \phi(u, v)$
        for $(u, v) \in E$
        and
            $\phi_{0}(u, v) \coloneqq 1$
            for $(u, v) \in \breve{E} \without E$.
        By construction and the assumptions on $\phi$,
        one has that
            $((\Omega, \breve{E}),\phi_{0})$
        satisfies
        the identity axiom \eqcref{ax:dyn:id:sig:article-graph-raj-dahya}.
        We now proceed to construct $\quer{\phi}$.

        \paragraph{Existence:}
        Recall that the map
            ${N_{\mathcal{G}} : G_{\mathcal{G}} \to \Gamma_{\mathcal{G}}^{\ast}}$
        defined in
            \eqcref{eq:graph-normal-form:sig:article-graph-raj-dahya},
        which associates to each equivalence class a unique normal form,
        is a choice function.
        Let $x \in \Gamma_{\mathcal{G}}^{\ast}$ be arbitrary.
        Then
            $\hat{x} \coloneqq N_{\mathcal{G}}([x]) \in [x]$
        is irreducible.
        By \Cref{prop:graph-form-irred-elements:sig:article-graph-raj-dahya},
        the word $\hat{x}$ may be uniquely written
        as a product of letters
            $
                \prod_{i=1}^{n}
                    \letter{(u_{i}, v_{i})}
            $,
        where
            $n\in\naturalsZero$
            and $\{(u_{i},v_{i})\}_{i=1}^{n} \subseteq \breve{E} \without \Gph{\id_{\Omega}}$
        is a non-coalescent sequence.
        Finally, we define

            \begin{restoremargins}
            \begin{equation}
            \label{eq:extension-normal-form:sig:article-graph-raj-dahya}
                \quer{\phi}([x])
                    \coloneqq
                        \prod_{i=1}^{n}
                            \phi_{0}(u_{i}, v_{i}),
            \end{equation}
            \end{restoremargins}

        \continueparagraph
        which in lieu of the choice function
        makes $\quer{\phi}$ a well-defined map
        from $G_{\mathcal{G}}$ into $M$.
        By construction, one clearly has
            $
                \ran(\phi)
                \subseteq
                    \Generate{\ran(\phi_{0})}
                = \Generate{\ran(\phi) \cup \{1\}}
            $.

        \paragraph{Properties:}
        Let $(u, v) \in E$ be arbitrary
        and let $x \coloneqq \letter{(u, v)}$.
        Write
            $\hat{x} \coloneqq N_{\mathcal{G}}([x])$
        as a product of letters
            $
                \prod_{i=1}^{n}
                    \letter{(u_{i}, v_{i})}
            $,
        where
            $n\in\naturalsZero$
            and $\{(u_{i},v_{i})\}_{i=1}^{n} \subseteq \breve{E} \without \Gph{\id_{\Omega}}$
        is a non-coalescent sequence.
        If $u = v$,
        then by
            \Cref{prop:normal-form-for-single-letters-in-graph-alphabet:sig:article-graph-raj-dahya},
        we have
            $
                \hat{x}
                = N_{\mathcal{G}}([x])
                = 1
                \in \Gamma_{\mathcal{G}}^{\ast}
            $.
        So $n = 0$
        and the right hand side of \eqcref{eq:extension-normal-form:sig:article-graph-raj-dahya}
        is empty, whence
            $
                \quer{\phi}([x])
                = 1
                = \phi_{0}(u, v)
                = \phi(u, v)
            $.
        If $u \neq v$,
        then by
            \Cref{prop:normal-form-for-single-letters-in-graph-alphabet:sig:article-graph-raj-dahya},
        we have
            $
                \hat{x}
                = N_{\mathcal{G}}([x])
                = \letter{(u, v)}
                \in \Gamma_{\mathcal{G}}^{\ast}
            $,
        so that $n = 1$ and $(u_{1}, v_{1}) = (u, v) \in E$.
        By \eqcref{eq:extension-normal-form:sig:article-graph-raj-dahya},
        we thus obtain
            $
                \quer{\phi}([x])
                = \phi_{0}(u_{1}, v_{1})
                = \phi_{0}(u, v)
                = \phi(u, v)
            $.
        Choosing any $u \in \Omega$,
        one has
            $
                \quer{\phi}(1)
                = \quer{\phi}([1])
                = \quer{\phi}([(u, u)])
                = \phi_{0}(u, u)
                = 1
            $,
        since $\phi_{0}$ satisfies
        the identity axiom \eqcref{ax:dyn:id:sig:article-graph-raj-dahya}.
    \end{proof}

We shall refer to the construction in
\eqcref{eq:extension-normal-form:sig:article-graph-raj-dahya}
as the \highlightTerm{normal form extension}.

%% ********** END OF FILE: body/sec-4-extensions/sec-1-general.tex **********

%% ********************************************************************************
%% FILE: body/sec-4-extensions/sec-2-lo-divisible.tex
%% ********************************************************************************

\subsection[Continuous extensions for divisible systems]{Continuous extensions for divisible systems}
\label{sec:ext:lo:divisible:sig:article-graph-raj-dahya}

\firstparagraph
To make progress towards \ref{goal:cts:sig:article-graph-raj-dahya}
expressed at the start of this paper,
we now consider the case that the set of nodes $\Omega$ in the graph is topologised.
We seek sufficient conditions, under which the notion of continuity
presented in \Cref{defn:embedded-cts:sig:article-graph-raj-dahya}
holds for an extension.%
\footnote{%
    in \S{}\ref{sec:dilation:stroescu:cts:sig:article-graph-raj-dahya}
    this shall prove sufficient to establish continuity of dilations.
}
We immediately note that the normal form extension
    \eqcref{eq:extension-normal-form:sig:article-graph-raj-dahya}
constructed in
    \Cref{lemm:extension-normal-form:sig:article-graph-raj-dahya}
does not appear to be a promising candidate.
So our first task shall be to provide a new extension.

In order to achieve this, we narrow down further
and consider graphs $\mathcal{G} = (\Omega, E)$
for which $E$ is a reflexive linear ordering on $\Omega$.
As before we let $\mathrel{\prec}$
denote the irreflexive part of $E$,
so that $E = \{(u, v) \in \Omega \mid u \preceq v\}$.
Given such a linear ordering, we endow $\Omega$
with the order topology.%
\footnote{
    recall that this is the topology
    generated by sets of the form
        $\{u \in \Omega \mid u \prec v\}$
        and
        $\{u \in \Omega \mid v \prec u\}$
    for $v \in \Omega$
    (\cf \S{}\ref{sec:examples:hamiltonian:sig:article-graph-raj-dahya}).
}
Observe that since $E$ is a linear ordering,
the equivalence relation generated by $E$
is $\breve{E} = \Omega \times \Omega$.

%% ********************************************************************************
%% FILE: body/sec-4-extensions/sec-2-para-1-cover.tex
%% ********************************************************************************

Before establishing a new extension result for operator families defined on such graphs,
we need the following technical means:

\begin{enumerate}[
    label={\bfseries Cov\textsubscript{\arabic*})},
    ref={\bfseries Cov\textsubscript{\arabic*}},
]
    \item\label{it:cover:intervals:sig:article-graph-raj-dahya}
        For $u, v \in \Omega$ define
        $[u) \coloneqq \{u' \in \Omega \mid u' \prec u\}$
        and
        $
            [u,\:v)
            \coloneqq [v) \without [u)
            = \{u' \in \Omega \mid u \preceq u' \prec v\}
        $,
        which of course is only non-empty if $u \prec v$.

    \item\label{it:cover:atoms:sig:article-graph-raj-dahya}
        The map
            ${\onefct_{A} : \Omega \to \{0, 1\} \subseteq \reals}$
        denotes the indicator function
        of a subset $A \subseteq \Omega$.
        For $u, v \in \Omega$,
        we define
            $
                \cover_{u,v}
                \coloneqq \onefct_{[v)} - \onefct_{[u)}
            $.

    \item\label{it:cover:words:sig:article-graph-raj-dahya}
        Let $x \in \Gamma_{\mathcal{G}}^{\ast}$,
        which can be (uniquely) written as
            $x = \prod_{i=1}^{n}\letter{(u_{i}, v_{i})}$,
        for some $n \in \naturalsZero$
        and $
                \{(u_{i}, v_{i})\}_{i=1}^{n}
                \subseteq
                \quer{E}
                = \Omega \times \Omega
            $.
        Set

            \begin{restoremargins}
            \begin{equation}
            \label{eq:cover-of-element:sig:article-graph-raj-dahya}
                \cover_{x}
                \coloneqq
                \sum_{i=1}^{n}
                    \cover_{u_{i},v_{i}},
            \end{equation}
            \end{restoremargins}

        \continueparagraph
        which in general is a bounded map from $\Omega$ to (a finite subset of) $\integers$.
        We refer to $\cover_{x}$ as the \highlightTerm{cover} of $x$.
        By the above construction one has
            $\cover_{1} \equiv 0$,
            $
                \cover_{y y'}
                = \cover_{y} + \cover_{y'}
            $
            and thus
            $
                \cover_{y y'} = \cover_{y' y}
            $
            for all $y,y' \in \Gamma_{\mathcal{G}}^{\ast}$.

    \item\label{it:cover:support:sig:article-graph-raj-dahya}
        Continuing with the notation in \eqcref{it:cover:words:sig:article-graph-raj-dahya},
        we further define
            $\supp(x) \coloneqq \{u \in \Omega \mid \cover_{x}(u) > 0\}$
        to be the \highlightTerm{positive support of $x$}.

    \item\label{it:cover:simple-fct:sig:article-graph-raj-dahya}
        Let $x$, $n$, $\{(u_{i}, v_{i})\}_{i=1}^{n}$
        be as in \eqcref{it:cover:words:sig:article-graph-raj-dahya}.
        If $n = 0$, \idest $x = 1$,
        observe that
            $
                \cover_{x}
                = 0 \cdot \onefct_{[w_{0}, w_{1})}
            $
        for any elements $w_{0},w_{1}\in\Omega$
        with $w_{0} \preceq w_{1}$.
        Otherwise, by ordering the elements
            $u_{1},v_{1},u_{2},v_{2},\ldots,u_{n},v_{n}$
        \wrt $(\Omega,\prec)$,
        and without removing duplicates,
        one can construct
        a finite sequence
            $\{w_{k}\}_{k=0}^{m} \subseteq \Omega$
        with
            $m = 2n - 1$
        and
            $w_{0} \preceq w_{1} \preceq \ldots \preceq w_{m}$
        such that
            $
                \{w_{k} \mid k \in \{0,1,\ldots,m\}\}
                =
                \{u_{1},v_{1},u_{2},v_{2},\ldots,u_{n},v_{n}\}
            $.
        Let $i \in \{1,2,\ldots,n\}$ be arbitrary.
        Set $c_{i} \in \{0, 1, -1\}$
        depending on whether $u_{i} = v_{i}$, $u_{i} \prec v_{i}$, or $v_{i} \prec u_{i}$.
        By construction of the monotone sequence $\{w_{k}\}_{k=0}^{m}$,
        one can find
            $l_{i}, r_{i} \in \{0,1,\ldots,m\}$
        with $l_{i} < r_{i}$
        such that
            $w_{l_{i}} = \min\{u_{i},v_{i}\}$
            and
            $w_{r_{i}} = \max\{u_{i},v_{i}\}$.
        Using this one obtains

            \begin{restoremargins}
            \begin{equation}
            \label{eq:cover-fct-as-simple-fct:sig:article-graph-raj-dahya}
            \everymath={\displaystyle}
            \begin{array}[m]{rcl}
                \cover_{x}
                &=
                    &\sum_{i=1}^{n}
                        \onefct_{[v_{i})}
                        -
                        \onefct_{[u_{i})}
                    \\
                &=
                    &\sum_{i=1}^{n}
                        c_{i}
                        \cdot
                        (
                            \onefct_{[\max\{u_{i},v_{i}\})}
                            -
                            \onefct_{[\min\{u_{i},v_{i}\})}
                        )
                        \\
                &=
                    &\sum_{i=1}^{n}
                        c_{i}
                        \cdot
                        (
                            \onefct_{[w_{r_{i}})}
                            -
                            \onefct_{[w_{l_{i}})}
                        )
                        \\
                &=
                    &\sum_{i=1}^{n}
                        c_{i}
                        \cdot
                        \sum_{k=l_{i}+1}^{r_{i}}
                        (
                            \onefct_{[w_{k})}
                            -
                            \onefct_{[w_{k-1})}
                        )
                        \\
                &=
                    &\sum_{k=1}^{m}
                        \tilde{c}_{k}
                        \onefct_{[w_{k-1},\:w_{k})}
            \end{array}
            \end{equation}
            \end{restoremargins}

        \continueparagraph
        where
        $
            \tilde{c}_{k}
            \coloneqq
            \sum_{
                i \in \{1,2,\ldots,n\}:
                ~l_{i} < k \leq r_{i}
            }
                c_{i}
            \in \integers
        $.
        Since the sequence $\{w_{k}\}_{k=0}^{m} \subseteq \Omega$
        is monotone,
        the intervals in
            $\{[w_{k-1},\:w_{k})\}_{k=1}^{m}$
        are pairwise disjoint.

    \item\label{it:cover:equivalence-classes:sig:article-graph-raj-dahya}
        For $u \in \Omega$
        one has
            $\cover_{u, u} \equiv 0$,
        and thus
            $
                \cover_{\letter{(u, u)}}
                = \cover_{u, u}
                = \cover_{1}
            $.
        And for $u, v, w \in \Omega$
        one has
            $
                \cover_{\letter{(u, v)}\letter{(v, w)}}
                = \cover_{u, v} + \cover_{v, w}
                = (\onefct_{[v)} - \onefct_{[u)}) + (\onefct_{[w)} - \onefct_{[v)})
                = \onefct_{[w)} - \onefct_{[u)}
                = \cover_{u, w}
                = \cover_{\letter{(u, w)}}
            $.
        Working with the reduction rules
        defined in \eqcref{eq:rules-graph:sig:article-graph-raj-dahya}
        and the context-free grammar for string-rewriting systems
        defined in \eqcref{eq:context-free:sig:article-graph-raj-dahya},
        it follows that
            $\cover_{x} = \cover_{x'}$
        for all $x, x' \in \Gamma_{\mathcal{G}}^{\ast}$
        with $x \overset{\ast}{\to}_{\mathcal{G}} x'$.
        By a simple induction it thus follows that
            $\cover_{x} = \cover_{x'}$
        for all $x, x' \in \Gamma_{\mathcal{G}}^{\ast}$
        with $x \overset{\ast}{\leftrightarrow}_{\mathcal{G}} x'$.
        We may thus define
            $\cover_{[x]} \coloneqq \cover_{x}$
        and
            $\supp([x]) \coloneqq \supp(x)$
        for $x \in \Gamma_{\mathcal{G}}^{\ast}$.
\end{enumerate}

%% ********** END OF FILE: body/sec-4-extensions/sec-2-para-1-cover.tex **********

%% ********************************************************************************
%% FILE: body/sec-4-extensions/sec-2-para-2-extension.tex
%% ********************************************************************************

Using these tools we can obtain the following:

\begin{highlightboxWithBreaks}
\begin{lemm}[\First cover extension for linearly ordered graphs]
\makelabel{lemm:extension-cover:I:sig:article-graph-raj-dahya}
    Let $M$ be a monoid
    and $\mathcal{G} = (\Omega, E)$
    be a graph where $E$ is a reflexive linear ordering.
    Let ${\phi : E \to M}$
    satisfy the identity axiom \eqcref{ax:dyn:id:sig:article-graph-raj-dahya}
    and divisibility axiom \eqcref{ax:dyn:div:sig:article-graph-raj-dahya}.
    Then there exists a map
        ${\quer{\phi} : G_{\mathcal{G}} \to M}$
    with $\ran(\quer{\phi}) \subseteq \Generate{\ran(\phi)}$,%
    \footnoteref{ft:algebraic-generator:sig:article-graph-raj-dahya}
    satisfying
        $\quer{\phi}(1) = 1$
        and
        $\quer{\phi}([\letter{(u, v)}]) = \phi(u, v)$
    for all $(u, v) \in E$.
    Moreover
        $\quer{\phi}$ has cyclic invariance,
    \idest
        $\quer{\phi}(gh) = \quer{\phi}(hg)$
    for all $g,h \in G_{\mathcal{G}}$.
\end{lemm}
\end{highlightboxWithBreaks}

    \begin{proof}
        \paragraph{Existence:}
        Let $x \in \Gamma_{\mathcal{G}}^{\ast}$ be an arbitrary word.
        By \eqcref{eq:cover-fct-as-simple-fct:sig:article-graph-raj-dahya}
        and \eqcref{it:cover:equivalence-classes:sig:article-graph-raj-dahya} above
        we can write

            \begin{restoremargins}
            \begin{equation}
            \label{eq:0:\beweislabel}
                \cover_{[x]}
                = \sum_{i=1}^{m}
                    c_{i}\onefct_{[w_{i-1},\:w_{i})}
            \end{equation}
            \end{restoremargins}

        \continueparagraph
        for some $m\in\naturals$,
        some monotone sequence
            $\{w_{i}\}_{i=0}^{m} \subseteq \Omega$,
        some integer sequence
            $\{c_{i}\}_{i=1}^{m} \subseteq \integers$.
        Given such an expression, we wish to define

            \begin{restoremargins}
            \begin{equation}
            \label{eq:extension-cover:I:sig:article-graph-raj-dahya}
                \quer{\phi}([x])
                \coloneqq
                \prod_{\substack{i=1,\\c_{i} > 0}}^{m}
                    \phi(w_{i-1}, w_{i})
                =
                \prod_{i=1}^{m}
                    \phi(w_{i-1}, w_{i})^{\chi_{i}}
            \end{equation}
            \end{restoremargins}

        \continueparagraph
        where each $\chi_{i} \in \{0,1\}$
        indicates whether $c_{i} > 0$.
        Our goal is to show that \eqcref{eq:extension-cover:I:sig:article-graph-raj-dahya}
        is independent of the exact form of the expression in \eqcref{eq:0:\beweislabel}.

        To this end, let
            $\{w'_{j}\}_{j=0}^{n} \subseteq \Omega$
        be monotone and
            $\{c'_{i}\}_{i=1}^{n} \subseteq \integers$
        be any sequence for which
            $
                \cover_{[x]}
                = \sum_{j=1}^{n}
                    c'_{j}\onefct_{[w_{j-1},\:w_{j})}
            $.
        By ordering the elements
            $w_{0},w_{1},\ldots,w_{m},w'_{0},w'_{1},\ldots,w'_{n}$
        and removing duplicates,
        we can find a strictly monotone sequence
            $\{w''_{k}\}_{k=0}^{l} \subseteq \Omega$
        as well as monotone sequences
            $0 \leq k_{0} \leq k_{1} \leq \ldots \leq k_{m} \leq l$
        and
            $0 \leq k'_{0} \leq k'_{1} \leq \ldots \leq k'_{n} \leq l$
        such that
            $w_{i} = w''_{k_{i}}$
            for all $i \in \{0,1,\ldots,m\}$
            and
            $w'_{j} = w''_{k'_{j}}$
            for all $j \in \{0,1,\ldots,n\}$.
        Letting

            \begin{restoremargins}
            \begin{equation}
            \label{eq:1:\beweislabel}
            \everymath={\displaystyle}
            \begin{array}[m]{rcl}
                I
                    &\coloneqq
                    &\bigcup_{\substack{i=1,\\c_{i} > 0}}^{m}
                        \{k_{i-1}+1,\ldots,k_{i}\},
                ~\text{and}\\
                J
                    &\coloneqq
                    &\bigcup_{\substack{j=1,\\c'_{j} > 0}}^{n}
                        \{k'_{j-1}+1,\ldots,k'_{j}\},\\
            \end{array}
            \end{equation}
            \end{restoremargins}

        \continueparagraph
        we now claim that

            \begin{restoremargins}
            \begin{equation}
            \label{eq:2:\beweislabel}
                I = \{ k \in \{1,2,\ldots,m\} \mid \supp([x]) \supseteq [w''_{k-1},\:w''_{k}) \},
            \end{equation}
            \end{restoremargins}

        \continueparagraph
        and analogously for $J$,
        making $I = J$.
        Towards the $\subseteq$\=/inclusion,
        let $k \in I$ be arbitrary.
        Then for some $i \in \{1,2,\ldots,m\}$
        one has $k_{i-1} + 1 \leq k \leq k_{i}$
        and $c_{i} > 0$.
        By strict monotonicity,
        it follows that
            $
                w_{i-1}
                = w''_{k_{i-1}}
                \preceq w''_{k-1}
                \prec w''_{k}
                \preceq w''_{k_{i}}
                = w_{i}
            $,
        making $[w_{i-1},\:w_{i}) \supseteq [w''_{k-1},\:w''_{k})$
        non-empty intervals.
        Since $\cover_{[x]} \equiv c_{i}$ on $[w_{i-1},\:w_{i})$
        and thus on the subinterval $[w''_{k-1},\:w''_{k})$,
        it follows that
            $[w''_{k-1},\:w''_{k}) \subseteq \supp([x])$.
        Towards the $\supseteq$\=/inclusion,
        let $k \in \{1,\ldots,l\}$
        be such that
            $[w''_{k-1},\:w''_{k}) \subseteq \supp([x])$.
        Since this interval is non-empty,
            $\cover_{[x]}(w_{k-1}) > 0$,
        so that by \eqcref{eq:0:\beweislabel},
        some $i \in \{1,2,\ldots,m-1\}$ must exist,
        such that $c_{i} > 0$ and $w''_{k-1} \in [w_{i-1},\:w_{i})$.
        In particular $[w_{i-1},\:w_{i})$ must be non-empty,
        making $w_{i-1} \preceq w''_{k-1} \prec w_{i}$.
        By construction of the refinement it follows that
            $w_{i-1} \preceq w''_{k-1} \prec w''_{k} \preceq w_{i}$.
        So, by \eqcref{eq:1:\beweislabel}, one has $k \in I$.
        Thus \eqcref{eq:2:\beweislabel} holds.
        Analogously, the same equation holds for $J$.
        Thus $I = J$ as claimed.

        By making use of
        the divisibility axioms \eqcref{ax:dyn:div:sig:article-graph-raj-dahya},
        one obtains%
        \footnote{%
            note the products over $I$, $J$ are compute in the order of the indices.
        }

        \begin{shorteqnarray}
            \prod_{\substack{i=1,\\c_{i} > 0}}^{m}
                \phi(w_{i-1}, w_{i})
            &= &\prod_{\substack{i=1,\\c_{i} > 0}}^{m}
                    \phi(w''_{k_{i-1}}, w''_{k_{i}})
                \\
            &\eqcrefoverset{ax:dyn:div:sig:article-graph-raj-dahya}{=}
                &\prod_{\substack{i=1,\\c_{i} > 0}}^{m}
                \prod_{k=k_{i-1} + 1}^{k_{i}}
                    \phi(w''_{k-1}, w''_{k})
                \\
            &\eqcrefoverset{eq:2:\beweislabel}{=}
                &\prod_{k \in I}
                    \phi(w''_{k-1}, w''_{k})
                \\
            &=
                &\prod_{k \in J}
                    \phi(w''_{k-1}, w''_{k})
                \\
            &= &\prod_{\substack{j=1,\\c'_{j} > 0}}^{n}
                \phi(w'_{j-1}, w'_{j}),
        \end{shorteqnarray}

        \continueparagraph
        where the final expressions holds by the above claim that $I = J$,
        and by computations similar to the first expressions.
        It follows that
        \eqcref{eq:extension-cover:I:sig:article-graph-raj-dahya}
        depends only on $\cover_{[x]}$
        and not on the exact expression of this function.
        We thereby obtain a well-defined map
            ${\quer{\phi} : G_{\mathcal{G}} \to M}$,
        which clearly satisfies
            $
                \ran(\quer{\phi})
                \subseteq \Generate{\ran(\phi) \cup \{1\}}
                = \Generate{\ran(\phi)}
            $.

        \paragraph{Properties:}
        Let $(u, v) \in E$ be arbitrary
        and set $x \coloneqq \letter{(u, v)}$.
        Since $u \preceq v$,
        one has
            $
                \cover_{[x]}
                = \cover_{x}
                = \onefct_{[v]} - \onefct_{[u]}
                = \onefct_{[u,\:v)}
            $.
        By \eqcref{eq:extension-cover:I:sig:article-graph-raj-dahya},
        it follows that
            $
                \quer{\phi}([\letter{(u, v)}])
                = \quer{\phi}([x])
                = \phi(u, v)
            $.
        Furthermore,
        letting $u \in \Omega$ be arbitrary,
        we have
            $
                \quer{\phi}(1)
                = \quer{\phi}([1])
                = \quer{\phi}([(u, u)])
                = \phi(u, u)
                = 1
            $,
        since $\phi$ satisfies
        the identity axiom \eqcref{ax:dyn:id:sig:article-graph-raj-dahya}.
        Finally since
            $\quer{\phi}([x])$ is completely determined by $\cover_{x}$
            for each $x \in \Gamma_{\mathcal{G}}^{\ast}$,
        and since
        by \eqcref{it:cover:words:sig:article-graph-raj-dahya},
        the map
            ${\Gamma_{\mathcal{G}}^{\ast} \ni x \mapsto \cover_{x}}$
        has cyclic invariance,
        it follows that $\quer{\phi}$
        has cyclic invariance.
    \end{proof}

We shall refer to
the particular construction in
    \eqcref{eq:extension-cover:I:sig:article-graph-raj-dahya}
as the \highlightTerm{\First cover extension}.

%% ********** END OF FILE: body/sec-4-extensions/sec-2-para-2-extension.tex **********

%% ********************************************************************************
%% FILE: body/sec-4-extensions/sec-2-para-3-continuity.tex
%% ********************************************************************************

\begin{highlightboxWithBreaks}
\begin{lemm}[Continuity of the \First cover extension]
\makelabel{lemm:extension-cover-continuous:sig:article-graph-raj-dahya}
    Let $\mathcal{G} = (\Omega, E)$
    be a graph where $E$ is a reflexive linear ordering.
    Let $\{\phi(u,v)\}_{(u,v) \in E} \subseteq \BoundedOps{\BanachRaum}$
    be a family of bounded operators
    (\resp contractions)
    on a Banach space $\BanachRaum$.
    Suppose further that $\phi$ is a divisible dynamical system on $\mathcal{G}$,
    \idest
        satisfies
        the identity axiom \eqcref{ax:dyn:id:sig:article-graph-raj-dahya}
        and
        the divisibility axiom \eqcref{ax:dyn:div:sig:article-graph-raj-dahya}.
    Suppose further that $\phi$ has geometric growth.
    Then there exists a family
        $\{\quer{\phi}(g)\}_{g \in G_{\mathcal{G}}} \subseteq \BoundedOps{\BanachRaum}$
    of bounded operators
    (\resp contractions)
    with $\ran(\quer{\phi}) \subseteq \Generate{\ran(\phi)}$,%
    \footnoteref{ft:algebraic-generator:sig:article-graph-raj-dahya}
    satisfying
        $\quer{\phi}(1) = \onematrix$
        and
        $\quer{\phi}([\letter{(u, v)}]) = \phi(u, v)$
        for all $(u, v) \in E$.
    Moreover, $\quer{\phi}$ has embedded uniform strong continuity
    \wrt the map ${\iota : E \ni e \to [\letter{e}] \in G_{\mathcal{G}}}$.%
    \footnoteref{ft:1:\beweislabel}
\end{lemm}
\end{highlightboxWithBreaks}

    \footnotetext[ft:1:\beweislabel]{%
        see \Cref{defn:embedded-cts:sig:article-graph-raj-dahya}.
    }

    \begin{proof}
        Setting $M$ to be the monoid of all bounded operators (\resp all contractions) on $\BanachRaum$
        under operator multiplication,
        the conditions of \Cref{lemm:extension-cover:I:sig:article-graph-raj-dahya}
        are clearly fulfilled.
        We can thus choose
            $\quer{\phi}$
        to be the \First cover extension
        of $\phi$.
        It remains to demonstrate the continuity claim.
        To this end we first establish the following estimate

            \begin{restoremargins}
            \begin{equation}
            \label{eq:1:\beweislabel}
                    \norm{
                        \quer{\phi}([x] [\letter{(u, v)}] [x'])
                        -
                        \quer{\phi}([x] [x'])
                    }
                    \leq e^{\ell(u, v)} - 1
            \end{equation}
            \end{restoremargins}

        \continueparagraph
        for $x, x' \in \Gamma^{\ast}_{\mathcal{G}}$
        and $(u, v) \in E$.
        Fixing such elements,
        by \eqcref{it:cover:words:sig:article-graph-raj-dahya}
        and \eqcref{it:cover:equivalence-classes:sig:article-graph-raj-dahya}
        above one has
            $
                \cover_{[x][\letter{(u,v)}][x']}
                = \cover_{[x'][x][\letter{(u,v)}]}
                = \cover_{[x'][x]} + \cover_{[\letter{(u,v)}]}
                = \cover_{[x][x']} + \cover_{\letter{(u,v)}}
                = \cover_{[xx']} + \cover_{u,v}
            $.
        Since $u \preceq v$ by definition of $E$,
        one has
            $
                \cover_{u,v}
                = \onefct_{[v)} - \onefct_{[u)}
                = \onefct_{[u,\:v)}
            $.
        By \eqcref{eq:cover-fct-as-simple-fct:sig:article-graph-raj-dahya}
        and \eqcref{it:cover:equivalence-classes:sig:article-graph-raj-dahya} above
        one can write
            $
                \cover_{[xx']}
                = \sum_{i=1}^{m}
                    c_{i}\onefct_{[w_{i-1},\:w_{i})}
            $
        for some $m\in\naturals$,
        some monotone sequence
            $\{w_{i}\}_{i=0}^{m} \subseteq \Omega$,
        some integer sequence
            $\{c_{i}\}_{i=1}^{m} \subseteq \integers$.
        By refinement, one can assume
            $w_{i_{0}} = u$
            and
            $w_{j_{0}} = v$
        for some $i_{0} < j_{0}$ in $\{0,1,\ldots,m\}$.
        By the construction in \eqcref{eq:extension-cover:I:sig:article-graph-raj-dahya},
        one thus has
        $
            \phi([x][x'])
            = \prod_{i=1}^{m}
            \phi(w_{i-1},w_{i})^{\chi_{i}}
        $
        and

            \begin{shorteqnarray}
                \phi([x][\letter{(u,v)}][x'])
                    =
                        \prod_{i=1}^{m}
                            \phi(w_{i-1},w_{i})^{\chi'_{i}}
                    = \prod_{i=1}^{m}
                        (\phi(w_{i-1},w_{i})^{\chi_{i}} + (\chi'_{i} - \chi_{i}) \cdot (\phi(w_{i-1},w_{i}) - \onematrix)),
            \end{shorteqnarray}

        \continueparagraph
        where each $\chi_{i} \in \{0,1\}$ indicates whether $c_{i} > 0$
        and each $\chi'_{i} \in \{0,1\}$ indicates whether $c_{i} + \onefct_{\{i_{0}+1,\ldots,j_{0}\}}(i) > 0$.
        Observe in particular that $\chi' \geq \chi$ pointwise
        with
            $
                I \coloneqq \{i \in \{1,2,\ldots,m\} \mid \chi_{i} \neq \chi'_{i}\}
                \subseteq \{i_{0}+1,\ldots,j_{0}\}
            $.
        Applying geometric growth thus yields

        \begin{longeqnarray}
            \norm{
                \phi([x][\letter{(u,v)}][x'])
                -
                \phi([x][x'])
            }
            &=
                &\normLong{
                    \sum_{\substack{
                        C \subseteq I,\\
                        C \neq \emptyset
                    }}
                        \prod_{i=1}^{m}
                            \begin{cases}
                                \phi(w_{i-1},w_{i}) - \onematrix
                                    &: &i \in C\\
                                \phi(w_{i-1},w_{i})^{\chi_{i}}
                                    &: &i \notin C\\
                            \end{cases}
                }
            \\
            &\leq
                &\sum_{\substack{
                    C \subseteq I,\\
                    C \neq \emptyset
                }}
                    \prod_{i=1}^{m}
                    \begin{cases}
                        \underbrace{
                            \norm{\onematrix - \phi(w_{i-1},w_{i})}
                        }_{
                            \leq \ell(w_{i-1}, w_{i})
                        }
                            &: &i \in C\\
                        \underbrace{
                            \norm{\phi(w_{i-1},w_{i})}
                        }_{\leq 1}
                        {}^{\chi_{i}}
                            &: &i \notin C\\
                    \end{cases}
            \\
            &\leq
                &\sum_{\substack{
                    C \subseteq I,\\
                    C \neq \emptyset
                }}
                    \prod_{i \in C}
                    \ell(w_{i-1}, w_{i})
            \\
            &=
                &\prod_{i \in I}
                    (1 + \ell(w_{i-1}, w_{i}))
                - 1
            \\
            &\leq
                &\prod_{i \in I}
                    e^{\ell(w_{i-1}, w_{i})}
                - 1
            \\
            &=
                &e^{\sum_{i \in I}\ell(w_{i-1}, w_{i})} - 1
            \\
            &\overset{(\ast)}{\leq}
                &e^{\sum_{i = i_{0} + 1}^{j_{0}}\ell(w_{i-1}, w_{i})} - 1
            \\
            &\overset{(\ast\ast)}{\leq}
                &e^{\ell(u, v)} - 1,
        \end{longeqnarray}

        \continueparagraph
        where
        ($\ast$) holds since $I \subseteq \{i_{0}+1,\ldots,j_{0}\}$
        and
        ($\ast\ast$) holds by superadditivity of $\ell$.
        Hence we have established \eqcref{eq:1:\beweislabel}

        Consider now arbitrary $x, x' \in \Gamma^{\ast}_{\mathcal{G}}$
        and $(u_{0}, v_{0}) \in E$.
        For $(u, v) \in E$
        set $\quer{u} \coloneqq \min\{u_{0}, u\}$
        and $\quer{v} \coloneqq \max\{v_{0}, v\}$.
        Then
            $
                (\quer{u},u_{0}),
                (\quer{u},u),
                (v_{0},\quer{v}),
                (v,\quer{v})
                \in E
            $.
        The estimate in \eqcref{eq:1:\beweislabel} yields

        \begin{shorteqnarray}
            &&\norm{
                \quer{\phi}([x] [\letter{(u, v)}] [x'])
                -
                \quer{\phi}([x] [\letter{(u_{0}, v_{0})}] [x'])
            }\\
            &\leq
                &\begin{array}[t]{0l}
                    \norm{
                        \quer{\phi}([x] [\letter{(u, v)}] [x'])
                        -
                        \quer{\phi}([x] [\letter{(\quer{u}, v)}] [x'])
                    }\\
                    + \norm{
                        \quer{\phi}([x] [\letter{(\quer{u}, v)}] [x'])
                        -
                        \quer{\phi}([x] [\letter{(\quer{u}, \quer{v})}] [x'])
                    }\\
                    + \norm{
                        \quer{\phi}([x] [\letter{(\quer{u}, \quer{v})}] [x'])
                        -
                        \quer{\phi}([x] [\letter{(\quer{u}, v_{0})}] [x'])
                    }\\
                    + \norm{
                        \quer{\phi}([x] [\letter{(\quer{u}, v_{0})}] [x'])
                        \quer{\phi}([x] [\letter{(u_{0}, v_{0})}] [x'])
                    }
                \end{array}
            \\
            &\leq
                &\begin{array}[t]{0l}
                    \norm{
                        \quer{\phi}([x] [\letter{(u, v)}] [x'])
                        -
                        \quer{\phi}([x] [\letter{(\quer{u}, u)}] [\letter{(u, v)}] [x'])
                    }\\
                    + \norm{
                        \quer{\phi}([x] [\letter{(\quer{u}, v)}] [x'])
                        -
                        \quer{\phi}([x] [\letter{(\quer{u}, v)}] [\letter{(v, \quer{v})}] [x'])
                    }\\
                    + \norm{
                        \quer{\phi}([x] [\letter{(\quer{u}, v_{0})}] [\letter{(v_{0}, \quer{v})}] [x'])
                        -
                        \quer{\phi}([x] [\letter{(\quer{u}, v_{0})}] [x'])
                    }\\
                    + \norm{
                        \quer{\phi}([x] [\letter{(\quer{u}, u_{0})}] [\letter{(u_{0}, v_{0})}] [x'])
                        \quer{\phi}([x] [\letter{(u_{0}, v_{0})}] [x'])
                    }
                \end{array}
            \\
            &\leq
                &e^{\ell(\quer{u}, u)} - 1
                + e^{\ell(v, \quer{v})} - 1
                + e^{\ell(v_{0}, \quer{v})} - 1
                + e^{\ell(\quer{u}, u)} - 1
            \\
            &= &\begin{array}[t]{0l}
                    e^{\ell(\min\{u_{0}, u\}, u)} - 1\\
                    + e^{\ell(v, \max\{v_{0}, v\})} - 1\\
                    + e^{\ell(v_{0}, \max\{v_{0}, v\})} - 1\\
                    + e^{\ell(\min\{u_{0}, u\}, u_{0})} - 1,
                \end{array}
        \end{shorteqnarray}

        \continueparagraph
        which, by the continuity of $\min$, $\max$, and $\ell$,
        converges uniformly (\idest independently of $x,x'$)
        to $0$ as $(u,v) \longrightarrow (u_{0}, v_{0})$.
        Thus $\quer{\phi}$ exhibits the desired continuity.
    \end{proof}

\begin{rem}
\makelabel{rem:strong-continuity-for-embedded-continuity:sig:article-graph-raj-dahya}
    As observed in \Cref{prop:geom-implies-norm-cts:dyn:sig:article-graph-raj-dahya},
    the divisibility and geometric growth
    assumed in \Cref{lemm:extension-cover-continuous:sig:article-graph-raj-dahya}
    necessarily imply norm-continuity of $\phi$.
    Geometric growth is thus quite a strong assumption.
    It would be useful to know whether upon replacing this by strong continuity,
    one is still able to achieve the desired embedded uniform continuity of $\quer{\phi}$.
\end{rem}

%% ********** END OF FILE: body/sec-4-extensions/sec-2-para-3-continuity.tex **********

%% ********** END OF FILE: body/sec-4-extensions/sec-2-lo-divisible.tex **********

%% ********************************************************************************
%% FILE: body/sec-4-extensions/sec-3-lo-indivisible.tex
%% ********************************************************************************

\subsection[Continuous extensions for indivisible systems]{Continuous extensions for indivisible systems}
\label{sec:ext:lo:indivisible:sig:article-graph-raj-dahya}

\firstparagraph
In order to handle dynamical systems for which the divisibility axiom may fail,
we need a further extension result.
We again consider graphs $\mathcal{G} = (\Omega, E)$
for which $E$ is a reflexive linear ordering on $\Omega$.
As above,
we let $\mathrel{\prec}$ denote the irreflexive part of $E$,
so that $E = \{(u, v) \in \Omega \mid u \preceq v\}$.
And we topologise $\Omega$ by the order topology,
and $E \subseteq \Omega \times \Omega$
with the relative topology.
We further restrict ourselves to families of the form
    $\{e^{A(u,v)}\}_{(u, v) \in E}$.

\begin{highlightboxWithBreaks}
\begin{lemm}[\Second cover extension for linearly ordered graphs]
\makelabel{lemm:extension-cover:II:sig:article-graph-raj-dahya}
    Let $\BanachRaum$ be a Banach space
    and
        $\phi = \{e^{A(u,v)}\}_{(u, v) \in E}  \subseteq \BoundedOps{\BanachRaum}$
    be a family of contractions,
    where
        $A = \{A(u,v)\}_{(u, v) \in E} \subseteq \BoundedOps{\BanachRaum}$
    is an additive family of bounded dissipative operators.
    Then there exists a family
        $\quer{\phi} = \{\quer{\phi}(g)\}_{g \in G_{\mathcal{G}}} \subseteq \BoundedOps{\BanachRaum}$
    of contractions
    satisfying
        $\quer{\phi}(1) = \onematrix$
        and
        $\quer{\phi}([\letter{(u, v)}]) = \phi(u, v)$
    for all $(u, v) \in E$.
    Moreover
        $\quer{\phi}$ has cyclic invariance,
    \idest
        $\quer{\phi}(gh) = \quer{\phi}(hg)$
    for all $g,h \in G_{\mathcal{G}}$.
    If $A$ has geometric growth,
    then $\quer{\phi}$ has embedded uniform strong continuity
    \wrt the map ${\iota : E \ni e \to [\letter{e}] \in G_{\mathcal{G}}}$.%
    \footnoteref{ft:1:\beweislabel}
\end{lemm}
\end{highlightboxWithBreaks}

    \footnotetext[ft:1:\beweislabel]{%
        see \Cref{defn:embedded-cts:sig:article-graph-raj-dahya}.
    }

    \begin{proof}
        \paragraph{Existence:}
        Let $x \in \Gamma_{\mathcal{G}}^{\ast}$ be an arbitrary word.
        By \eqcref{eq:cover-fct-as-simple-fct:sig:article-graph-raj-dahya}
        and \eqcref{it:cover:equivalence-classes:sig:article-graph-raj-dahya} above
        we can write

            \begin{restoremargins}
            \begin{equation}
            \label{eq:0:\beweislabel}
                \cover_{[x]}
                = \sum_{i=1}^{m}
                    c_{i}\onefct_{[w_{i-1},\:w_{i})}
            \end{equation}
            \end{restoremargins}

        \continueparagraph
        for some $m\in\naturals$,
        some monotone sequence
            $\{w_{i}\}_{i=0}^{m} \subseteq \Omega$,
        some integer sequence
            $\{c_{i}\}_{i=1}^{m} \subseteq \integers$.
        Let ${\chi : \{1,2,\ldots,m\} \to \{0,1\}}$
        be such that each $\chi_{i}$ indicates whether $c_{i} > 0$.
        Analogous to the \emph{existence} part of the proof of \Cref{lemm:extension-cover:I:sig:article-graph-raj-dahya},
        appealing to the additivity axiom on $\{A(u,v)\}_{(u, v) \in E}$,
        one has that

            \begin{restoremargins}
            \begin{equation}
            \label{eq:extension-cover:II:sig:article-graph-raj-dahya}
            \everymath={\displaystyle}
            \begin{array}[m]{rcl}
                \quer{\phi}([x])
                    &\coloneqq
                        &e^{A([x])},
                    ~\text{where}
                \\
                A([x])
                    &\coloneqq
                        &\sum_{\substack{i=1,\\c_{i} > 0}}^{m}
                            A(w_{i-1}, w_{i})
                    =
                        \sum_{i=1}^{m}
                            \chi_{i}A(w_{i-1}, w_{i}),
            \end{array}
            \end{equation}
            \end{restoremargins}

        \continueparagraph
        is independent of the
        exact form of the expression in \eqcref{eq:0:\beweislabel}.
        This yields a well-defined map
            ${\quer{\phi} : G_{\mathcal{G}} \to \BoundedOps{\BanachRaum}}$.
        Furthermore, as a positive linear combination
        of finitely many bounded dissipative operators,
            $A([x]) = \sum_{i=1}^{m} \chi_{i} A(w_{i-1}, w_{i})$
        is a bounded dissipative operator,
        and thus $\quer{\phi}([x])$ is a contraction.%
        \footnote{%
            \cf the discussion at the start of \S{}\ref{sec:examples:sig:article-graph-raj-dahya}.
        }

        \paragraph{Properties:}
        The basic properties of $\quer{\phi}$
        may be demonstrated analogously to \Cref{lemm:extension-cover:I:sig:article-graph-raj-dahya}.

        \paragraph{Continuity:}
        Suppose now that $A$ has geometric growth.
        Fix a continuous superadditive length function
            ${\ell : E \to [0,\infty)}$
        such that $\norm{A(u,v)} \leq \ell(u,v)$
        for all $(u, v) \in E$.
        We first establish the following estimate

            \begin{restoremargins}
            \begin{equation}
            \label{eq:1:\beweislabel}
                    \norm{
                        \quer{\phi}([x] [\letter{(u, v)}] [x'])
                        -
                        \quer{\phi}([x] [x'])
                    }
                    \leq \ell(u, v)
            \end{equation}
            \end{restoremargins}

        \continueparagraph
        for $x, x' \in \Gamma^{\ast}_{\mathcal{G}}$
        and $(u, v) \in E$.
        Fixing such elements,
        as in the proof of \Cref{lemm:extension-cover-continuous:sig:article-graph-raj-dahya},
        one has
            $
                \cover_{[x][\letter{(u,v)}][x']}
                = \cover_{[xx']} + \onefct_{[u,\:v)}
            $
        and one can write
            $
                \cover_{[x]}
                = \sum_{i=1}^{m}
                    c_{i}\onefct_{[w_{i-1},\:w_{i})}
            $
        for some $m\in\naturals$,
        some monotone sequence
            $\{w_{i}\}_{i=0}^{m} \subseteq \Omega$,
        some integer sequence
            $\{c_{i}\}_{i=1}^{m} \subseteq \integers$
        and by refinement, one can assume
            $w_{i_{0}} = u$
            and
            $w_{j_{0}} = v$
        for some $i_{0} < j_{0}$ in $\{0,1,\ldots,m\}$.
        By the construction in \eqcref{eq:extension-cover:I:sig:article-graph-raj-dahya},
        one thus has
            $
                A([x][x'])
                = \sum_{i=1}^{m}
                    \chi_{i}A(w_{i-1},w_{i})
            $
        and
            $
                A([x][\letter{(u,v)}][x'])
                = \sum_{i=1}^{m}
                    \chi'_{i}A(w_{i-1},w_{i})
            $,
        where each $\chi_{i} \in \{0,1\}$ indicates whether $c_{i} > 0$
        and each $\chi'_{i} \in \{0,1\}$ indicates whether $c_{i} + \onefct_{\{i_{0}+1,\ldots,j_{0}\}}(i) > 0$.
        Observe in particular that $\chi' \geq \chi$ pointwise
        with
            $
                I \coloneqq \{i \in \{1,2,\ldots,m\} \mid \chi_{i} \neq \chi'_{i}\}
                \subseteq \{i_{0}+1,\ldots,j_{0}\}
            $.
        So
            $
                A([x][\letter{(u,v)}][x'])
                = A([x][x']) + \sum_{i \in I}A(w_{i-1},w_{i})
            $.
        Applying estimates of perturbations of operator exponentials
        in \Cref{prop:bch-exp-perturbation:sig:article-graph-raj-dahya},
        and since $A([x][x'])$ and $\sum_{i \in I}A(w_{i-1},w_{i})$
        are dissipative,
        one obtains

        \begin{longeqnarray}
            \norm{
                \phi([x][\letter{(u,v)}][x'])
                -
                \phi([x][x'])
            }
            &=
                &\normLong{
                    e^{A([x][x']) + \sum_{i \in I}A(w_{i-1},w_{i})}
                    -
                    e^{A([x][x'])}
                }
            \\
            &\leq
                &\normLong{
                    \sum_{i \in I}A(w_{i-1},w_{i})
                }
            \\
            &\leq
                &\sum_{i \in I}
                    \ell(w_{i-1}, w_{i})
            \\
            &\overset{(\ast)}{\leq}
                &\sum_{i = i_{0} + 1}^{j_{0}}\ell(w_{i-1}, w_{i})
            \\
            &\overset{(\ast\ast)}{\leq}
                &\ell(u, v),
        \end{longeqnarray}

        \continueparagraph
        where
        ($\ast$) holds since $I \subseteq \{i_{0}+1,\ldots,j_{0}\}$
        and
        ($\ast\ast$) holds by superadditivity of $\ell$.
        Hence we have established \eqcref{eq:1:\beweislabel}.
        Analogous to the proof of \Cref{lemm:extension-cover-continuous:sig:article-graph-raj-dahya},
        one can derive from \eqcref{eq:1:\beweislabel}
        the expression

            \begin{shorteqnarray}
                &&\norm{
                    \quer{\phi}([x] [\letter{(u, v)}] [x'])
                    -
                    \quer{\phi}([x] [\letter{(u_{0}, v_{0})}] [x'])
                }\\
                &&\leq
                    \ell(\min\{u_{0}, u\}, u)
                    + \ell(v, \max\{v_{0}, v\})
                    + \ell(v_{0}, \max\{v_{0}, v\})
                    + \ell(\min\{u_{0}, u\}, u_{0}),
            \end{shorteqnarray}

        \continueparagraph
        for $x, x' \in \Gamma^{\ast}_{\mathcal{G}}$,
        $(u_{0}, v_{0}), (u, v) \in E$.
        By the continuity of $\min$, $\max$, and $\ell$,
        we thus obtain the embedded uniform strong continuity
        of $\quer{\phi}$.
    \end{proof}

We shall refer to the construction in
\eqcref{eq:extension-cover:II:sig:article-graph-raj-dahya}
as the \highlightTerm{\Second cover extension}.

%% ********** END OF FILE: body/sec-4-extensions/sec-3-lo-indivisible.tex **********

%% ********** END OF FILE: body/sec-4-extensions/.index.tex **********

%% ********************************************************************************
%% FILE: body/sec-5-dilation/.index.tex
%% ********************************************************************************

\section[Non-classical dilation theorems]{Non-classical dilation theorems}
\label{sec:dilation:sig:article-graph-raj-dahya}

\firstparagraph
In the previous section,
we established means to extend families of operators living on graphs
to operator families which live on groups.
Towards \ref{goal:dilation:sig:article-graph-raj-dahya},
in this section we both recall and extend existing dilation results
for dynamical systems parameterised by groups.

We begin with results by Stroescu \cite{Stroescu1973ArticleBanachDilations}
for operator families on Banach spaces
and adapt her result to families
of positive unital operators on \TextCStarAlgs.
Continuing in this non-classical direction,
we consider the narrower setting
of CPTP\=/maps (defined below).
We recall Kraus's representation theorems \cite{%
    Kraus1971Article,%
    Kraus1983%
},
then present a generalisation of
a result due to vom~Ende and Dirr \cite{vomEnde2019unitaryDildiscreteCPsemigroups},
which in turn builds on Kraus\=/dilations.

The Stroescu\=/dilations have the advantage that they contain conditions to guarantee continuity.
The dilations based on Kraus and vom Ende--Dirr
are applicable to discrete groups,
but have the advantage that
they yield more concrete unitary representations.

%% ********************************************************************************
%% FILE: body/sec-5-dilation/sec-1-stroescu.tex
%% ********************************************************************************

\subsection[Stroescu\=/Dilations]{Stroescu\=/Dilations}
\label{sec:dilation:stroescu:sig:article-graph-raj-dahya}

\firstparagraph
The classical result of Sz.-Nagy \cite[Theorem~I.7.1]{Nagy1970}
provides dilations
of strongly continuous operator-valued maps on Hilbert spaces
which are parameterised by topological groups,
to strongly continuous unitary representations on larger Hilbert spaces.
From this basis, results were derived for $1$- and $2$\=/parameter semigroups
(see
    \cite[Theorems~I.4.2 and I.8.1]{Nagy1970},
    \cite{Ando1963pairContractions}
    \cite{Slocinski1974},
    \cite[Theorem~2]{Slocinski1982}%
).
The conditions of this general theorem are however not always satisfied
(see
    \cite[\S{}3]{Parrott1970counterExamplesDilation},
    \cite[Theorem~1]{Varopoulos1974counterexamples},
    \cite[Theorem~1.5 and Corollary~1.7~b)]{Dahya2024interpolation}%
).
By relaxing the desired properties,
Stroescu \cite{Stroescu1973ArticleBanachDilations}
proved by contrast that (strongly continuous) dilations
to representations on Banach spaces \emph{always} exist
under modest algebraic and topological conditions.
We shall adapt her result to handle families of operators defined on \TextCStarAlgs.

%% ********************************************************************************
%% FILE: body/sec-5-dilation/sec-1-para-1-banach.tex
%% ********************************************************************************

\subsubsection[Stroescu\=/Dilations for Banach spaces]{Stroescu\=/Dilations for Banach spaces}
\label{sec:dilation:stroescu:banach:sig:article-graph-raj-dahya}

\firstparagraph
Let $\BanachRaum$ be a Banach space, $G$ a topological group with neutral element $1$,
and $\phi : G \to \BoundedOps{\BanachRaum}$ an arbitrary operator-valued function.
We say that $(\tilde{\BanachRaum},U,j,r)$
is a (continuous) \highlightTerm{Banach space dilation}
of $(\BanachRaum,\phi)$ or simply $\phi$
via bounded operators (\resp contractions \resp surjective isometries)
if $\tilde{\BanachRaum}$ is a Banach space,
$j,r$ are linear maps with
    $j : \tilde{\BanachRaum} \to \BanachRaum$ being a surjective contraction
and
    $r : \BanachRaum \to \tilde{\BanachRaum}$ an isometry
satisfying $j \circ r = \onematrix$,
and $U : G \to \BoundedOps{\tilde{\BanachRaum}}$
is a(n $\topSOT$-continuous) representation of $G$ on $\tilde{\BanachRaum}$
consisting of operators
with $U(x)$ a bounded operator
(\resp a contraction \resp a surjective isometry)
and

\begin{restoremargins}
\begin{equation}
\label{eq:banach-dilation:sig:article-graph-raj-dahya}
    j\:U(x)\:r = \phi(x)
\end{equation}
\end{restoremargins}

\continueparagraph
for all $x \in G$.
In the case of dilations via surjective isometries,
which generalise unitary operators on Hilbert spaces,
we shall write $U \in \Repr{G}{\tilde{\BanachRaum}}$.

\begin{highlightboxWithBreaks}
\begin{thm}[Stroescu, 1973]
\makelabel{thm:stroescu:banach:sig:article-graph-raj-dahya}
    Let $G$ be a topological group
    and
        $K : G \to (0,\:\infty)$
    a continuous submultiplicative function,%
    \footnoteref{ft:1:\beweislabel}
    with $K(1) = 1$.
    Suppose that
        $\{\phi(x)\}_{x \in G}$
    is a family of bounded operators on a Banach space
        $\BanachRaum$
    with
        $\norm{\phi(x)} \leq K(x)$,
        and
        $\phi(1) = \onematrix$.
    Suppose further that $\phi$
    satisfies the following left-uniform $K$-continuity:

        \begin{restoremargins}
        \begin{equation}
        \label{eq:uniform-cts:banach:sig:article-graph-raj-dahya}
            \sup_{u \in G}
                K(u)^{-1}
                \norm{(\phi(ux') - \phi(ux))\xi}
            \longrightarrow 0
        \end{equation}
        \end{restoremargins}

    \continueparagraph
    as $x' \longrightarrow x$,
    for each $x \in G$
    and $\xi \in \BanachRaum$.
    Then $\phi$ admits a continuous Banach space dilation $(\tilde{\BanachRaum},U,j,r)$
    via bounded operators satisfying
        $\norm{U(x)\xi} \in [K(x^{-1})^{-1}\norm{\xi},\:K(x)\norm{\xi}]$
    for all $x \in G$, $\xi \in \tilde{\BanachRaum}$.
    In particular, if $K(\cdot) \equiv 1$, then $U$ consists of surjective isometries.
\end{thm}
\end{highlightboxWithBreaks}

\footnotetext[ft:1:\beweislabel]{%
    \idest $K(xy) \leq K(x)K(y)$ for all $x,y \in G$.
}

For a proof, see \cite{Stroescu1973ArticleBanachDilations}.
As an immediate consequence of \Cref{thm:stroescu:banach:sig:article-graph-raj-dahya},
a Banach space version of Sz.-Nagy's result \cite[Theorem~I.8.1]{Nagy1970} for the dilation of one-parameter semigroups on Hilbert spaces
easily follows (see \cite[Corollary~1]{Stroescu1973ArticleBanachDilations}).

\begin{rem}
    In the original result in \cite{Stroescu1973ArticleBanachDilations}
    Stroescu first constructs the larger dilation
    without assuming the left-uniform $K$-continuity of $\phi$.
    Both variants are equivalent,
    as one can simply impose the discrete topology on $G$ throughout,
    which renders the left-uniform $K$-continuity assumption trivially fulfilled.
\end{rem}

%% ********** END OF FILE: body/sec-5-dilation/sec-1-para-1-banach.tex **********

%% ********************************************************************************
%% FILE: body/sec-5-dilation/sec-1-para-2-cstar.tex
%% ********************************************************************************

\subsubsection[Stroescu\=/Dilations for \TextCStarAlgs]{Stroescu\=/Dilations for \TextCStarAlgs}
\label{sec:dilation:stroescu:cstar:sig:article-graph-raj-dahya}

\firstparagraph
Consider now the case that the Banach space $\BanachRaum$ is a unital \TextCStarAlg $\CStarAlg$.
Recall that a linear operator ${\Phi : \CStarAlg \to \CStarAlg}$
is called \highlightTerm{positive} if
    $\Phi(a)$ is positive for all positive elements%
    \footnote{%
        recall that $a \in \CStarAlg$
        is positive,
        written $a \geq \zeromatrix$,
        if $a = b^{\ast}b$
        for some $b \in \CStarAlg$.
    }
    $a \in \CStarAlg$,
and \highlightTerm{unital} if $\Phi(1) = 1$,
where $1$ is the unit element of $\CStarAlg$.
By the Russo--Dye theorem,
positive maps always satisfy
    $\norm{\Phi} = \norm{\Phi(1)}$
(see
    \cite[Corollary~2.9]{Paulsen2002book},
    \cite[Theorem~1.3.3]{Stoermer2013BookPosOps}%
),
and thus positive unital maps are always contractions.
Let
    $\phi = \{\Phi_{x}(\cdot)\}_{x \in G}$
be an arbitrary family of positive unital linear operators.
We shall say that $(\tilde{\CStarAlg},U,j,r)$
is a (continuous) \highlightTerm{\TextCStarAlg dilation}
of $(\CStarAlg,\phi)$ or simply $\phi$
via representations
if
    $\tilde{\CStarAlg}$ is a unital \TextCStarAlg,
    $r$ is an isometric positive unital map,
    and
    $j$ a surjective unital {}\textsuperscript{\ensuremath{\ast}}\=/homomorphism,
such that
    $j \circ r = \id_{\CStarAlg}$,
    $U : G \to \Aut{\tilde{\CStarAlg}}$
    is a(n $\topSOT$-continuous)%
    \footnote{%
        \viz
        ${G \ni x \mapsto U(x)a \in \tilde{\CStarAlg}}$
        is continuous \wrt the norm topology on the
        \TextCStarAlg
        for all $a \in \tilde{\CStarAlg}$.
    }
    representation of $G$
    via {}\textsuperscript{\ensuremath{\ast}}\=/automorphisms of $\tilde{\CStarAlg}$,
and

\begin{restoremargins}
\begin{equation}
\label{eq:banach-dilation:sig:article-graph-raj-dahya}
    j\:U(x)\:r = \Phi_{x}
\end{equation}
\end{restoremargins}

\continueparagraph
for all $x \in G$.
We now adapt Stroescu's theorem
to provide a version for \TextCStarAlgs.

\begin{highlightboxWithBreaks}
\begin{cor}[cf. Stroescu, 1973]
\makelabel{cor:stroescu:cstar:sig:article-graph-raj-dahya}
    Let $G$ be a topological group
    and
        $\phi = \{\Phi_{x}(\cdot)\}_{x \in G}$
    a family of positive unital linear operators
    on a (commutative) unital \TextCStarAlg $\CStarAlg$,
    with $\Phi_{1} = \id_{\CStarAlg}$.
    Suppose further that $\phi$ is uniformly left-continuous,
    \idest

        \begin{restoremargins}
        \begin{equation}
        \label{eq:uniform-cts:cstar:sig:article-graph-raj-dahya}
            \sup_{u \in G}
                \norm{\Phi_{ux'}(a) - \Phi_{ux}(a)}
            \longrightarrow 0
        \end{equation}
        \end{restoremargins}

    \continueparagraph
    as $x' \longrightarrow x$,
    for each $x \in G$, $a \in \CStarAlg$.
    Then $\phi$ admits a continuous \TextCStarAlg dilation $(\tilde{\CStarAlg},U,j,r)$,
    where $\tilde{\CStarAlg}$ is a (commutative) unital \TextCStarAlg.
\end{cor}
\end{highlightboxWithBreaks}

    \begin{proof}
        The proof consists of a few stages:
        construction of the larger \TextCStarAlg;
        construction of the embeddings;
        construction of the representation;
        and proof of the properties.

        \paragraph{Construction of the larger \TextCStarAlg:}
        Observe that
            $\CStarAlg_{1} \coloneqq \Cts_{b}{G}{\CStarAlg}$
        under the uniform norm
        constitutes a (commutative) unital \TextCStarAlg,
        whereby the unit is simply the function constantly
        equal to $1$ (the unit of $\CStarAlg$) on $G$.
        For $x \in G$ let $R_{x}$ denote the right-shift
        on $\Cts_{b}{G}{\CStarAlg}$.
        Since $\phi$ is $\topSOT$\=/continuous
        and each $\Phi_{\cdot}$ is a contraction,
        it is straightforward to see that
            ${R_{x}\Phi_{\cdot}(a) : G \ni y \mapsto \Phi_{yx}(a) \in \CStarAlg}$
        is a bounded continuous map
        for each $x \in G$ and $a \in \CStarAlg$.
        Thus

        \begin{restoremargins}
        \begin{equation}
        \label{eq:stroescu-theta:sig:article-graph-raj-dahya}
        \everymath={\displaystyle}
        \begin{array}[m]{rcccl}
            \theta
            &:
                &c_{00}(G, \CStarAlg)
                &\to
                &\Cts_{b}{G}{\CStarAlg}\\
            &&f
                &\mapsto
                    &\sum_{x \in \supp(f)}
                        R_{x}\Phi_{\cdot}(f(x)),
        \end{array}
        \end{equation}
        \end{restoremargins}

        \continueparagraph
        is well-defined.
        The multiplicative closure
            $\GenSet{\ran(\theta)}$
        is thus a (commutative) unital ${}^{\ast}$-subalgebra,
        and thereby the closure
            $\CStarAlg_{\theta} \coloneqq \quer{\GenSet{\ran(\theta)}}$
        a (commutative) unital \TextCStarSubAlg of $\CStarAlg_{1}$.

        \paragraph{Construction of the embeddings:}
        Define the linear maps
            ${j:\CStarAlg_{\theta} \ni f \mapsto f(1) \in \CStarAlg}$
        and
            ${r \colon \CStarAlg \ni a \mapsto \Phi_{\cdot}(a) \in \CStarAlg_{\theta}}$.
        It is straightforward to see that
            $j$ is a {}\textsuperscript{\ensuremath{\ast}}\=/homomorphism
        and that $r$ is unital,
        as $(r\:1)(x) = \Phi_{x}(1) = 1$ for all $x \in G$.
        Since each $\Phi_{x}$ is contractive
        one has
            $
                \norm{a}
                = \norm{\Phi_{1}(a)}
                \leq \sup_{x}\norm{\Phi_{x}(a)}
                \leq \norm{a}
            $
        and thus
            $\norm{r\:a} = \sup_{x}\norm{\Phi_{x}(a)} = \norm{a}$,
        \idest $r$ is an isometry.

        To show that $r$ is positive,
        consider a positive element $a \in \CStarAlg$.
        Now by the Stone-Weierstra{\ss} theorem,
        there exists a sequence $\{p_{n}\}_{n\in\naturals}$
        of (real-valued) polynomials, such that
            ${
                \sup_{t \in [0,\:\norm{a}]}\abs{\sqrt{t} - p_{n}(t)}
                \underset{n}{\longrightarrow}
                0
            }$.
        Let $x \in G$ be arbitrary.
        Since by assumption $\Phi_{x}(a) \in \CStarAlg$ is positive,
        $\sqrt{\Phi_{x}(a)}$ exists in $\CStarAlg$.
        Moreover, since $\norm{\Phi_{x}(a)} \leq \norm{a}$,
        the spectrum satisfies
            $
                \sigma(\Phi_{x}(a))
                \subseteq [0,\:\norm{\Phi_{x}(a)}]
                \subseteq [0,\:\norm{a}]
            $,
        whence we may apply the polynomial-approximants
        to compute the square roots.
        In particular, by the spectral theory for continuous functions,
            $
                \sup_{x \in G}
                    \norm{\sqrt{\Phi_{x}(a)} - p_{n}(\Phi_{x}(a))}
                = \sup_{x \in G}
                  \sup_{t \in \sigma(\Phi_{x}(a))}
                    \abs{\sqrt{t} - p_{n}(t)}
                \leq
                    \sup_{t \in [0,\:\norm{a}]}
                        \abs{\sqrt{t} - p_{n}(t)}
            $,
        which converges uniformly (\idest independently of $x$) to $0$
        as ${n \longrightarrow \infty}$.
        It follows that
            $(p_{n}(\Phi_{\cdot}(a)))_{n \in \naturals} \subseteq \CStarAlg_{\theta}$
        is a Cauchy-sequence \wrt the uniform norm
        with limit $\sqrt{\Phi_{\cdot}(a)}$.
        Since $\CStarAlg_{\theta}$ is closed,
        it follows that $g \coloneqq \sqrt{\Phi_{\cdot}(a)} \in \CStarAlg_{\theta}$.
        And since clearly
            $\Phi_{\cdot}(a) = g^{\ast}g$,
        it follows that
            $r\:a = \Phi_{\cdot}(a)$
        is positive.

        \paragraph{Construction of the representation:}
        Let $x \in G$.
        We define $U(x)$ to be the right-shift $R_{x}$
        acting on $\CStarAlg_{1}=\Cts_{b}{G}{\CStarAlg}$.
        It is a straightforward exercise to verify the following:

        \begin{enumerate}[
            label={\bfseries {\arabic*}.},
            ref={\arabic*},
        ]
            \item\label{prop:1:\beweislabel}
                $R_{x}$ is a unital {}\textsuperscript{\ensuremath{\ast}}\=/automorphism of
                $\Cts_{b}{G}{\CStarAlg}$.

            \item\label{prop:2:\beweislabel}
                The intertwining property
                    $R_{x} \circ \theta = \theta \circ R_{x}$
                holds,
                where on the left hand side of this expression
                $R_{x}$ denotes the right-shift on $\Cts{G}{\CStarAlg}$
                and on the right hand side
                the right-shift on $c_{00}(G, \CStarAlg)$.
                Since $c_{00}(G, \CStarAlg)$ is closed under right-shifts
                and $R_{x}$ is a homomorphism on $\Cts_{b}{G}{\CStarAlg}$,
                it follows that
                    $
                        R_{x}\GenSet{\ran(\theta)}
                        \subseteq
                        \GenSet{\ran(\theta)}
                    $.
        \end{enumerate}

        \continueparagraph
        So since $\CStarAlg_{\theta}$ is the closure of
            $\GenSet{\ran(\theta)}$
        within $(\CStarAlg_{1},\norm{\cdot})$
        and since $R_{x}$ preserves norms,
        the invariance in \eqcref{prop:2:\beweislabel}
        implies that
            $R_{x}\CStarAlg_{\theta} \subseteq \CStarAlg_{\theta}$
        for all $x \in G$.
        Since $R_{x^{-1}}$ inverts $R_{x}$,
        it follows from \eqcref{prop:1:\beweislabel}
        that $U(x) = R_{x}$ is a {}\textsuperscript{\ensuremath{\ast}}\=/automorphism
        of $\CStarAlg_{\theta}$.

        \paragraph{Dilation property:}
        For $x \in G$ and $a \in \CStarAlg$
        one has
        $
            j\:U(x)\:r\:a
            = (U(x)\Phi_{\cdot}(a))(1)
            = (R_{x}\Phi_{\cdot}(a))(1)
            = (\Phi_{\cdot x}(a))(1)
            = \Phi_{x}(a)
        $.
        Thus $(\CStarAlg_{\theta},U,j,r)$
        is a dilation of $(\CStarAlg,\phi)$.

        \paragraph{Continuity of $U$:}
        First consider an arbitrary element
            $
                h = \theta f \in \ran(\theta)
            $,
        where $f \in c_{00}(G, \CStarAlg)$.
        For $x,x' \in G$
        one obtains

            \begin{longeqnarray}
                \norm{(U(x') - U(x))h}
                &= &\norm{R_{x'}\theta f - R_{x}\theta f}\\
                &= &\norm{
                        \sum_{y \in \supp(f)}
                            (R_{x'y}\Phi_{\cdot} - R_{xy}\Phi_{\cdot})f(y)
                    }
                    \\
                &\leq
                    &\sum_{y \in \supp(f)}
                        \norm{
                            (R_{x'y}\Phi_{\cdot} - R_{xy}\Phi_{\cdot})f(y)
                        }
                    \\
                &= &\sum_{y \in \supp(f)}
                    \sup_{u \in G}
                        \norm{
                            \Phi_{ux'y}(f(y))
                            -
                            \Phi_{uxy}(f(y))
                        },
            \end{longeqnarray}

        \continueparagraph
        from which the continuity
        of ${G \ni x \mapsto U(x)\theta f \in \CStarAlg_{\theta}}$
        follows by the assumed
        uniform left-continuity \eqcref{eq:uniform-cts:cstar:sig:article-graph-raj-dahya}.

        Now consider an arbitrary element
        $
            a
            \coloneqq \prod_{i=1}^{n}\theta f_{i}
        $
        of the dense subspace $\GenSet{\ran(\theta)} \subseteq \CStarAlg_{\theta}$,
        where
            $f_{1},f_{2},\ldots,f_{n} \in c_{00}(G, \CStarAlg)$
        and $n\in\naturals$.
        Since each $U(x)$ is a homomorphism
        and
        multiplication is continuous \wrt the topology on $\CStarAlg_{\theta}$,
        it follows that
            ${
                G \ni x
                \mapsto
                U(x)a
                = \prod_{i=1}^{n}U(x)\theta f_{i}
            }$
        is continuous.

        Finally, consider arbitrary
            $x \in G$,
            $a \in \CStarAlg_{\theta}$,
            and
            $\eps > 0$.
        By density one may find
            $\tilde{a} \in \GenSet{\ran(\theta)}$
        with
            $
                \norm{a - \tilde{a}}
                < \frac{\eps}{4}
            $.
        By the established continuity of
            $U(\cdot)\tilde{a}$,
        there exists a neighbourhood
            $W \subseteq G$
        of $x$
        such that
            $\norm{(U(x') - U(x))\tilde{a}} < \tfrac{\eps}{2}$
        for all $x' \in W$.
        Relying on this, one obtains

            \begin{longeqnarray}
                \norm{(U(x') - U(x))a}
                &\leq
                    &\norm{(U(x') - U(x))\tilde{a}}
                    + \norm{U(x')(a - \tilde{a})}
                    + \norm{U(x)(a - \tilde{a})}\\
                &< &\frac{\eps}{2}
                    + 2\norm{a - \tilde{a}}
                < \frac{\eps}{2}
                    + 2\frac{\eps}{4}
                = \eps
            \end{longeqnarray}

        \continueparagraph
        for all $x' \in W$.
        This establishes the $\topSOT$-continuity of the representation.
    \end{proof}

\begin{rem}
    In \cite[Theorem~1]{Evans1976ArticlePos}
    Evans established a \TextCStarAlg dilation result
    as a response to Stroescu's result.
    However the conditions in his result are stricter
    than \Cref{cor:stroescu:cstar:sig:article-graph-raj-dahya},
    requiring the maps to be \highlightTerm{completely positive}
    (which we shall discuss in \S{}\ref{sec:dilation:kraus:sig:article-graph-raj-dahya}).
    This is needed as his approach relies on the Stinespring dilation theorem.
    Stroescu's approach by contrast,
    which we followed above,
    is not based on Stinespring's result.
\end{rem}

%% ********** END OF FILE: body/sec-5-dilation/sec-1-para-2-cstar.tex **********

%% ********************************************************************************
%% FILE: body/sec-5-dilation/sec-1-para-3-continuity.tex
%% ********************************************************************************

\subsubsection[Restricted continuity]{Restricted continuity}
\label{sec:dilation:stroescu:cts:sig:article-graph-raj-dahya}

\firstparagraph
The results of Stroescu provide us with
(necessary and) sufficient means to achieve continuity of the dilations.%
\footnote{%
    the necessity of the conditions can be demonstrated
    by applying the dilation to the expressions in
    \eqcref{eq:uniform-cts:banach:sig:article-graph-raj-dahya}
    and
    \eqcref{eq:uniform-cts:cstar:sig:article-graph-raj-dahya}.
}
These conditions can however be quite imposing and difficult to ensure.
Fortunately we are less interested in the continuity of the dilations
on their full definition sets,
but rather on a subset.
In this subsection we establish sufficient conditions to achieve this restricted continuity.

\begin{highlightboxWithBreaks}
\begin{lemm}
\makelabel{lemm:stroescu:banach:restr:sig:article-graph-raj-dahya}
    Let ${\iota : E \to G}$ be an arbitrary map between
    a topological space $E$ and a group $G$,
    and let
        $\{\phi(x)\}_{x \in G} \subseteq \BoundedOps{\BanachRaum}$
    be a family of contractions on a Banach space $\BanachRaum$,
    with $\phi(1) = \onematrix$.
    Consider the Stroescu\=/dilation
        $(\tilde{\BanachRaum},U,j,r)$
    of $\phi$ under discretisation of $G$
    in \Cref{thm:stroescu:banach:sig:article-graph-raj-dahya}.%
    \footnoteref{ft:1:\beweislabel}
    If $\phi$ has embedded uniform strong continuity \wrt $\iota$,%
    \footnoteref{ft:defn:embedded-cts:sig:article-graph-raj-dahya}
    then
        ${U \circ \iota : E \to \BoundedOps{\tilde{\BanachRaum}}}$
    is strongly continuous.
\end{lemm}
\end{highlightboxWithBreaks}

    \footnotetext[ft:1:\beweislabel]{%
        Here we apply the special case $K \equiv 1$ of this result.
        By discretising $G$,
        condition
        \eqcref{eq:uniform-cts:banach:sig:article-graph-raj-dahya}
        is trivially fulfilled.
    }

    \footnotetext[ft:defn:embedded-cts:sig:article-graph-raj-dahya]{%
        see \Cref{defn:embedded-cts:sig:article-graph-raj-dahya}.
    }

    \begin{proof}
        The constructions in the proof of \Cref{thm:stroescu:banach:sig:article-graph-raj-dahya}
        are similar to (but slightly simpler than)
        those involved in the proof of the \TextCStarAlg version:
        One has that each $U(x)$ is an isometry
        and
            $\tilde{\BanachRaum} = \quer{\ran}(\theta)$,
        \idest the norm closure of $\ran(\theta)$,
        where
            ${\theta : c_{00}(G, \BanachRaum) \to \tilde{\BanachRaum}}$
        is a map defined
        analogously to \eqcref{eq:stroescu-theta:sig:article-graph-raj-dahya}.
        To prove the continuity of
            $U \circ \iota$,
        it thus suffices to prove
        the norm-continuity of
            ${E \ni e \mapsto U(\iota(e))\theta f \in \tilde{\BanachRaum}}$
        for each $f \in c_{00}(G, \BanachRaum)$.
        So let $f \in c_{00}(G, \BanachRaum)$ and $e \in E$ be arbitrary.
        For $e' \in E$ one computes

            \begin{longeqnarray}
                \norm{U(\iota(e'))\theta f - U(\iota(e))\theta f}
                &= &\norm{R_{\iota(e')}\theta f - R_{\iota(e)}\theta f}\\
                &= &\norm{
                        R_{\iota(e')}
                        \sum_{y \in \supp(f)}
                            R_{y}\phi(\cdot)(f(y))
                        -
                        R_{\iota(e)}
                        \sum_{y \in \supp(f)}
                            R_{y}\phi(\cdot)(f(y))
                    }
                    \\
                &\leq
                    &\sum_{y \in \supp(f)}
                        \norm{
                            (
                                R_{\iota(e')y}\phi(\cdot)
                                -
                                R_{\iota(e)y}\phi(\cdot)
                            )(f(y))
                        }
                    \\
                &=
                    &\sum_{y \in \supp(f)}
                    \sup_{x \in G}
                        \norm{
                            (
                                \phi(x\iota(e')y)
                                -
                                \phi(x\iota(e)y)
                            )(f(y))
                        }
                    \\
                &=
                    &\card{\supp(f)}
                        \sup_{x,y \in G}
                        \norm{
                            (
                                \phi(x\iota(e')y)
                                -
                                \phi(x\iota(e)y)
                            )(f(y))
                        },
            \end{longeqnarray}

        \continueparagraph
        which converges to $0$
        as $e' \longrightarrow e$
        by the continuity condition \eqcref{eq:embedded-sot-continuity:sig:article-graph-raj-dahya}
        and since $\supp(f)$ is finite.
    \end{proof}

\begin{highlightboxWithBreaks}
\begin{lemm}
\makelabel{lemm:stroescu:cstar:restr:sig:article-graph-raj-dahya}
    Let ${\iota : E \to G}$ be an arbitrary map between
    a topological space $E$ and a group $G$,
    and let
        $\phi = \{\Phi_{x}(\cdot)\}_{x \in G} \subseteq \BoundedOps{\CStarAlg}$
    be a family of positive unital linear operators
    on a unital \TextCStarAlg $\CStarAlg$,
    with $\Phi_{1} = \id_{\CStarAlg}$.
    Consider the Stroescu\=/dilation
        $(\tilde{\CStarAlg},U,j,r)$
    of $\phi$ under discretisation of $G$
    in \Cref{cor:stroescu:cstar:sig:article-graph-raj-dahya}.%
    \footnoteref{ft:1:\beweislabel}
    If $\phi$ has embedded uniform strong continuity \wrt $\iota$,%
    \footnoteref{ft:defn:embedded-cts:sig:article-graph-raj-dahya}
    then
        ${U \circ \iota : E \to \BoundedOps{\tilde{\CStarAlg}}}$
    is strongly continuous.
\end{lemm}
\end{highlightboxWithBreaks}

    \footnotetext[ft:1:\beweislabel]{%
        by discretising $G$,
        condition
        \eqcref{eq:uniform-cts:cstar:sig:article-graph-raj-dahya}
        is trivially fulfilled.
    }

    \begin{proof}
        We work with the construction in \Cref{cor:stroescu:cstar:sig:article-graph-raj-dahya}.
        Note that each $U(x)$ is a {}\textsuperscript{\ensuremath{\ast}}\=/automorphism
        of $\tilde{\CStarAlg}$
        and thereby an isometry.
        Moreover
            $\tilde{\CStarAlg} = \quer{\Generate{\ran(\theta)}}$,
        \idest the norm closure of the multiplicative closure of $\ran(\theta)$,
        where
            ${\theta : c_{00}(G, \CStarAlg) \to \tilde{\CStarAlg}}$
        is the map defined in \eqcref{eq:stroescu-theta:sig:article-graph-raj-dahya}.
        To prove the continuity of
            $U \circ \iota$,
        it thus suffices to prove
        the norm-continuity of
            ${E \ni e \mapsto U(\iota(e))\theta f \in \tilde{\CStarAlg}}$
        for each $f \in c_{00}(G, \CStarAlg)$.
        This in turn may be shown
        by a computation analogous to the one in the proof \Cref{lemm:stroescu:banach:restr:sig:article-graph-raj-dahya}
        and relying on the assumed continuity condition satisfied by $\phi$.
    \end{proof}

%% ********** END OF FILE: body/sec-5-dilation/sec-1-para-3-continuity.tex **********

%% ********** END OF FILE: body/sec-5-dilation/sec-1-stroescu.tex **********

%% ********************************************************************************
%% FILE: body/sec-5-dilation/sec-2-kraus.tex
%% ********************************************************************************

\subsection[Kraus\=/Dilations of CPTP\=/maps]{Kraus\=/Dilations of CPTP\=/maps}
\label{sec:dilation:kraus:sig:article-graph-raj-dahya}

\firstparagraph
We now consider the narrower context of \TextCStarAlgs
being the full space of bounded operators on a Hilbert space,
\idest $\CStarAlg = \BoundedOps{\HilbertRaum}$,
which constitutes a von~Neumann algebra.
In the previous subsection we established dilations
for families of positive unital operators on \TextCStarAlgs
to families of automorphisms.
In order to further obtain \emph{inner} automorphisms,
\idest automorphisms obtained by
the adjoint action of unitaries (and ultimately unitary representations),
we require the stronger notion of \highlightTerm{complete positivity}
(see \S{}\ref{sec:intro:notation:sig:article-graph-raj-dahya}).

%% ********************************************************************************
%% FILE: body/sec-5-dilation/sec-2-para-1-basic.tex
%% ********************************************************************************

\subsubsection[CPTP\=/maps]{CPTP\=/maps}
\label{sec:dilation:kraus:examples:sig:article-graph-raj-dahya}

\firstparagraph
In the present context, one may consider linear maps
on $\BoundedOps{\HilbertRaum}$
(for the so-called \usesinglequotes{Heisenberg picture}),
or dually (\cf
    \cite[\S{}2]{Kraus1983}%
)
the predual $\BoundedOps{\HilbertRaum}_{\ast} \cong L^{1}(\HilbertRaum)$
of trace class operators
(for the so-called \usesinglequotes{Schr\"{o}dinger picture}).
Consider Hilbert spaces $H_{1}$, $H_{2}$
and a linear map
    ${\Phi : L^{1}(H_{1}) \to L^{1}(H_{2})}$.
Recall from \S{}\ref{sec:intro:notation:sig:article-graph-raj-dahya}
that $\Phi$ is called a CPTP\=/map,
if it is completely positive
and $\tr(\Phi(s)) = \tr(s)$
for all $s \in L^{1}(H_{1})$.%
\footnote{%
    Note that Haag, Kastler, and Kraus
    refer to CPTP\=/maps as \highlightTerm{(physical) operations}
    (see
        \cite[\S{}I--II]{HaagKastler1964Article},
        \cite[\S{}2]{Kraus1971Article},
        \cite[\S{}2]{Kraus1983}%
    ).
    These are also referred to as \highlightTerm{quantum channels}
    for the Schr\"{o}dinger picture.
}

In particular,
CPTP\=/maps preserve trace one positive operators.
But what do such operators signify, in particular in the quantum setting?
Consider a Hilbert space $\HilbertRaum$.
Recall that $\ketbra{\xi}{\eta} \in \BoundedOps{\HilbertRaum}$
denotes the operator
defined by $\ketbra{\xi}{\eta}x = \brkt{x}{\eta}\:\xi$
for all vectors $\xi,\eta,x\in\HilbertRaum$.
In quantum mechanics,
each unit vector $\xi \in \HilbertRaum$,
or the corresponding operator
    $\ketbra{\xi}{\xi}$
is interpreted as a \highlightTerm{pure state} of a physical system
and can be used to compute expectations.%
\footnote{%
    The reason for this becomes apparent
    when the Hilbert space is $\HilbertRaum = L^{2}(X)$
    for some measure space $(X, \mu)$.
    Let $\xi \in L^{2}(X)$ be a unit vector.
    Then $\abs{\xi}^{2} \in L^{1}(X)$
    and $\mu_{\xi} \coloneqq \abs{\xi}^{2} \cdot \mu$
    defines a probability measure on $X$.
    For an arbitrary function $f \in L^{\infty}(X)$,
    the measure theoretic expectation
    can be related to an operator theoretic notion as follows:
    $
        \mathbb{E}_{\mu_{\xi}}[f]
        = \int f \:\dee\mu_{\xi}
        = \int f\cdot|\xi|^{2}\:\dee\mu
        = \int f\cdot\xi\cdot\xi^{\ast}\:\dee\mu
        = \int M_{f}\xi\cdot\xi^{\ast}\:\dee\mu
        = \brkt{M_{f}\xi}{\xi}
        = \tr(M_{f}\ketbra{\xi}{\xi})
    $,
    where $M_{f} \in \BoundedOps{\HilbertRaum}$
    is the multiplication operator.
    We thus see that the operator
        $\ketbra{\xi}{\xi}$
    can be used to compute expectations.
}
By the spectral theory of compact self-adjoint operators,
each trace one positive operator $\rho \in L^{1}(\HilbertRaum)$
admits a representation of the form

\begin{shorteqnarray}
    \rho
    = \sum_{e \in B} p_{e} \ketbra{e}{e}
\end{shorteqnarray}

\continueparagraph
where $B \subseteq \HilbertRaum$
is an ONB of $\HilbertRaum$,
and
$\{p_{e}\}_{e \in B}$ is a probability distribution.
Trace one positive operators
thus admit an interpretation
as probabilistic \highlightTerm{ensembles} of pure states,
and thereby \highlightTerm{general (quantum) states}
of a physical system.

We now consider examples (and counterexamples) of linear operators
which are well-known to constitute (\resp fail to be) CPTP\=/maps,
the proofs of which are left as an exercise.

\begin{e.g.}[Identity]
    Let $\HilbertRaum$ be a Hilbert space.
    Then the identity map
    ${
        \id_{L^{1}(\HilbertRaum)}
        :
        L^{1}(\HilbertRaum)
        \to
        L^{1}(\HilbertRaum)
    }$
    is a CPTP\=/map.
\end{e.g.}

\begin{e.g.}[Adjoints]
    Let $\HilbertRaum$ be a Hilbert space
    and $u \in \BoundedOps{\HilbertRaum}$ be a unitary.
    Then
        $\Phi(s) = \adjoint_{u}(s) = u\:s\:u^{\ast}$
        for $s \in L^{1}(\HilbertRaum)$
    defines a CPTP\=/map.
\end{e.g.}

\begin{e.g.}[Embeddings]
    Let $\HilbertRaum$ and $H_{2}$ be Hilbert spaces
    and $\omega \in L^{1}(H_{2})$
    with $\omega \geq \zeromatrix$
    and $\tr(\omega) = 1$.
    Then
        $\Phi(s) = s \otimes \omega$
    defines a CPTP\=/map
    from $L^{1}(\HilbertRaum)$
    to $L^{1}(\HilbertRaum \otimes H_{2})$.
\end{e.g.}

\begin{e.g.}[Partial trace]
    Let $\HilbertRaum$ and $H_{2}$ be Hilbert spaces.
    Then the \highlightTerm{partial trace}

    \begin{shorteqnarray}
        \tr_{2}
        : L^{1}(\HilbertRaum \otimes H_{2})
        \to
        L^{1}(\HilbertRaum),
    \end{shorteqnarray}

    \continueparagraph
    which associates to each $s \in L^{1}(\HilbertRaum \otimes H_{2})$
    the unique linear operator $\tr_{2}(s) \in L^{1}(\HilbertRaum)$
    satisfying

    \begin{shorteqnarray}
        \tr(T\tr_{2}(s))
        = \tr((T \otimes \onematrix)s)
    \end{shorteqnarray}

    \continueparagraph
    for all bounded operators $T \in \BoundedOps{\HilbertRaum}$,
    constitutes a CPTP\=/map
    (\cf \cite[Theorem~3.5]{Kraus1971Article}).
\end{e.g.}

\begin{e.g.}[Composition]
\label{e.g.:qo:composition:sig:article-graph-raj-dahya}
    Let
        $H_{1}$, $H_{2}$, $H_{3}$
    be Hilbert spaces
    and let
        ${\Phi_{1} : L^{1}(H_{1}) \to L^{1}(H_{2})}$
        and
        ${\Phi_{2} : L^{1}(H_{2}) \to L^{1}(H_{3})}$
    be CPTP\=/maps.
    Then $\Phi_{2} \circ \Phi_{1}$ is a CPTP\=/map.
    In particular,
    by the above examples,
    letting
        $\HilbertRaum$ and $H_{2}$
        be Hilbert spaces,
        $\omega \in L^{1}(H_{2})$
        a general state
        (\idest $\omega \geq \zeromatrix$ and $\tr(\omega) = 1$),
        and
        $u \in \BoundedOps{\HilbertRaum \otimes H_{2}}$
        a unitary operator,
    the linear operator ${\Phi : L^{1}(\HilbertRaum) \to L^{1}(\HilbertRaum)}$
    defined by

    \begin{restoremargins}
    \begin{equation}
    \label{eq:cptp-comp:sig:article-graph-raj-dahya}
        \Phi(s) \coloneqq \tr_{2}(\adjoint_{u}(s \otimes \omega))
    \end{equation}
    \end{restoremargins}

    \continueparagraph
    for $s \in \HilbertRaum$,
    constitutes a CPTP\=/map.
\end{e.g.}

Note that the complete positivity axiom is indeed strictly stronger than positivity:

\begin{e.g.}[Counterexamples]
    Consider the Hilbert space $\complex^{d}$
    for some $d \in \{2,3,\ldots\}$.
    The transposition map
        ${M_{d \times d}(\complex) \ni s \mapsto s^{T} \in M_{d \times d}(\complex)}$
    \wrt the standard basis
        $\{\BaseVector{1},\BaseVector{2},\ldots,\BaseVector{d}\}$
    as well as
    ${
        M_{d \times d}(\complex)
        \ni s
        \mapsto
        \frac{d+1}{d}\tr(s)\onematrix - s
        \in M_{d \times d}(\complex)
    }$
    are positive trace-preserving operators,
    which fail to be completely positive.
    Proof of this is a simple exercise
    involving us of Choi matrices
    (see
        \cite[Definition~4.1.1 and Theorem~4.1.8]{Stoermer2013BookPosOps}%
    ).
\end{e.g.}

%% ********** END OF FILE: body/sec-5-dilation/sec-2-para-1-basic.tex **********

%% ********************************************************************************
%% FILE: body/sec-5-dilation/sec-2-para-2-thm-I.tex
%% ********************************************************************************

\subsubsection[The \First representation theorem of Kraus]{The \First representation theorem of Kraus}
\label{sec:dilation:kraus:I:sig:article-graph-raj-dahya}

\firstparagraph
It turns out that \emph{all} CPTP\=/maps take the form
of the expression
in \eqcref{eq:cptp-comp:sig:article-graph-raj-dahya}
of \Cref{e.g.:qo:composition:sig:article-graph-raj-dahya}.
This is particularly useful as it
provides a bridge between linear operators on the state space on the one hand,
and inner automorphism via unitaries on the other.
This is the content of the \Second of two representation theorems due to Kraus,
both of which we now recall.

\begin{thm}[Kraus, 1971]
\makelabel{thm:kraus:I:sig:article-graph-raj-dahya}
    Let $\HilbertRaum$ be an arbitrary Hilbert space
    and ${\Phi : L^{1}(\HilbertRaum) \to L^{1}(\HilbertRaum)}$
    a linear operator.
    Then $\Phi$ is CPTP if and only if
    a family of bounded operators
        $\{w_{i}\}_{i \in I} \subseteq \BoundedOps{\HilbertRaum}$
    exists,
    satisfying
        $\sum_{i \in I}w_{i}w_{i}^{\ast} = \onematrix$
        computed \wrt the ultra-weak topology,%
        \footnoteref{ft:1:\beweislabel}
    such that

        \begin{restoremargins}
        \begin{equation}
        \label{eq:kraus:I:sig:article-graph-raj-dahya}
            \Phi(s)
                =
                    \sum_{i \in I}
                        w_{i}^{\ast}\:s\:w_{i}
        \end{equation}
        \end{restoremargins}

    \continueparagraph
    for all $s \in L^{1}(\HilbertRaum)$,
    where the sum converges in the trace-norm sense.
\end{thm}

\footnotetext[ft:1:\beweislabel]{%
    in fact, this converges \wrt the strong operator topology,
    \cf \Cref{rem:isometric-partitions:sig:article-graph-raj-dahya}.
}

\begin{rem}[Dimension]
    In the original proof by Kraus
    (see
        \cite[Theorem~4.1]{Kraus1971Article},
        \cite[Theorem~1]{Kraus1983}%
    ),
    separability of the Hilbert spaces are assumed.
    This turns out to be an unnecessary requirement
    (see \exempli
        \cite[Theorem~9.2.3]{Davies1976BookQuantumOpenSys},
        \cite[Proposition~2.3.10, Remark~2.3.11, and Appendix~A.5.3]{vomEnde2020PhdThesis}%
    ).
    In particular, the key ingredients in Kraus's theorem,
    \viz the Stinespring dilation theorem
        (see
            \cite[Theorem~4.8]{Pisier2001bookCBmaps},
            \cite[Theorem~9.2.1]{Davies1976BookQuantumOpenSys}%
        )
    and Naimark's representation theorem
        (see
            \cite[Theorem~3]{Naimark1972normedalg},
            \cite[Lemma~9.2.2]{Davies1976BookQuantumOpenSys}%
        ),
    do not require separability.
\end{rem}

\begin{rem}[From Kraus operators to isometric partitions]
\makelabel{rem:isometric-partitions:sig:article-graph-raj-dahya}
    Consider the family $\{w_{i}\}_{i \in I}$
    in \Cref{thm:kraus:I:sig:article-graph-raj-dahya}
    (whose involutions are referred to as \highlightTerm{Kraus operators}).
    Define
        $\tilde{\HilbertRaum} \coloneqq \HilbertRaum \otimes \ell^{2}(I)$.
    and let
        $\{\BaseVector{i}\}_{i \in I}$
    denote the canonical ONB for $\ell^{2}(I)$.
    Since
        $
            \sum_{i \in I} w_{i}w_{i}^{\ast}
            = \onematrix_{\HilbertRaum}
        $,
    one can easily verify that
        ${
            v : \HilbertRaum \ni \xi
            \mapsto
            \sum_{i \in I}
                w_{i}^{\ast}\xi \otimes \BaseVector{i}
            \in \tilde{\HilbertRaum}
        }$
    defines an isometry
    and that for each $i \in I$

        \begin{restoremargins}
        \begin{equation}
        \label{eq:kraus-ops-isometric-decomp:sig:article-graph-raj-dahya}
            w_{i} = v^{\ast}v_{i},
        \end{equation}
        \end{restoremargins}

    \continueparagraph
    where
        ${
            v_{i} : \HilbertRaum \ni \xi
            \mapsto \xi \otimes \BaseVector{i}
            \in \tilde{\HilbertRaum}
        }$,
    which is clearly an isometry.%
    \footnote{%
        by going through the proofs in
        \cite[Theorem~4.1]{Kraus1971Article},
        \cite[Theorem~1]{Kraus1983},
        this form of $w_{i}$ can be explicitly seen
        in Kraus's arguments.
    }
    Observe that
        $
            v_{j}^{\ast}v_{i}
            = \delta_{ij}\cdot\onematrix_{\HilbertRaum}
        $
    for $i,j \in I$
    and
        $
            \sum_{i \in I}v_{i}v_{i}^{\ast}
            = \onematrix_{\tilde{\HilbertRaum}}
        $.
    We shall refer to
        $\{v_{i}\}_{i \in I}$
    as an
    \highlightTerm{isometric partition}
    \resp an \highlightTerm{isometric partition of the identity}
    if it satisfies
    the second last property
    \resp both of these properties.
\end{rem}

\begin{rem}[Families of CPTP\=/maps, I]
\makelabel{rem:kraus:I:family:sig:article-graph-raj-dahya}
    Consider a family
        $\{\Phi_{\alpha}\}_{\alpha \in \Lambda}$
    of CPTP\=/maps on $L^{1}(\HilbertRaum)$,
    where again $\HilbertRaum$ is an arbitrary Hilbert space.
    For each $\alpha \in \Lambda$,
    by \Cref{thm:kraus:I:sig:article-graph-raj-dahya}
    there exists a Hilbert space
        $H_{\alpha}$,
    an isometry
    $v_{\alpha} \in \BoundedOps{\HilbertRaum}{H_{\alpha}}$,
    and an isometric partition of the identity
        $
            \{v_{\alpha,i}\}_{i \in I_{\alpha}}
            \subseteq
            \BoundedOps{\HilbertRaum}{H_{\alpha}}
        $,
    such that

        \begin{restoremargins}
        \begin{equation}
        \label{eq:kraus:I:alpha:sig:article-graph-raj-dahya}
            \Phi_{\alpha}(s)
                =
                    \sum_{i \in I_{\alpha}}
                        v_{\alpha,i}^{\ast}v_{\alpha}
                        \:s
                        \:v_{\alpha}^{\ast}v_{\alpha,i}
        \end{equation}
        \end{restoremargins}

    \continueparagraph
    for $s \in L^{1}(\HilbertRaum)$.
    We now show that a common Hilbert space can be chosen
    in place of the $H_{\alpha}$:

    Let $\kappa_{\alpha} \coloneqq \dim(H_{\alpha})$ for each $\alpha\in\Lambda$
    and set
        $
            \kappa
            \coloneqq
            \max\{
                \aleph_{0},
                \sup_{\alpha \in \Lambda}\kappa_{\alpha}
            \}
        $.
    Note that we can view the cardinal $\kappa$ itself as a set.%
    \footnote{%
        \viz the set of all \usesinglequotes{ordinals} below $\kappa$,
        see \exempli \cite[(2.2) and Chapter~3, Alephs]{Jech2003}.
    }
    Observe that since each $v_{\alpha}$ isometrically embeds $\HilbertRaum$ into $H_{\alpha}$,
    we have that each $\kappa_{\alpha} \geq \dim(\HilbertRaum)$,
    and thus $\kappa$ is an infinite cardinal at least as large as $\dim(\HilbertRaum)$.
    We now consider the Hilbert space
        $\ell^{2}(\kappa)$
    which has dimension $\kappa$.

    Let $\alpha$ be arbitrary.
    By the choice of $\kappa$,
    we can find isometries
        $\iota_{\alpha} \in \BoundedOps{H_{\alpha}}{\ell^{2}(\kappa)}$,
    such that
        $\dim(\ran(\iota_{\alpha})^{\perp}) = \kappa$.
    Replacing $v_{\alpha}$ by $\iota_{\alpha} \circ v_{\alpha}$
    and each $v_{\alpha,i}$ by $\iota_{\alpha} \circ v_{\alpha,i}$,
    we have that $\{v_{\alpha,i}\}_{i \in I_{\alpha}}$
    remains an isometric partition
    and that \eqcref{eq:kraus:I:alpha:sig:article-graph-raj-dahya}
    continues to hold.
    We also have
        $
            \sum_{i \in I_{\alpha}}
                v_{\alpha,i}v_{\alpha,i}^{\ast}
            = \iota_{\alpha}\iota_{\alpha}^{\ast}
        $,
    which is the projection in $\ell^{2}(\kappa)$
    onto $\ran(\iota_{\alpha})$.
    Our goal is to extend $\{v_{\alpha,i}\}_{i \in I_{\alpha}}$
    to an isometric partition of the identity,
    in a way that \eqcref{eq:kraus:I:alpha:sig:article-graph-raj-dahya}
    continues to hold.

    By the properties of $\kappa$ as an infinite cardinal,
    applying basic cardinal arithmetic,%
    \footnote{%
        see \exempli \cite[(3.14)]{Jech2003}.
    }
    we have
        $
            \kappa \cdot \dim(\HilbertRaum)
            = \max\{\kappa,\dim(\HilbertRaum)\}
            = \kappa
            = \dim(\ran(\iota_{\alpha})^{\perp})
        $.
    There thus exists a decomposition
        $
            \ran(\iota_{\alpha})^{\perp}
            = \bigoplus_{\gamma \in \kappa}
                H_{\alpha,\gamma}
        $,
    where each $H_{\alpha,\gamma}$ is a Hilbert space
    with $\dim(H_{\alpha,\gamma}) = \dim(\HilbertRaum)$.
    Now since
    $
        \ell^{2}(\kappa)
        = \ran(\iota_{\alpha}) \oplus \ran(\iota_{\alpha})^{\perp}
        = \ran(\iota_{\alpha})
            \oplus
            \bigoplus_{\gamma \in \kappa}
                H_{\alpha,\gamma}
    $
    and each $H_{\alpha,\gamma}$ is isomorphic to $\HilbertRaum$,
    we can extend the family
        $\{v_{\alpha,i}\}_{i \in I_{\alpha}}$
    to a family
        $
            \{u_{\alpha,j}\}_{j \in J_{\alpha}}
            \subseteq
            \BoundedOps{\HilbertRaum}{\ell^{2}(\kappa)}
        $
    which now constitutes an isometric partition of the identity.
    Due to the decomposition,
    and since
        $\ran(v_{\alpha}) \subseteq \ran(\iota_{\alpha})$,
    we have that
        $u_{\alpha,j}^{\ast}v = \zeromatrix$
    for all $u_{\alpha,j}$ not in the original family of isometries.
    Thus \eqcref{eq:kraus:I:alpha:sig:article-graph-raj-dahya}
    implies
        $
            \Phi_{\alpha}(s)
            = \sum_{j \in J_{\alpha}}
                u_{\alpha,j}^{\ast}v_{\alpha}
                \:s
                \:v_{\alpha}^{\ast}u_{\alpha,j}
        $.
    This shows that the \First representation theory can be applied
    to arbitrary families of CPTP\=/maps,
    in a way that yields representations involving a common larger Hilbert space.
\end{rem}

%% ********** END OF FILE: body/sec-5-dilation/sec-2-para-2-thm-I.tex **********

%% ********************************************************************************
%% FILE: body/sec-5-dilation/sec-2-para-3-thm-II.tex
%% ********************************************************************************

\subsubsection[The \Second representation theorem of Kraus]{The \Second representation theorem of Kraus}
\label{sec:dilation:kraus:II:sig:article-graph-raj-dahya}

\begin{highlightboxWithBreaks}
\begin{thm}[Kraus, 1983]
\makelabel{thm:kraus:II:sig:article-graph-raj-dahya}
    Let $\HilbertRaum$ be an arbitrary Hilbert space
    and $\Phi : L^{1}(\HilbertRaum) \to L^{1}(\HilbertRaum)$
    a linear operator.
    Then $\Phi$ is CPTP if and only if there exists
    a Hilbert space $\tilde{\HilbertRaum}$,
    and a general state $\omega \in L^{1}(\tilde{\HilbertRaum})$,
    and a unitary operator $u \in \BoundedOps{\HilbertRaum \otimes \tilde{\HilbertRaum}}$
    such that

        \begin{restoremargins}
        \begin{equation}
        \label{eq:kraus-II:sig:article-graph-raj-dahya}
            \Phi(s) = \tr_{2}(\adjoint_{u}(s \otimes \omega))
        \end{equation}
        \end{restoremargins}

    \continueparagraph
    for $s \in \HilbertRaum$.
    Moreover,
    $u$ can be chosen to be a reflection
    and
    the state $\omega$ can be arbitrarily chosen to be any pure state
    and this choice can be made independently of $\Phi$.%
    \footnoteref{ft:1:\beweislabel}
\end{thm}
\end{highlightboxWithBreaks}

\footnotetext[ft:1:\beweislabel]{%
    recall that general states are the trace one positive elements
    $\omega \in L^{1}(\tilde{\HilbertRaum})$,
    and pure states are elements of the form
    $\ketbra{\eta}{\eta}$
    for unit vectors $\eta \in \tilde{\HilbertRaum}$.
}

The original proof can be found in \cite[Theorem~2]{Kraus1983}.
Since this was originally stated in terms of separable spaces,
we show how the unrestricted form of the \Second representation theorem
can be derived from the \First one.

\def\beweislabel{thm:kraus:II:sig:article-graph-raj-dahya}
\begin{proof}[of \Cref{\beweislabel}]
    The \usesinglequotes{if}\=/direction holds
    by \Cref{e.g.:qo:composition:sig:article-graph-raj-dahya}.
    Towards the \usesinglequotes{only if}\=/direction,
    suppose that $\Phi$ is a CPTP\=/map.
    By the \First representation theorem
    as well as
    \Cref{%
        rem:isometric-partitions:sig:article-graph-raj-dahya,%
        rem:kraus:I:family:sig:article-graph-raj-dahya%
    }
    and
    \eqcref{eq:kraus-ops-isometric-decomp:sig:article-graph-raj-dahya},
    there exists a Hilbert space
        $\tilde{\HilbertRaum}$
    with $\dim(\tilde{\HilbertRaum}) \geq \max\{\aleph_{0},\dim(\HilbertRaum)\}$,
    an isometry
        $v \in \BoundedOps{\HilbertRaum}{\tilde{\HilbertRaum}}$,
    and an isometric partition of the identity
        $
            \{v_{i}\}_{i \in I}
            \subseteq
            \BoundedOps{\HilbertRaum}{\tilde{\HilbertRaum}}
        $,
    such that \eqcref{eq:kraus:I:sig:article-graph-raj-dahya}
    holds with $w_{i} \coloneqq v^{\ast}v_{i}$.
    Working with this setup,
    let
        $H_{1} \coloneqq \HilbertRaum \otimes \HilbertRaum$,
        $H_{2} \coloneqq \HilbertRaum \otimes \tilde{\HilbertRaum}$,
    and

        \begin{displaymath}
            D
            \coloneqq
            \sum_{i \in I}
                v_{i}^{\ast}v
                \otimes
                v_{i}
            \in
            \BoundedOps{H_{1}}{H_{2}},
        \end{displaymath}

    \continueparagraph
    where the sum is computed \wrt the $\topSOT$\=/topology.
    By the properties of the family
    $\{v_{i}\}_{i \in I}$
    of isometries with orthogonal ranges,
    we have that
        $
            D^{\ast}D
            = \sum_{i,j \in I}
                v^{\ast}v_{j}v_{i}^{\ast}v
                \otimes
                v_{j}^{\ast}v_{i}
            = \sum_{i \in I}
                v^{\ast}v_{i}v_{i}^{\ast}v
                \otimes
                \onematrix
            = \onematrix
        $,
    \idest $D$ is an isometry.
    Consider the Hilbert space
        $\HilbertRaum \oplus \tilde{\HilbertRaum}$
    and let
        ${\iota_{1} : \HilbertRaum \to \HilbertRaum \oplus \tilde{\HilbertRaum}}$
        and
        ${\iota_{2} : \tilde{\HilbertRaum} \to \HilbertRaum \oplus \tilde{\HilbertRaum}}$
    denote the canonical isometric embeddings.
    Set
        $
            u_{0} \coloneqq
            (\onematrix \otimes \iota_{1})
            \:D^{\ast}
            \:(\onematrix \otimes \iota_{2})^{\ast}
            +
            (\onematrix \otimes \iota_{2})
            \:D
            \:(\onematrix \otimes \iota_{1})^{\ast}
            +
            (\onematrix \otimes \iota_{2})
            \:(\onematrix - D\:D^{\ast})
            \:(\onematrix \otimes \iota_{2})^{\ast}
        $,
    which is a bounded operator
    on $\HilbertRaum \otimes (\HilbertRaum \oplus \tilde{\HilbertRaum})$.
    Since this latter space
    is canonically isomorphic to
        $
            H_{1} \oplus H_{2}
        $,
    one may view $u_{0}$ as

        \begin{displaymath}
            \begin{matrix}{cc}
                \zeromatrix &D^{\ast}\\
                D &\onematrix - D\:D^{\ast}\\
            \end{matrix},
        \end{displaymath}

    \continueparagraph
    which is a unitary operator (in fact a reflection!).
    Finally, we choose unit vectors
        $\xi \in \HilbertRaum$
        and
        $\eta \in \tilde{\HilbertRaum}$
    (independently of $\Phi$)
    and fix the pure states
        $\omega_{0} \coloneqq \ketbra{\xi}{\xi}$
        and
        $\omega \coloneqq \ketbra{\eta}{\eta}$.

    Now since $\tilde{\HilbertRaum}$ is infinite dimensional and larger than $\HilbertRaum$,
    there exists a unitary operator
        $w \in \BoundedOps{\HilbertRaum \oplus \tilde{\HilbertRaum}}{\tilde{\HilbertRaum}}$
    and we can also ensure that $w^{\ast}\eta = \iota_{1}\xi$.
    Let $
        u
        \coloneqq \adjoint_{\onematrix \otimes w}u_{0}
        = (\onematrix \otimes w)\:u_{0}\:(\onematrix \otimes w^{\ast})
    $,
    which is a unitary operator on
        $\HilbertRaum \otimes \tilde{\HilbertRaum}$.
    Our goal is to demonstrate
        \eqcref{eq:kraus-II:sig:article-graph-raj-dahya}.

    Let
        $s \in L^{1}(\HilbertRaum)$
        and
        $T \in \BoundedOps{\HilbertRaum}$
    be arbitrary.
    By the choice of $\eta$ and $w$
    one has

        \begin{longeqnarray}
            \adjoint_{u}(s \otimes \omega)
            &= &\adjoint_{(\onematrix \otimes w)u_{0}}
                \:(\id \otimes \adjoint_{w^{\ast}})(s \otimes \omega)\\
            &= &\adjoint_{(\onematrix \otimes w)u_{0}}(s \otimes w^{\ast}\ketbra{\eta}{\eta}w)\\
            &= &\adjoint_{(\onematrix \otimes w)u_{0}}(s \otimes \iota_{1}\ketbra{\xi}{\xi}\iota_{1}^{\ast})\\
            &= &\adjoint_{(\onematrix \otimes w)u_{0}(\onematrix \otimes \iota_{1})}(s \otimes \omega_{0})\\
            &= &\adjoint_{(\onematrix \otimes w)(\onematrix \otimes \iota_{2})D}(s \otimes \omega_{0}),
        \end{longeqnarray}

    \continueparagraph
    where the final simplification holds,
    since
        $
            u_{0}\:(\onematrix \otimes \iota_{1})
            = \begin{smatrix}
                    \zeromatrix\\
                    D\\
                \end{smatrix}
            = (\onematrix \otimes \iota_{2})D
        $.
    Thus

        \begin{longeqnarray}
            \tr(
                (T \otimes \onematrix)
                \cdot
                \adjoint_{u}(s \otimes \omega)
            )
                &= &\tr(
                        (T \otimes \onematrix)
                        \:(\onematrix \otimes w\iota_{2})
                        \:D
                        \:(s \otimes \omega_{0})
                        \:D^{\ast}
                        \:(\onematrix \otimes w\iota_{2})^{\ast}
                    )\\
                &= &\tr(
                        (T \otimes \iota_{2}^{\ast}w^{\ast}w\iota_{2})
                        \:D
                        \:(s \otimes \omega_{0})
                        \:D^{\ast}
                    )\\
                &= &\tr(
                        (T \otimes \onematrix)
                        \:D
                        \:(s \otimes \omega_{0})
                        \:D^{\ast}
                    )\\
                &= &\sum_{i,j \in I}
                        \tr(
                            (T \otimes \onematrix)
                            \:(v_{j}^{\ast}v \otimes v_{j})
                            \:(s \otimes \omega_{0})
                            \:(v_{i}^{\ast}v \otimes v_{i})^{\ast}
                        )\\
                &= &\sum_{i,j \in I}
                        \tr(
                            T
                            \:v_{j}^{\ast}v
                            \:s
                            \:v^{\ast}v_{i}
                            \otimes
                            v_{j}
                            \omega_{0}
                            v_{i}^{\ast}
                        )
                    \\
                &= &\sum_{i,j \in I}
                        \tr(
                            T
                            \:v_{j}^{\ast}v
                            \:s
                            \:v^{\ast}v_{i}
                        )
                        \underbrace{
                            \tr(
                                v_{j}
                                \omega_{0}
                                v_{i}^{\ast}
                            )
                        }_{\substack{
                            = \tr(v_{i}^{\ast}v_{j}\omega_{0})\\
                            = \delta_{ij}\cancelto{1}{\tr(\omega_{0})}
                        }}
                    \\
                &= &\sum_{i \in I}
                        \tr(
                            T
                            \:v_{i}^{\ast}v
                            \:s
                            \:v^{\ast}v_{i}
                        )
                    \\
                &= &\tr(
                        T
                        \:
                        \sum_{i \in I}
                        v_{i}^{\ast}v
                        \:s
                        \:v^{\ast}v_{i}
                    )
                \eqcrefoverset{eq:kraus:I:sig:article-graph-raj-dahya}{=}
                    \tr(T\Phi(s)),
        \end{longeqnarray}

    \continueparagraph
    and since this holds for all $T$,
    it follows by definition of the partial trace
    that
        $
            \Phi(s)
            = \tr_{2}(\adjoint_{u}(s \otimes \omega))
        $
    for all $s \in L^{1}(\HilbertRaum)$.
\end{proof}

\begin{rem}[Physical interpretation]
\makelabel{rem:physical-interpretation-cptp:sig:article-graph-raj-dahya}
    The tensor product
    allows us to view the original system (S)
    as being naturally embedded in a larger system
    consisting of (S) together with an \usesinglequotes{environment} (E).
    A given state $\rho$ of (S) can be initially viewed
    as the separably coupled state $\rho \otimes \omega$ in (S+E),
    where $\omega$ is the state of (E).
    The general form of CPTP\=/maps
    in \eqcref{eq:kraus-II:sig:article-graph-raj-dahya}
    quantifies how $\omega$ can affect an otherwise undisturbed unitary evolution of $\rho$.
    For this reason, CPTP\=/maps
    for which the representation in \eqcref{eq:kraus-II:sig:article-graph-raj-dahya}
    does not simplify to simple unitary evolution,
    are referred to in the literature as
    \highlightTerm{noisy quantum channels}
    (under the Schr\"{o}dinger picture).
\end{rem}

\begin{rem}[Choice of pure state]
    The state $\omega$ being pure, was only used once,
    \viz in order to restrict the construction of $w$
    to ensure the desired intertwining property with $u_{0}$.
\end{rem}

\begin{rem}[Families of CPTP\=/maps, II]
\makelabel{rem:kraus:II:family:sig:article-graph-raj-dahya}
    Consider a family
        $\{\Phi_{\alpha}\}_{\alpha \in \Lambda}$
    of CPTP\=/maps on $L^{1}(\HilbertRaum)$,
    where $\HilbertRaum$ is an arbitrary Hilbert space.
    By \Cref{rem:kraus:I:family:sig:article-graph-raj-dahya},
    a single Hilbert space
        $\tilde{\HilbertRaum}$
    exists with
        $
            \dim(\tilde{\HilbertRaum})
            \geq
            \max\{
                \aleph_{0},
                \dim(\HilbertRaum)
            \}
        $
    such that each
        $\Phi_{\alpha}$
    has a representation \'{a}~la the \First representation theorem
    of the form
        $
            \Phi_{\alpha}(s)
            = \sum_{i \in I_{\alpha}}
                v_{\alpha,i}^{\ast}
                v_{\alpha}
                \:s
                \:v_{\alpha}^{\ast}
                v_{\alpha,i}
        $
    for $s \in L^{1}(\HilbertRaum)$,
    where $v_{\alpha}\in\BoundedOps{\HilbertRaum}{\tilde{\HilbertRaum}}$
    is an isometry
    and $\{v_{\alpha,i}\}_{i \in I_{\alpha}} \subseteq \BoundedOps{\HilbertRaum}{\tilde{\HilbertRaum}}$
    is an isometric partition of the identity.
    Fixing some pure state
        $\omega = \ketbra{\eta}{\eta}$
    for some unit vector $\eta \in \tilde{\HilbertRaum}$,
    we can run through the same arguments as in our proof of \Cref{thm:kraus:II:sig:article-graph-raj-dahya},
    and obtain unitaries
        $
            \{u_{\alpha}\}_{\alpha\in\Lambda}
            \subseteq
            \BoundedOps{\HilbertRaum\otimes\tilde{\HilbertRaum}}
        $
    (in fact reflections),
    such that
        $
            \Phi_{\alpha}(s)
            =
            \tr_{2}(\adjoint_{u_{\alpha}}(s \otimes \omega))
        $
    for all $s \in L^{1}(\HilbertRaum)$
    and all $\alpha \in \Lambda$.
\end{rem}

%% ********** END OF FILE: body/sec-5-dilation/sec-2-para-3-thm-II.tex **********

%% ********** END OF FILE: body/sec-5-dilation/sec-2-kraus.tex **********

%% ********************************************************************************
%% FILE: body/sec-5-dilation/sec-3-vom-ende.tex
%% ********************************************************************************

\subsection[Dilations for families of CPTP\=/maps]{Dilations for families of CPTP\=/maps}
\label{sec:dilation:vom-ende:sig:article-graph-raj-dahya}

\firstparagraph
We now consider families of CPTP\=/maps%
\footnote{%
    or dually: unital completely positive maps defined on von~Neumann algebras.
}
parameterised by (discrete) groups.
In the continuous setting, Davies
\cite[Theorem~2.1 and Theorem~3.1]{Davies1978dilationsCPmaps}
established dilation results involving strongly continuous unitary representations on Hilbert spaces.
However, the full expression of these dilations involve cumbersome limits
(%
    see
    \cite[Note~(ii), p.~335]{Davies1978dilationsCPmaps}%
).
More recently, vom~Ende and Dirr
obtained more concrete expressions
for families
    $\{\Phi_{n}\}_{n\in\naturalsZero}$
of CPTP\=/maps on $L^{1}(\HilbertRaum)$, where $\HilbertRaum$ is a separable Hilbert space,
parameterised by non-negative integers
(%
    see
    \cite[Theorem~4]{vomEnde2019unitaryDildiscreteCPsemigroups}%
).%
\footnote{%
    In the setup of the result of vom~Ende and Dirr,
    $\Phi_{n}(\cdot) = T^{n}$
    for a single CPTP\=/map $T$ defined on $L^{1}(\HilbertRaum)$.
    However their result applies more generally.
}
Their approach exploits the properties of Kraus's \Second representation theorem,
discussed in the preceding subsection.
Slightly adapting their approach, one may readily obtain the following generalisation:

\begin{highlightboxWithBreaks}
\begin{thm}[cf. vom~Ende--Dirr, 2019]
\makelabel{thm:vom-ende-dirr:sig:article-graph-raj-dahya}
    Let $\HilbertRaum$ be an arbitrary Hilbert space,
    and $G$ a discrete group with neutral element $1$.
    Let $\{\Phi_{x}\}_{x \in G}$ be a family of CPTP\=/maps on $L^{1}(\HilbertRaum)$
    with $\Phi_{1} = \id$.
    Then there exists a Hilbert space $\tilde{\HilbertRaum}$,
    a pure state $\omega \in L^{1}(\tilde{\HilbertRaum})$,
    and a unitary representation
        $U \in \Repr{G}{\HilbertRaum \otimes \tilde{\HilbertRaum}}$,
    such that

        \begin{restoremargins}
        \begin{equation}
        \label{eq:vom-ende-dirr:sig:article-graph-raj-dahya}
            \Phi_{x}(s)
            = \tr_{2}(\adjoint_{U(x)}(s \otimes \omega))
        \end{equation}
        \end{restoremargins}

    \continueparagraph
    holds for all $s \in L^{1}(\HilbertRaum)$
    and $x \in G$.
\end{thm}
\end{highlightboxWithBreaks}

\begin{proof}
    By \Cref{thm:kraus:II:sig:article-graph-raj-dahya}
    and \Cref{rem:kraus:II:family:sig:article-graph-raj-dahya}
    there exists a Hilbert space $H$,
    a pure state $\omega_{0} \in L^{1}(H)$,
    as well as unitaries
        $\{u_{x}\}_{x \in G} \subseteq \BoundedOps{\HilbertRaum \otimes H}$,
    such that
        $
            \Phi_{x}(s)
            =
            \tr_{2}(\adjoint_{u_{x}}(s \otimes \omega_{0}))
        $
    for all $s \in L^{1}(\HilbertRaum)$
    and all $x \in G$.
    Since $\Phi_{1} = \id$,
    \withoutlog we may assume that $u_{1} = \onematrix$.

    Consider the Hilbert space
        $
            \tilde{\HilbertRaum} \coloneqq H \otimes \ell^{2}(G)
        $.
    Let
        $
            \{\BaseVector{x}\}_{x \in G} \subseteq\ell^{2}(G)
        $
    denote the canonical ONB
    and observe that the diagonal construction

        \begin{shorteqnarray}
            u
            \coloneqq
            \sum_{x \in G}
                u_{x}
                \otimes
                \ketbra{\BaseVector{x}}{\BaseVector{x}}
        \end{shorteqnarray}

    \continueparagraph
    constitutes a unitary operator on
        $
            \HilbertRaum \otimes \tilde{\HilbertRaum}
        $.
    Let $L_{x}$ denote the left-shift on $\ell^{2}(G)$
    and set

        \begin{restoremargins}
        \begin{equation}
        \label{eq:1:\beweislabel}
        \everymath={\displaystyle}
        \begin{array}[m]{rcl}
            U(x)
            &\coloneqq
                &\adjoint_{u}(\onematrix \otimes \onematrix \otimes L_{x})
                \\
            &=
                &\sum_{y,y' \in G}
                    u_{y'}
                    u_{y}^{\ast}
                    \otimes
                    \ketbra{\BaseVector{y'}}{\BaseVector{y'}}
                    L_{x}
                    \ketbra{\BaseVector{y}}{\BaseVector{y}}
                \\
            &=
                &\sum_{y',y \in G}
                    u_{y'}
                    u_{y}^{\ast}
                    \otimes
                    \ketbra{\BaseVector{y'}}{\BaseVector{y'}}
                    \ketbra{\BaseVector{xy}}{\BaseVector{y}}
                \\
            &=
                &\sum_{y \in G}
                    u_{xy}
                    u_{y}^{\ast}
                    \otimes
                    \ketbra{\BaseVector{xy}}{\BaseVector{y}}
        \end{array}
        \end{equation}
        \end{restoremargins}

    \continueparagraph
    for $x \in G$.
    Since
        ${G \ni x \mapsto L_{x} \in \BoundedOps{\ell^{2}(G)}}$
    is a unitary representation of $G$,
    it follows that $U \in \Repr{G}{\HilbertRaum \otimes \tilde{\HilbertRaum}}$.
    Finally, set
        $
            \omega \coloneqq \omega_{0} \otimes \ketbra{\BaseVector{1}}{\BaseVector{1}}
        $,
    which constitutes a pure state on $\tilde{\HilbertRaum}$.

    For $x \in G$ and $s \in L^{1}(\HilbertRaum)$
    one has

        \begin{longeqnarray}
            \adjoint_{U(x)}(s \otimes \omega)
            &= &U(x)
                (s \otimes \omega_{0} \otimes \ketbra{\BaseVector{1}}{\BaseVector{1}})
                U(x)^{\ast}
                \\
            &\eqcrefoverset{eq:1:\beweislabel}{=}
                &\sum_{y,y' \in G}
                    u_{xy}
                    u_{y}^{\ast}
                    \:(s \otimes \omega_{0})
                    \:u_{y'}
                    u_{xy'}^{\ast}
                    \otimes
                    \ketbra{\BaseVector{xy}}{\BaseVector{y}}
                    \ketbra{\BaseVector{1}}{\BaseVector{1}}
                    \ketbra{\BaseVector{y'}}{\BaseVector{xy'}}
                \\
            &=
                &u_{x}
                u_{1}^{\ast}
                \:(s \otimes \omega_{0})
                \:u_{1}
                u_{x}^{\ast}
                \otimes
                \ketbra{\BaseVector{x}}{\BaseVector{x}}
                \\
            &=
                &\adjoint_{u_{x}}(s \otimes \omega_{0})
                \otimes
                \ketbra{\BaseVector{x}}{\BaseVector{x}},
        \end{longeqnarray}

        \continueparagraph
        since $u_{1} = \onematrix$ (see above).
        For $T \in \BoundedOps{\HilbertRaum}$ one thus obtains

        \begin{longeqnarray}
            \tr(
                (T \otimes \onematrix_{\tilde{\HilbertRaum}})
                \cdot
                \adjoint_{U(x)}(s \otimes \omega)
            )
            &=
                &\tr(
                    (T \otimes \onematrix_{H} \otimes \onematrix_{\ell^{2}(G)})
                    \cdot
                    (
                        \adjoint_{u_{x}}(s \otimes \omega_{0})
                        \otimes
                        \ketbra{\BaseVector{x}}{\BaseVector{x}}
                    )
                )
                \\
            &=
                &\tr(
                    (T \otimes \onematrix_{H})
                    \cdot
                    \adjoint_{u_{x}}(s \otimes \omega_{0})
                )
                \cdot\cancelto{1}{\tr(\ketbra{\BaseVector{x}}{\BaseVector{x}})}
                \\
            &\overset{(\ast)}{=}
                &\tr(
                    T
                    \cdot
                    \tr_{2}(\adjoint_{u_{x}}(s \otimes \omega_{0}))
                )
                \\
            &=
                &\tr(T \Phi_{x}(s))
        \end{longeqnarray}

    \continueparagraph
    where ($\ast$) holds per definition of the partial trace
    computed for $\HilbertRaum \otimes H$.
    It follows that
        $
            \tr_{2}(\adjoint_{U(x)}(s \otimes \omega))
            = \Phi_{x}(s)
        $,
    whereby the partial trace here is computed for
    $\HilbertRaum \otimes \tilde{\HilbertRaum}$.
\end{proof}

%% ********** END OF FILE: body/sec-5-dilation/sec-3-vom-ende.tex **********

%% ********** END OF FILE: body/sec-5-dilation/.index.tex **********

%% ********************************************************************************
%% FILE: body/sec-6-results/.index.tex
%% ********************************************************************************

\section[Proof of main results]{Proof of main results}
\label{sec:results:sig:article-graph-raj-dahya}

\firstparagraph
In
    \S{}\ref{sec:algebra:sig:article-graph-raj-dahya}
    and
    \S{}\ref{sec:ext:sig:article-graph-raj-dahya},
we established natural group structures associated with graphs
as well as means to lift operator families on graphs
to (continuous) operator families on their associated groups.
In \S{}\ref{sec:dilation:sig:article-graph-raj-dahya}
we recalled and extended known dilation results for dynamical systems
defined on topological and discrete groups.
Piecing these together, we obtain our main results.

%% ********************************************************************************
%% FILE: body/sec-6-results/para-1-discrete.tex
%% ********************************************************************************

We first prove \Cref{thm:result:graph-dilations:discrete:sig:article-graph-raj-dahya}
by making use of the normal form extension in \Cref{lemm:extension-normal-form:sig:article-graph-raj-dahya},
as well as the dilation results of Stroescu and Kraus / vom~Ende--Dirr.

\def\beweislabel{thm:result:graph-dilations:discrete:sig:article-graph-raj-dahya}
\begin{proof}[of \Cref{\beweislabel}]
    Let $G = G_{\mathcal{G}}$
    be the edge group associated to the graph $\mathcal{G}$
    and define
        ${\iota : E \to G}$
    by $\iota(e) \coloneqq [\letter{e}]$
    for $e \in E$.
    Apply \Cref{lemm:extension-normal-form:sig:article-graph-raj-dahya}
    to obtain the normal form extension
        ${\quer{\phi} : G \to \BoundedOps{\BanachRaum}}$
    of $\phi$,
    which satisfies
        $\quer{\phi}(1) = \onematrix$,
        $\quer{\phi} \circ \iota = \phi$,
    and $\ran(\quer{\phi}) \subseteq \Generate{\ran(\phi) \cup \{\onematrix\}}$.

    \paragraph{Claim~\punktcref{banach}:}
        Since
            $\ran(\quer{\phi}) \subseteq \Generate{\ran(\phi) \cup \{\onematrix\}}$,
        one has that
            $\{\quer{\phi}(x)\}_{x \in G}$
        is a family of contractions on $\BanachRaum$
        with $\quer{\phi}(1) = \onematrix$.
        Viewing $G$ with the discrete topology,
        we may apply \Cref{thm:stroescu:banach:sig:article-graph-raj-dahya}
        (Stroescu\=/dilations for Banach spaces)
        under the special case of $K \equiv 1$,
        and obtain a Banach space dilation
            $(\tilde{\BanachRaum},\quer{U},j,r)$
        of $(G,\quer{\phi})$,
        which immediately delivers the desired properties for $j$ and $r$.
        Finally,
        set
            ${U \coloneqq \quer{U} \circ \iota : E \to \BoundedOps{\tilde{\BanachRaum}}}$.
        We now verify the properties of $U$.
        By the properties of the dilation and the extension,
        one has that
            $\{U(u,v)\}_{(u, v) \in E}$
        is a family of surjective isometries on $\tilde{\BanachRaum}$
        with
            $
                j\:U(u, v)\:r
                = j\:\quer{U}(\iota(u, v))\:r
                = \quer{\phi}(\iota(u, v))
                = \phi(u, v)
            $
        for all $(u, v) \in E$.
        Since $\quer{U}$ is a representation of $G$
        one has
            $
                U(u, u)
                = \quer{U}(\iota(u, u))
                = \quer{U}(1)
                = \onematrix
            $
            for $u \in \Omega$
            for which $(u, u) \in E$.
        And
            $
                U(u, v)U(v, w)
                = \quer{U}(\iota(u, v))
                \quer{U}(\iota(v, w))
                = \quer{U}([\letter{(u, v)}][\letter{(v, w)}])
                = \quer{U}([\letter{(u, v)}\letter{(v, w)}])
                = \quer{U}([\letter{(u, w)}])
                = \quer{U}(\iota(u, w))
                = U(u, w)
            $
        for all $u, v, w \in \Omega$
        with $(u, v), (v, w), (u, w) \in E$.
        Hence $U$ is a divisible dynamical system on the graph $\mathcal{G}$.

    \paragraph{Claim~\punktcref{cstar}:}
        Since
            $\ran(\quer{\phi}) \subseteq \Generate{\ran(\phi) \cup \{\id_{\CStarAlg}\}}$,
        one has that
            $\quer{\phi} = \{\quer{\Phi}_{g}\}_{g \in G}$
        is a family of positive unital operators on $\CStarAlg$
        with $\Phi_{1} = \id_{\CStarAlg}$.
        Viewing $G$ with the discrete topology,
        we may apply \Cref{cor:stroescu:cstar:sig:article-graph-raj-dahya}
        (Stroescu\=/dilations for \TextCStarAlgs).
        The remainder of the proof is analogous to Claim~\punktcref{banach}.

    \paragraph{Claim~\punktcref{cptp}:}
        Since
            $\ran(\quer{\phi}) \subseteq \Generate{\ran(\phi) \cup \{\id_{L^{1}(\HilbertRaum)}\}}$,
        one has that
            $\quer{\phi} = \{\quer{\Phi}_{g}\}_{g \in G}$
        is a family of CPTP\=/maps on $L^{1}(\HilbertRaum)$
        with $\Phi_{1} = \id_{L^{1}(\HilbertRaum)}$.
        We may thus apply
        \Cref{thm:vom-ende-dirr:sig:article-graph-raj-dahya}
        (vom~Ende--Dirr dilations for CPTP\=/maps).
        The remainder of the proof is analogous to Claim~\punktcref{banach}.
\end{proof}

%% ********** END OF FILE: body/sec-6-results/para-1-discrete.tex **********

%% ********************************************************************************
%% FILE: body/sec-6-results/para-2-continuous-divisible.tex
%% ********************************************************************************

We now prove \Cref{thm:result:graph-dilations:cts:divisible:sig:article-graph-raj-dahya}
by making use of the \First cover extension in \Cref{lemm:extension-cover:I:sig:article-graph-raj-dahya},
the condition of geometric growth (see \Cref{defn:geom-growth:op-family:sig:article-graph-raj-dahya})
and the dilation results of Stroescu.

\def\beweislabel{thm:result:graph-dilations:cts:divisible:sig:article-graph-raj-dahya}
\begin{proof}[of \Cref{\beweislabel}]
    Let $G = G_{\mathcal{G}}$
    be the edge group associated to the graph $\mathcal{G}$
    and define
        ${\iota : E \to G}$
    by $\iota(e) \coloneqq [\letter{e}]$
    for $e \in E$.
    Apply \Cref{lemm:extension-cover:I:sig:article-graph-raj-dahya}
    to obtain the \First cover extension
        ${\quer{\phi} : G \to \BoundedOps{\BanachRaum}}$
    of $\phi$,
    which satisfies
        $\quer{\phi}(1) = \onematrix$,
        $\quer{\phi} \circ \iota = \phi$,
    and $\ran(\quer{\phi}) \subseteq \Generate{\ran(\phi)}$.
    Since $\phi$
    satisfies
        the identity axiom \eqcref{ax:dyn:id:sig:article-graph-raj-dahya}
        and
        the divisibility axiom \eqcref{ax:dyn:div:sig:article-graph-raj-dahya}
    and has geometric growth,
    by
        \Cref{lemm:extension-cover-continuous:sig:article-graph-raj-dahya}
    the \First cover extension $\quer{\phi}$
    has embedded uniform strong continuity \wrt the map $\iota$.

    \paragraph{Claim~\punktcref{banach}:}
        %% dilation
        Since
            $\ran(\quer{\phi}) \subseteq \Generate{\ran(\phi)}$,
        one has that
            $\{\quer{\phi}(x)\}_{x \in G}$
        is a family of contractions on $\BanachRaum$
        with $\quer{\phi}(1) = \onematrix$.
        Viewing $G$ with the discrete topology,
        we may apply \Cref{thm:stroescu:banach:sig:article-graph-raj-dahya}
        (Stroescu\=/dilations for Banach spaces)
        under the special case of $K \equiv 1$,
        and obtain a Banach space dilation
            $(\tilde{\BanachRaum},\quer{U},j,r)$
        of $(G,\quer{\phi})$.
        Setting
            ${U \coloneqq \quer{U} \circ \iota : E \to \BoundedOps{\tilde{\BanachRaum}}}$,
        we obtain
        all the desired properties for $(U,j,r)$
        bar continuity,
        analogous to the proof of
        \Crefit{thm:result:graph-dilations:discrete:sig:article-graph-raj-dahya}{banach}.
        %% continuity
        Since $\quer{\phi}$ has embedded uniform strong continuity \wrt $\iota$,
        the conditions of \Cref{lemm:stroescu:banach:restr:sig:article-graph-raj-dahya}
        are fulfilled,
        which implies that
            ${U = \quer{U} \circ \iota : E \to \BoundedOps{\tilde{\BanachRaum}}}$
        is strongly continuous.

    \paragraph{Claim~\punktcref{cstar}:}
        %% dilation
        Since
            $\ran(\quer{\phi}) \subseteq \Generate{\ran(\phi)}$,
        one has that
            $\quer{\phi} = \{\quer{\Phi}_{g}\}_{g \in G}$
        is a family of positive unital operators on $\CStarAlg$
        with $\Phi_{1} = \id_{\CStarAlg}$.
        Viewing $G$ with the discrete topology,
        we may apply \Cref{cor:stroescu:cstar:sig:article-graph-raj-dahya}
        (Stroescu\=/dilations for \TextCStarAlgs)
        to obtain a \TextCStarAlg dilation
            $(\tilde{\CStarAlg},\quer{U},j,r)$
        of $(G,\quer{\phi})$,
        where $\tilde{\CStarAlg}$
        is a unital (\resp unital commutative) \TextCStarAlg.
        Setting
            ${U \coloneqq \quer{U} \circ \iota : E \to \BoundedOps{\tilde{\CStarAlg}}}$,
        we obtain
        all the desired properties for $(U,j,r)$
        bar continuity,
        analogous to the proof of
        \Crefit{thm:result:graph-dilations:discrete:sig:article-graph-raj-dahya}{cstar}.
        %% continuity
        Since $\quer{\phi}$ has embedded uniform strong continuity \wrt $\iota$,
        the conditions of \Cref{lemm:stroescu:cstar:restr:sig:article-graph-raj-dahya}
        are fulfilled,
        which implies that
            ${U = \quer{U} \circ \iota : E \to \BoundedOps{\tilde{\CStarAlg}}}$
        is strongly continuous.
\end{proof}

%% ********** END OF FILE: body/sec-6-results/para-2-continuous-divisible.tex **********

%% ********************************************************************************
%% FILE: body/sec-6-results/para-3-continuous-indivisible.tex
%% ********************************************************************************

Finally, we prove \Cref{thm:result:graph-dilations:cts:indivisible:sig:article-graph-raj-dahya}
by making use of the \Second cover extension in \Cref{lemm:extension-cover:II:sig:article-graph-raj-dahya}.

\def\beweislabel{thm:result:graph-dilations:cts:indivisible:sig:article-graph-raj-dahya}
\begin{proof}[of \Cref{\beweislabel}]
    Let $G = G_{\mathcal{G}}$
    be the edge group associated to the graph $\mathcal{G}$
    and define
        ${\iota : E \to G}$
    by $\iota(e) \coloneqq [\letter{e}]$
    for $e \in E$.

    \paragraph{Claim~\punktcref{banach}:}
    Apply \Cref{lemm:extension-cover:II:sig:article-graph-raj-dahya}
    to obtain the \Second cover extension
        ${\quer{\phi} : G \to \BoundedOps{\BanachRaum}}$
    of $\phi$,
    which satisfies
        $\quer{\phi}(1) = \onematrix$
        and
        $\quer{\phi}(\iota(u, v)) = \phi(u, v)$
        for all $(u, v) \in E$,
    and which is a family of contractions on $\BanachRaum$.
    Since $A$ is assumed to have geometric growth,
    the \Second cover extension $\quer{\phi}$
    has embedded uniform strong continuity \wrt the map $\iota$.
    The remainder of the proof is as in the proof of
    \Cref{thm:result:graph-dilations:cts:divisible:sig:article-graph-raj-dahya}%
    ~\eqcref{it:banach:thm:result:graph-dilations:cts:divisible:sig:article-graph-raj-dahya}.

    \paragraph{Claim~\punktcref{cstar}:}
    By assumption,
        $
            \phi = \{\Phi_{(u,v)} \coloneqq e^{A(u,v)}\}_{(u,v) \in E}
        $,
    where
    each $A(u, v)$
    is a self-adjoint map
        $L_{(u, v)}$
    satisfying
        $L_{(u,v)}(1) = \zeromatrix$
    and
        $D_{L_{(u, v)}}(a, a) \geq \zeromatrix$
    for all $a \in \CStarAlg$.%
    \footnote{%
        see \S{}\ref{sec:intro:notation:sig:article-graph-raj-dahya}
        for the definition of the dissipation map, $D_{L}$.
    }
    By \Cref{cor:linblad:dissipative:sig:article-graph-raj-dahya}
    each $L_{(u, v)}$ is a dissipative operator on $\CStarAlg$.
    %% part 2
    We thus have a similar setup to Claim~\punktcref{banach}:
    Applying \Cref{lemm:extension-cover:II:sig:article-graph-raj-dahya},
    we obtain the \Second cover extension
        ${\quer{\phi} : G \to \BoundedOps{\CStarAlg}}$
    of $\phi$,
    which satisfies
        $\quer{\phi}(1) = \id$
        and
        $\quer{\phi}(\iota(u, v)) = \Phi_{(u, v)}$
        for all $(u, v) \in E$.
    Since $A$ is assumed to have geometric growth,
    the \Second cover extension $\quer{\phi}$
    has embedded uniform strong continuity \wrt the map $\iota$.

    Now by the construction of this cover extension
    one has
        $\quer{\phi}(g) = e^{A(g)}$
    for each $g \in G$,
    where by \eqcref{eq:extension-cover:II:sig:article-graph-raj-dahya}
    each $A(g)$ is a positive linear combination
    of the elements in
        $\{A(u, v) = L_{(u, v)}\}_{(u, v) \in E} \subseteq \BoundedOps{\CStarAlg}$.
    By \Cref{rem:class-of-linblad-schwarz:sig:article-graph-raj-dahya},
    each $A(g)$ is a self-adjoint map $L$
    satisfying
        $L(1) = \zeromatrix$
    and
        $D_{L}(a, a) \geq \zeromatrix$
    for all $a \in \CStarAlg$.
    Thus by the correspondence
    in \Cref{lemm:linblad:schwarz:sig:article-graph-raj-dahya},
    the semigroup
        $\{e^{t A(g)}\}_{t \in \realsNonNeg}$
    consists of unital Schwarz and thus positive operators.
    In particular, each $\quer{\phi}(g)$
    is a unital positive operator on $\CStarAlg$.
    The remainder of the proof is as in the proof of
    \Cref{thm:result:graph-dilations:cts:divisible:sig:article-graph-raj-dahya}%
    ~\eqcref{it:cstar:thm:result:graph-dilations:cts:divisible:sig:article-graph-raj-dahya}.
\end{proof}

%% ********** END OF FILE: body/sec-6-results/para-3-continuous-indivisible.tex **********

%% ********************************************************************************
%% FILE: body/sec-6-results/para-4-remarks.tex
%% ********************************************************************************

\begin{rem}
    To establish continuous dilations,
    we restricted our attention to linearly ordered graphs
    and developed different extensions of $\{\phi(u,v)\}_{(u,v) \in E}$
    to operator families defined on the group $G_{\mathcal{G}}$
    with embedded uniform continuity.
    It would be interesting to know if similar results can be achieved
    for other classes of graphs.
\end{rem}

\begin{rem}
    Recall that the conditions imposed in
    \Cref{%
        thm:result:graph-dilations:cts:divisible:sig:article-graph-raj-dahya,%
        thm:result:graph-dilations:cts:indivisible:sig:article-graph-raj-dahya%
    }
    necessitate norm-continuity of the operator families
    (see \Cref{%
        prop:geom-implies-norm-cts:dyn:sig:article-graph-raj-dahya,%
        prop:geom-implies-norm-cts:gen:sig:article-graph-raj-dahya%
    }).
    It remains a challenge to obtain these results
    under the weaker assumption of strong continuity
    (\cf \Cref{rem:strong-continuity-for-embedded-continuity:sig:article-graph-raj-dahya}).
\end{rem}

\begin{rem}
    By the correspondence in \Cref{lemm:linblad:schwarz:sig:article-graph-raj-dahya},
    the setup in
    \Cref{thm:result:graph-dilations:cts:indivisible:sig:article-graph-raj-dahya}~%
    \eqcref{it:cstar:thm:result:graph-dilations:cts:indivisible:sig:article-graph-raj-dahya}
    implies that each $\Phi_{(u,v)} \coloneqq e^{A(u,v)}$
    is a Schwarz-operator on $\CStarAlg$.
    It would be interesting to know if the claim still holds
    under the weaker assumption of positivity
    (\cf
        \Cref{thm:result:graph-dilations:cts:divisible:sig:article-graph-raj-dahya}~%
        \eqcref{it:cstar:thm:result:graph-dilations:cts:divisible:sig:article-graph-raj-dahya}%
    ).
\end{rem}

\begin{rem}
    It would be very useful to know if
    \Cref{%
        thm:result:graph-dilations:cts:divisible:sig:article-graph-raj-dahya,%
        thm:result:graph-dilations:cts:indivisible:sig:article-graph-raj-dahya%
    }
    can be extended to dilate
    dynamical systems
        $
            \{\Phi_{(u, v)}(\cdot)\}_{(u, v) \in E}
            \subseteq L^{1}(\HilbertRaum)
        $
    consisting of CPTP\=/maps on the trace-class operators
    of a Hilbert space $\HilbertRaum$,
    to obtain continuous
    counterparts to
    \Crefit{thm:result:graph-dilations:discrete:sig:article-graph-raj-dahya}{cptp}.
\end{rem}

\begin{rem}
\makelabel{rem:literature-algebraic:sig:article-graph-raj-dahya}
    In the current paper we worked with presented groups
    which arise from a monoid of words and relations induced by graph edges.
    In the literature, Vernik, \andothers
    \cite{%
        Vernik2016Article,%
        Atkinson2019Article%
    }
    employ similar techniques to model dynamical systems,
    but with words induced by graph nodes.
    It is not immediately clear how to translate
    between the two algebraic frameworks.
    In particular, whilst our graphs encoded
    the notion of \highlightTerm{divisibility},
    in the afore mentioned works these structures
    are used to encode commutation relations.
    To obtain dilations of such dynamical systems,
    Vernik makes use of the concept of \highlightTerm{subproduct systems}
    developed by Shalit and Solel
    \cite{ShalitSolel2009Article}
    (see also
        \cite{%
            HartzShalit2024Article,%
            ShalitSkeide2022multiparam%
        }%
    ).
    It would be interesting to know if this technology in turn
    can be adapted to encode the algebraic relations needed in the present paper.
\end{rem}

Our main results,
\viz
\Cref{%
    thm:result:graph-dilations:discrete:sig:article-graph-raj-dahya,%
    thm:result:graph-dilations:cts:divisible:sig:article-graph-raj-dahya,%
    thm:result:graph-dilations:cts:indivisible:sig:article-graph-raj-dahya%
},
can immediately be applied to the examples presented in
\S{}\ref{sec:examples:hamiltonian:sig:article-graph-raj-dahya}.
Considering in particular
\Cref{e.g.:hamiltonian:indivisible:sig:article-graph-raj-dahya},
by
\Cref{thm:result:graph-dilations:cts:indivisible:sig:article-graph-raj-dahya}~%
\eqcref{it:cstar:thm:result:graph-dilations:cts:indivisible:sig:article-graph-raj-dahya}
for any
(not necessarily commuting)
uniformly bounded strongly measurable family
    $\{H_{\tau}\}_{\tau\in\realsNonNeg}$
of self-adjoint operators
on a Hilbert space $\HilbertRaum$,
the dynamical system
    $
        \{\Phi_{(t,s)}\}_{s,t \in \realsNonNeg, t \geq s}
    $
on the unital \TextCStarAlg $\CStarAlg \coloneqq \BoundedOps{\HilbertRaum}$
(or dually: $L^{1}(\HilbertRaum)$)
defined by

\begin{shorteqnarray}
    \Phi_{(t,s)}
    \coloneqq
    e^{\iunit [\int_{s}^{t} H_{\tau}\:\dee\tau,\:\cdot]}
    = \adjoint_{e^{\iunit \int_{s}^{t} H_{\tau}\:\dee\tau}}
\end{shorteqnarray}

\continueparagraph
for $t \geq s$,
may itself not be a divisible process,
but it can always be embedded into one
defined on a larger \TextCStarAlg.
More generally, if continuity is not a concern,
then by
    \Cref{thm:result:graph-dilations:discrete:sig:article-graph-raj-dahya}~%
    \eqcref{it:cptp:thm:result:graph-dilations:discrete:sig:article-graph-raj-dahya}
the embeddings themselves are physically meaningful.
This result confirms the understanding
that Markovian systems are unavoidable phenomena in the quantum setting.

%% ********** END OF FILE: body/sec-6-results/para-4-remarks.tex **********

%% ********** END OF FILE: body/sec-6-results/.index.tex **********

%% ********** END OF FILE: body/.index.tex **********

%% BACKMATTER:

%% ********************************************************************************
%% FILE: back/.index.tex
%% ********************************************************************************

\null

%% ********************************************************************************
%% FILE: back/thanks.tex
%% ********************************************************************************

\paragraph{Acknowledgement.}
The author is grateful
to Orr Shalit for helpful suggestions and references,
to Jacob Barandes for insight into the physics side of indivisibility,
to Leonardo Goller for useful exchanges and remarks regarding the BCH formula,
and to Tanja Eisner for her feedback.

%% ********** END OF FILE: back/thanks.tex **********

%% LITERATURVERZEICHNIS:
\bibliographystyle{siam}

\footnotesize
\begin{thebibliography}{10}

\bibitem{Ando1963pairContractions}
{\sc T.~And\^{o}}, {\em {On a pair of commutative contractions}}, Acta Sci.
  Math. (Szeged), 24 (1963), pp.~88--90.

\bibitem{Atkinson2019Article}
{\sc S.~Atkinson}, {\em {Graph products of completely positive maps}}, J.
  Operator Theory, 81 (2019), pp.~133--156.

\bibitem{Barandes2025MiscIndivis}
{\sc J.~A. Barandes}, {\em {Quantum Systems as Indivisible Stochastic
  Processes}}, 2025.
\newblock Preprint available under
  \url{https://doi.org/10.48550/arXiv.2507.21192}.

\bibitem{Barandes2025ArticleStochQuantum}
\leavevmode\vrule height 2pt depth -1.6pt width 23pt, {\em {The
  Stochastic-Quantum Correspondence}}, Philosophy of Physics, 3 (2025).

\bibitem{BarandesKagan2020MinimalModalInterpretation}
{\sc J.~A. Barandes and D.~Kagan}, {\em {Measurement and quantum dynamics in
  the minimal modal interpretation of quantum theory}}, Found. Phys., 50
  (2020), pp.~1189--1218.

\bibitem{Bergman1978DiamondLemma}
{\sc G.~M. Bergman}, {\em {The diamond lemma for ring theory}}, Adv. in Math.,
  29 (1978), pp.~178--218.

\bibitem{BookOtto1993BookRewritingSystems}
{\sc R.~V. Book and F.~Otto}, {\em {String-rewriting systems}}, {Texts and
  Monographs in Computer Science}, Springer-Verlag, New York, 1993.

\bibitem{ChiconeLatushkin1999Book}
{\sc C.~Chicone and Y.~Latushkin}, {\em {Evolution semigroups in dynamical
  systems and differential equations}}, vol.~70 of {Mathematical Surveys and
  Monographs}, American Mathematical Society, Providence, RI, 1999.

\bibitem{Dahya2024interpolation}
{\sc R.~Dahya}, {\em {Interpolation and non-dilatable families of
  $\mathcal{C}_{0}$-semigroups}}, Banach J. Math. Anal., 18 (2024), p.~Paper
  No. 34.

\bibitem{Davies1976BookQuantumOpenSys}
{\sc E.~B. Davies}, {\em {Quantum theory of open systems}}, Academic Press
  [Harcourt Brace Jovanovich, Publishers], London-New York, 1976.

\bibitem{Davies1978dilationsCPmaps}
\leavevmode\vrule height 2pt depth -1.6pt width 23pt, {\em {Dilations of
  completely positive maps}}, J. London Math. Soc. (2), 17 (1978),
  pp.~330--338.

\bibitem{vomEnde2020PhdThesis}
{\sc F.~v. Ende}, {\em {Reachability in Controlled Markovian Quantum Systems:
  An Operator-Theoretic Approach}}, PhD thesis, Munich, Tech. U., 9 2020.

\bibitem{Evans1976ArticlePos}
{\sc D.~E. Evans}, {\em {Positive linear maps on operator algebras}}, Comm.
  Math. Phys., 48 (1976), pp.~15--22.

\bibitem{Evans1976ArticlePerturb}
\leavevmode\vrule height 2pt depth -1.6pt width 23pt, {\em {Time dependent
  perturbations and scattering of strongly continuous groups on Banach
  spaces}}, Math. Ann., 221 (1976), pp.~275--290.

\bibitem{Goldstein1985semigroups}
{\sc J.~A. Goldstein}, {\em {Semigroups of linear operators \& applications}},
  {Oxford Mathematical Monographs}, The Clarendon Press, Oxford University
  Press, New York, 1985.

\bibitem{HaagKastler1964Article}
{\sc R.~Haag and D.~Kastler}, {\em {An Algebraic approach to quantum field
  theory}}, J. Math. Phys., 5 (1964), pp.~848--861.

\bibitem{Hall2015Book}
{\sc B.~Hall}, {\em {Lie groups, Lie algebras, and representations}}, vol.~222
  of {Graduate Texts in Mathematics}, Springer, Cham, second~ed., 2015.
\newblock An elementary introduction.

\bibitem{HartzShalit2024Article}
{\sc M.~Hartz and O.~M. Shalit}, {\em {Tensor algebras of subproduct systems
  and noncommutative function theory}}, Canad. J. Math., 76 (2024),
  pp.~1587--1608.

\bibitem{HoltEickObrien2005Book}
{\sc D.~F. Holt, B.~Eick, and E.~A. O'Brien}, {\em {Handbook of computational
  group theory}}, {Discrete Mathematics and its Applications (Boca Raton)},
  Chapman \& Hall/CRC, Boca Raton, FL, 2005.

\bibitem{Howland1974Article}
{\sc J.~S. Howland}, {\em {Stationary scattering theory for time-dependent
  Hamiltonians}}, Math. Ann., 207 (1974), pp.~315--335.

\bibitem{Jech2003}
{\sc T.~Jech}, {\em {Set theory}}, {Springer Monographs in Mathematics},
  Springer-Verlag, Berlin, 2003.
\newblock The third millennium edition, revised and expanded.

\bibitem{Kato1953Article}
{\sc T.~Kato}, {\em {Integration of the equation of evolution in a Banach
  space}}, J. Math. Soc. Japan, 5 (1953), pp.~208--234.

\bibitem{Kraus1971Article}
{\sc K.~Kraus}, {\em {General state changes in quantum theory}}, Ann. Physics,
  64 (1971), pp.~311--335.

\bibitem{Kraus1983}
\leavevmode\vrule height 2pt depth -1.6pt width 23pt, {\em {States, effects,
  and operations}}, vol.~190 of {Lecture Notes in Physics}, Springer-Verlag,
  Berlin, 1983.
\newblock Fundamental notions of quantum theory, Lecture notes edited by A.
  B\"{o}hm, J. D. Dollard and W. H. Wootters.

\bibitem{Lindblad1976Article}
{\sc G.~Lindblad}, {\em {On the generators of quantum dynamical semigroups}},
  Communications in Mathematical Physics, 48 (1976), pp.~119--130.

\bibitem{MilzKimPollock2019ArticleDivisibility}
{\sc S.~Milz, M.~S. Kim, F.~A. Pollock, and K.~Modi}, {\em {Completely positive
  divisibility does not mean Markovianity}}, Phys. Rev. Lett., 123 (2019),
  pp.~040401, 6.

\bibitem{MilzModi2021QSPNonMarkovian}
{\sc S.~Milz and K.~Modi}, {\em {Quantum Stochastic Processes and Quantum
  non-Markovian Phenomena}}, PRX Quantum, 2 (2021), p.~12417.

\bibitem{Murphy1990}
{\sc G.~J. Murphy}, {\em {C\textsuperscript{$\ast$}-algebras and operator
  theory}}, Academic Press, Inc., Boston, MA, 1990.

\bibitem{Naimark1972normedalg}
{\sc M.~A. Naimark}, {\em {Normed algebras}}, {Wolters-Noordhoff Series of
  Monographs and Textbooks on Pure and Applied Mathematics}, Wolters-Noordhoff
  Publishing, Groningen, 3~ed., 1972.
\newblock Translated from the second Russian edition by Leo F. Boron.

\bibitem{Newman1942DiamondLemma}
{\sc M.~H.~A. Newman}, {\em {On theories with a combinatorial definition of
  ``equivalence''}}, Ann. of Math. (2), 43 (1942), pp.~223--243.

\bibitem{Parrott1970counterExamplesDilation}
{\sc S.~Parrott}, {\em {Unitary dilations for commuting contractions}}, Pacific
  J. Math., 34 (1970), pp.~481--490.

\bibitem{Paulsen2002book}
{\sc V.~I. Paulsen}, {\em {Completely bounded maps and operator algebras}},
  vol.~78 of {Cambridge Studies in Advanced Mathematics}, Cambridge University
  Press, Cambridge, 2002.

\bibitem{Pazy1983Book}
{\sc A.~Pazy}, {\em {Semigroups of linear operators and applications to partial
  differential equations}}, vol.~44 of {Applied Mathematical Sciences},
  Springer-Verlag, New York, 1983.

\bibitem{Pedersen1989analysisBook}
{\sc G.~K. Pedersen}, {\em {Analysis now}}, vol.~118 of {Graduate Texts in
  Mathematics}, Springer-Verlag, New York, 1989.

\bibitem{Pisier2001bookCBmaps}
{\sc G.~Pisier}, {\em {Similarity problems and completely bounded maps}},
  vol.~1618 of {Lecture Notes in Mathematics}, Springer-Verlag, Berlin,
  expanded~ed., 2001.

\bibitem{RivasHuelgaPlenio2014ArticleNonMarkovian}
{\sc A.~Rivas, S.~F. Huelga, and M.~B. Plenio}, {\em {Quantum non-Markovianity:
  characterization, quantification and detection}}, Rep. Progr. Phys., 77
  (2014), pp.~094001, 26.

\bibitem{RussoDye1966Article}
{\sc B.~Russo and H.~A. Dye}, {\em {A note on unitary operators in {$C\sp{\ast}
  $}-algebras}}, Duke Math. J., 33 (1966), pp.~413--416.

\bibitem{RybarFilippovZiman2012Article}
{\sc T.~Ryb\'{a}r, S.~N. Filippov, M.~Ziman, and V.~Bu\v{z}ek}, {\em
  {Simulation of indivisible qubit channels in collision models}}, Journal of
  Physics B: Atomic, Molecular and Optical Physics, 45 (2012), p.~154006.

\bibitem{ShalitSkeide2022multiparam}
{\sc O.~M. Shalit and M.~Skeide}, {\em {CP-semigroups and dilations, subproduct
  systems and superproduct systems: the multi-parameter case and beyond}},
  Dissertationes Math., 585 (2023), p.~233.

\bibitem{ShalitSolel2009Article}
{\sc O.~M. Shalit and B.~Solel}, {\em {Subproduct systems}}, Doc. Math., 14
  (2009), pp.~801--868.

\bibitem{Slocinski1974}
{\sc M.~S\l{}oci\'{n}ski}, {\em {Unitary dilation of two-parameter semi-groups
  of contractions}}, Bull. Acad. Polon. Sci. S\'{e}r. Sci. Math. Astronom.
  Phys., 22 (1974), pp.~1011--1014.

\bibitem{Slocinski1982}
\leavevmode\vrule height 2pt depth -1.6pt width 23pt, {\em {Unitary dilation of
  two-parameter semi-groups of contractions II}}, Zeszyty Naukowe Uniwersytetu
  Jagiellońskiego, 23 (1982), pp.~191--194.

\bibitem{Stoermer2013BookPosOps}
{\sc E.~St{\o}rmer}, {\em {Positive linear maps of operator algebras}},
  {Springer Monographs in Mathematics}, Springer, Heidelberg, 2013.

\bibitem{Stroescu1973ArticleBanachDilations}
{\sc E.~Stroescu}, {\em {Isometric dilations of contractions on Banach
  spaces}}, Pacific J. Math., 47 (1973), pp.~257--262.

\bibitem{Nagy1970}
{\sc B.~Sz\H{o}kefalvi-Nagy and C.~Foia\c{s}}, {\em {Harmonic analysis of
  operators on Hilbert space}}, North-Holland Publishing Co., Amsterdam-London;
  American Elsevier Publishing Co., Inc., New York; Akad\'{e}miai Kiad\'{o},
  Budapest, 1970.
\newblock Translated from the French and revised.

\bibitem{Takesaki2002BookI}
{\sc M.~Takesaki}, {\em {Theory of operator algebras. I}}, vol.~124 of
  {Encyclopaedia of Mathematical Sciences}, Springer-Verlag, Berlin, 2002.
\newblock Reprint of the first (1979) edition, Operator Algebras and
  Non-commutative Geometry, 5.

\bibitem{Takesaki2003BookIII}
\leavevmode\vrule height 2pt depth -1.6pt width 23pt, {\em {Theory of operator
  algebras. III}}, vol.~127 of {Encyclopaedia of Mathematical Sciences},
  Springer-Verlag, Berlin, 2003.
\newblock Operator Algebras and Non-commutative Geometry, 8.

\bibitem{Varopoulos1974counterexamples}
{\sc N.~T. Varopoulos}, {\em {On an inequality of von Neumann and an
  application of the metric theory of tensor products to operators theory}}, J.
  Functional Analysis, 16 (1974), pp.~83--100.

\bibitem{Vernik2016Article}
{\sc A.~Vernik}, {\em {Dilations of CP-maps commuting according to a graph}},
  Houston J. Math., 42 (2016), pp.~1291--1329.

\bibitem{vomEnde2019unitaryDildiscreteCPsemigroups}
{\sc F.~vom Ende and G.~Dirr}, {\em {Unitary dilations of discrete-time
  quantum-dynamical semigroups}}, J. Math. Phys., 60 (2019), pp.~122702, 17.

\bibitem{Wilcox1967ArticleExp}
{\sc R.~M. Wilcox}, {\em {Exponential operators and parameter differentiation
  in quantum physics}}, J. Mathematical Phys., 8 (1967), pp.~962--982.

\bibitem{WolfCirac2008Article}
{\sc M.~M. Wolf and J.~I. Cirac}, {\em {Dividing quantum channels}}, Comm.
  Math. Phys., 279 (2008), pp.~147--168.

\end{thebibliography}
\def\bibname{References}
\bgroup

\egroup

%% ********** END OF FILE: back/.index.tex **********

\addresseshere
\end{document}

%% ********** END OF FILE: root.tex **********